\documentclass[12pt,a4paper,oneside]{article}

\usepackage[T1]{fontenc}
\usepackage[utf8]{inputenc}
\usepackage{lmodern}

\usepackage{amsmath,amssymb,theorem}

\usepackage[left=2.9cm,right=2.9cm,top=2cm,bottom=2cm,includefoot]{geometry}

\usepackage[yyyymmdd,hhmmss]{datetime}

\usepackage{enumerate}

\usepackage[nottoc]{tocbibind}

\usepackage{hyperref}

\usepackage{cite}

\usepackage{epsfig}
\usepackage{epic}
\usepackage{eepic}
\usepackage{graphics}
\usepackage{graphicx}
\usepackage{color}
\usepackage{longtable}
\usepackage{subfigure}

\usepackage{caption}
\newtheorem{theorem}{Theorem}
\newtheorem{corollary}[theorem]{Corollary}

\newtheorem{lemma}[theorem]{Lemma}
{\theorembodyfont{\rmfamily}}
\newtheorem{proposition}[theorem]{Proposition}
{\theorembodyfont{\rmfamily}}
\newenvironment{proof}{\noindent\textit{Proof.\ }}{\hspace*{\fill}$\Box$\medskip}
\newenvironment{proofof}[1]{\noindent\textit{Proof of #1.\ }}{\hspace*{\fill}$\Box$\medskip}

\numberwithin{equation}{section}
\numberwithin{theorem}{section}
\numberwithin{figure}{section}

\renewcommand{\geq}{\geqslant}
\renewcommand{\leq}{\leqslant}
\newcommand{\tsum}{\textstyle\sum\limits}
\newcommand{\tprod}{\textstyle\prod\limits}

\DeclareMathOperator*{\esssup}{ess\,sup}

\newcommand{\vect}[1]{\underline{\boldsymbol{#1}}}

\usepackage{soul}
\newcommand{\vectgamma}{\text{\setul{1.3pt}{.4pt}\ul{$\boldsymbol{\gamma}$}}}
\newcommand{\indexvectgamma}{\text{\setul{0.9pt}{.4pt}\ul{$\boldsymbol{\gamma}$}}}

\DeclareMathOperator{\lcm}{lcm}
\DeclareMathOperator{\vspan}{span}
\DeclareMathOperator{\tmodop}{mod}
\newcommand{\tmod}{\,\tmodop}

\DeclareMathOperator{\LCop}{LC}
\newcommand{\LC}{\vect{\LCop}}
\DeclareMathOperator{\Iop}{I}
\newcommand{\I}{\vect{\Iop}}
\DeclareMathOperator{\id}{id}
\DeclareMathOperator{\FL}{L\!}

\newcommand{\llceil}{\lceil\;\;\;\!\!\!\!\!\!\!\lceil}
\newcommand{\rrceil}{\rceil\;\;\;\!\!\!\!\!\!\!\rceil}
\newcommand{\llfloor}{\lfloor\;\;\;\!\!\!\!\!\!\!\lfloor}
\newcommand{\rrfloor}{\rfloor\;\;\;\!\!\!\!\!\!\!\rfloor}

\allowdisplaybreaks

\begin{document}

\definecolor{light-gray}{gray}{0.2}
\color{light-gray}

\setlength\abovedisplayskip      { 7pt }
\setlength\abovedisplayshortskip { 3pt }
\setlength\belowdisplayskip      { 7pt }
\setlength\belowdisplayshortskip { 3pt }

\title{Lebesgue constants for polyhedral sets and
\text{polynomial interpolation on Lissajous-Chebyshev nodes}}

\author{ \begin{tabular}{ll} Peter Dencker$^{\mathrm{1)}}$\setcounter{footnote}{-1}\footnote{$^{\mathrm{1)}}$Institut f\"ur Mathematik, Universit\"at zu L\"ubeck.},\quad & Wolfgang Erb$^{\mathrm{1)},\mathrm{a)}}$,\\
 Yurii Kolomoitsev$^{\mathrm{1)},\mathrm{2)},\mathrm{b)},\mathrm{*}}$\setcounter{footnote}{-1}\footnote{$^{\mathrm{2)}}$Institute of Applied Mathematics and Mechanics, Academy of Sciences of Ukraine.}\setcounter{footnote}{-1}\footnote{$^{\mathrm{a)}}$Supported by the  German Research Foundation (DFG, grant number ER 777/1-1).}\setcounter{footnote}{-1}\footnote{$^{\mathrm{b)}}$Supported by H2020-MSCA-RISE-2014 Project number 645672.}\setcounter{footnote}{-1}\footnote{$\ ^{\mathrm{*}}$Corresponding author.}, \quad &  Tetiana Lomako$^{\mathrm{1)},\mathrm{2)},\mathrm{b)}}$
 \\[5mm]
\texttt{\normalsize dencker@math.uni-luebeck.de,} &  \texttt{\normalsize erb@math.uni-luebeck.de,}\\[-0.2em]
\texttt{\normalsize kolomus1@mail.ru,} &
\texttt{\normalsize tlomako@yandex.ru}
\end{tabular}}

\date{\ \\[2mm] \normalsize 23. August 2017}

\maketitle
\tableofcontents

\bigskip

\begin{abstract}
To analyze the absolute condition number of multivariate polynomial interpolation on Lissajous-Chebyshev node points, we derive upper and lower bounds for
the respective Lebesgue constant. The proof is based on a relation between the Lebesgue constant for the polynomial interpolation problem and the
Lebesgue constant linked to the polyhedral partial sums of Fourier series. The magnitude of the obtained bounds is determined by a product of logarithms of the side lengths of
the considered polyhedral sets and shows the same behavior as the magnitude of the Lebesgue constant for polynomial interpolation on the tensor product Chebyshev grid.

\vspace{2em}

\noindent\hspace{-2em}\text{\textbf{Keywords:} interpolation, Lissajous-Chebyshev nodes, Lebesgue constants, polyhedra}
\end{abstract}

\newpage

\section{Introduction}

In \cite{DenckerErb2015a,ErbKaethnerAhlborgBuzug2015,ErbKaethnerDenckerAhlborg2015}, a multivariate polynomial interpolation scheme was developed to interpolate function values on
equidistant node points along Lissajous trajectories. The consideration of such node points is motivated by applications
in a novel medical imaging modality called Magnetic Particle Imaging (MPI) (see \cite{BuzugKnopp2012,KaethnerErbAhlborg2016}). In this imaging technology,
the magnetic response of superparamagnetic nanoparticles is measured along particular sampling paths generated by applied magnetic fields. For a typical kind of MPI scanner,
these sampling paths are Lissajous curves. 

In two dimensions, the polynomial interpolation scheme given in \cite{ErbKaethnerAhlborgBuzug2015} was used to recover
the distribution of the magnetic particles from a reduced reconstruction on equidistant nodes along the Lissajous trajectory \cite{KaethnerErbAhlborg2016}. 
A particular feature of this bivariate interpolation scheme is the fact that the self-intersection and the boundary points of Lissajous curves are used as interpolation nodes and that 
the spectral index set of the underlying polynomial space has a triangular structure. 
In \cite{DenckerErb2015a}, this bivariate construction was extended to higher dimensional Lissajous curves by using polynomial spaces with a particular polygonal spectral structure that
will be studied in more detail in this work.
In the literature, there exist also other polynomial approximation schemes that use Lissajous trajectories as generating curves. 
Two such constructions for three and more dimensions for polynomial spaces of a bounded total or maximal degree can be found in \cite{BosDeMarchiVianello2017,BosDeMarchiVianello2017b}.
Note that in the choice of the Lissajous curves and the polynomial spaces these constructions differ from the approach considered in this work.

Using polynomials for interpolation, special attention has to be given to the numerical condition of the interpolation scheme. 
In order to exclude bad conditioning, the structure of the interpolation nodes as well as the spectral structure of
the polynomial interpolants have to be studied. The goal of this article is to provide such an analysis for the absolute condition number of the
polynomial interpolation schemes considered in \cite{DenckerErb2015a,ErbKaethnerAhlborgBuzug2015,ErbKaethnerDenckerAhlborg2015}. The interpolation nodes under consideration have been 
introduced in \cite{DenckerErb2015a} as
Lissajous-Chebyshev node points $\LC^{(\epsilon\vect{n})}_{\vect{\kappa}}$ (see \eqref{eq:201605170810}). In this notation, the parameters
$\vect{\kappa}\in  \mathbb{Z}^{\mathsf{d}}$ and $\epsilon\in\{1,2\}$ determine the underlying types of Lissajous curves, and the vector
\begin{equation}\label{1608042338}
\text{$\vect{n}=(n_1,\ldots,n_{\mathsf{d}})\in \mathbb{N}^{\mathsf{d}}$ with \emph{pairwise relatively prime} entries $n_1,\ldots,n_{\mathsf{d}}\in\mathbb{N}$}
\end{equation}
describes the frequencies of the Lissajous curve with respect to the coordinate axis.
The interpolation problem itself is given as follows:

\vspace{2mm}

\emph{For the node points $\LC^{(\epsilon\vect{n})}_{\vect{\kappa}}$  and a function $f:[-1,1]^{\mathsf{d}} \to \mathbb{R}$ with values
$f(\vect{z})$ at the node points $\vect{z} = (z_1, \ldots, z_{\mathsf{d}}) \in $ $\LC^{(\epsilon\vect{n})}_{\vect{\kappa}}$, find a $\mathsf{d}$-variate interpolation polynomial $P^{(\epsilon\vect{n})}_{\!\vect{\kappa}}f$ such that
\begin{equation} \label{1502182305} P^{(\epsilon\vect{n})}_{\!\vect{\kappa}}f (\vect{z}) = f(\vect{z}) \quad \text{for all}\quad \vect{z}\in \LC^{(\epsilon\vect{n})}_{\vect{\kappa}}.\end{equation}}

\vspace{-4mm}

It was shown in \cite{DenckerErb2015a} that the interpolation problem \eqref{1502182305} has a unique solution in the polynomial space $\Pi^{(\epsilon\vect{n})}_{\vect{\kappa}}$
that is linearly spanned by all $\mathsf{d}$-variate Chebyshev polynomials $T_{\indexvectgamma}$, where  $\vectgamma$ is an element of the spectral index set
\[\vect{\Gamma}^{(\epsilon\vect{n})}_{\vect{\kappa}} = \left\{\,\vectgamma\in\mathbb{N}_0^{\mathsf{d}}\ \left|\begin {array}{rl}
\gamma_{\mathsf{i}}/n_{\mathsf{i}}< \epsilon &  \text{$\forall\,\mathsf{i}\in\{1,\ldots,\mathsf{d}\}$,} \\
\gamma_{\mathsf{i}}/n_{\mathsf{i}}+\gamma_{\mathsf{j}}/n_{\mathsf{j}}
\leq \epsilon &
\text{$\forall\,\mathsf{i},\mathsf{j}$ with $\mathsf{i} \neq \mathsf{j}$}, \\
\gamma_{\mathsf{i}}/n_{\mathsf{i}}+\gamma_{\mathsf{j}}/n_{\mathsf{j}}
< \epsilon & \text{$\forall\,\mathsf{i},\mathsf{j}$ with  $\kappa_{\mathsf{i}} \not \equiv \kappa_{\mathsf{j}} \tmod 2 $}
\end{array}\right.
\right\}\cup\{(0,\ldots,0,\epsilon n_{\mathsf{d}})\}.\]
The nodes $\LC^{(\epsilon\vect{n})}_{\vect{\kappa}}$, the Chebyshev polynomials $T_{\indexvectgamma}$, and the interpolation problem will be recapitulated in more detail
in Section \ref{1502171626} of this article.

\medskip

The absolute condition number of the interpolation problem \eqref{1502182305} with respect to the uniform norm
$\|f\|_{\infty} = \displaystyle\esssup_{\vect{x} \in [-1,1]^{\mathsf{d}}}|f(\vect{x})|$ (see \cite[p. 26]{DeuflhardHohmann2003}) is given by the  \emph{Lebesgue constant of the interpolation problem}, i.e.
\begin{equation} \label{eq:201605281730} \Lambda^{(\epsilon\vect{n})}_{\vect{\kappa}} = \sup\limits_{f\in C([-1,1]^{\mathsf{d}}):\,\|f\|_{\infty}\leq 1}\|P^{(\epsilon\vect{n})}_{\!\vect{\kappa}}f\|_{\infty}.
\end{equation}
Besides its relation to the numerical stability of the interpolation problem \eqref{1502182305}, the Lebesgue constant \eqref{eq:201605281730} is also an essential tool for the investigation of the approximation error $\|f - P^{(\epsilon\vect{n})}_{\!\vect{\kappa}}f\|_\infty$.

 A main goal of this article is to provide for all $\vect{n}$ satisfying \eqref{1608042338}  asymptotic upper and  lower bounds for the  Lebesgue constants \eqref{eq:201605281730} in the sense of \eqref{1606181233}. The  corresponding result in
Theorem \ref{1503120230} states

\vspace{-1.5em}

\begin{equation} \label{1606180830}
\Lambda^{(\epsilon\vect{n})}_{\vect{\kappa}}  \asymp \prod_{\mathsf{i}=1}^{\mathsf{d}}\ln (n_{\mathsf{i}}+1).
\end{equation}
In particular, the upper and lower estimates  have asymptotically the same magnitude as the Lebesgue constants for polynomial
interpolation on the tensor product Chebyshev grid (see \cite{Brutman1997}). Therefore, the interpolation problem \eqref{1502182305} in  $\Pi^{(\epsilon\vect{n})}_{\vect{\kappa}}$ is asymptotically  as well-conditioned as the mentioned tensor product case.
The upper estimate in \eqref{1606180830} of the Lebesgue constant $\Lambda^{(\epsilon\vect{n})}_{\vect{\kappa}}$ is further  used in Corollary \ref{cor-dinilipschitz} to formulate a multivariate error estimate and an example  of a Dini-Lipschitz-type  condition for the uniform convergence of the  interpolation polynomials $P^{(\epsilon\vect{n})}_{\!\vect{\kappa}}f$.

In the bivariate setting, the obtained results are generalizations of the corresponding results for the Padua points in \cite{BosDeMarchiVianelloXu2006,CaliariDeMarchiVianello2008,
VecchiaMastroianniVertesi2009} and improvements of estimates given in \cite{Erb2015}.

\medskip

We sketch  our program for the proof of \eqref{1606180830}.
For a finite set $\vect{\Gamma}\subset\mathbb{Z}^{\mathsf{d}}$, the \emph{Lebesgue constant} $\FL\left(\vect{\Gamma}\right)$ related to partial Fourier series is defined as
\[\FL\left(\vect{\Gamma}\right)=\dfrac1{(2\pi)^{\mathsf{d}}}\int_{[-\pi,\pi)^{\mathsf{d}}}\left|\tsum_{\indexvectgamma\in \vect{\Gamma}}
\mathrm{e}^{\mathbf{i}(\indexvectgamma,\vect{t})}\right|\,\mathsf{d}\vect{t},\]
where \[(\vectgamma,\vect{t}) = \displaystyle \sum_{\mathsf{i} = 1}^{\mathsf{d}} \gamma_{\mathsf{i}} t_{\mathsf{i}}.\]
To obtain the upper and lower bounds for \eqref{eq:201605281730}, our strategy in the proof of Theorem~\ref{1503120230} consists in establishing the relations
\begin{equation} \label{1606180824}
\Lambda^{(\epsilon\vect{n})}_{\vect{\kappa}} \lesssim \FL\left(\vect{\Gamma}^{(\epsilon\vect{n}),\ast}_{\vect{\kappa}}\right)+\prod_{\mathsf{i}=1}^{\mathsf{d}}\ln (n_{\mathsf{i}}+1),\qquad \FL\left(\vect{\Gamma}^{(\epsilon\vect{n}),\ast}_{\vect{\kappa}}\right)\lesssim \Lambda^{(\epsilon\vect{n})}_{\vect{\kappa}}
\end{equation}  between $\Lambda^{(\epsilon\vect{n})}_{\vect{\kappa}}$ and
the Lebesgue constants $\FL\left(\vect{\Gamma}^{(\epsilon\vect{n}),\ast}_{\vect{\kappa}}\right)$ of the symmetrized sets $\vect{\Gamma}^{(\epsilon\vect{n}),\ast}_{\vect{\kappa}}$.
  Here and in the following, for every $\vect{\Gamma} \subset {\mathbb{Z}^{\mathsf{d}}}$ its symmetrization  $\vect{\Gamma}^{\, \ast}$ is defined as
\begin{equation}\label{201606171550}
\vect{\Gamma}^{\, \ast}=\left\{\,\vectgamma\in \mathbb{Z}^{\mathsf{d}}\,\left|\ (|\gamma_1|,\ldots,|\gamma_{\mathsf{d}}|)\in \vect{\Gamma}\right.\right\}.
\end{equation}
Using the  methods developed in  Section  \ref{1606171653},  Corollary \ref{1502171206} states that
\begin{equation} \label{201609051530}
 \FL\left(\vect{\Gamma}^{(\epsilon\vect{n}),\ast}_{\vect{\kappa}}\right) \asymp \prod_{\mathsf{i}=1}^{\mathsf{d}}\ln (n_{\mathsf{i}}+1).
\end{equation}
Then, combining \eqref{1606180824} and  \eqref{201609051530}  yields \eqref{1606180830}.

\begin{figure}[htb]
	\centering
\includegraphics[scale=0.25]{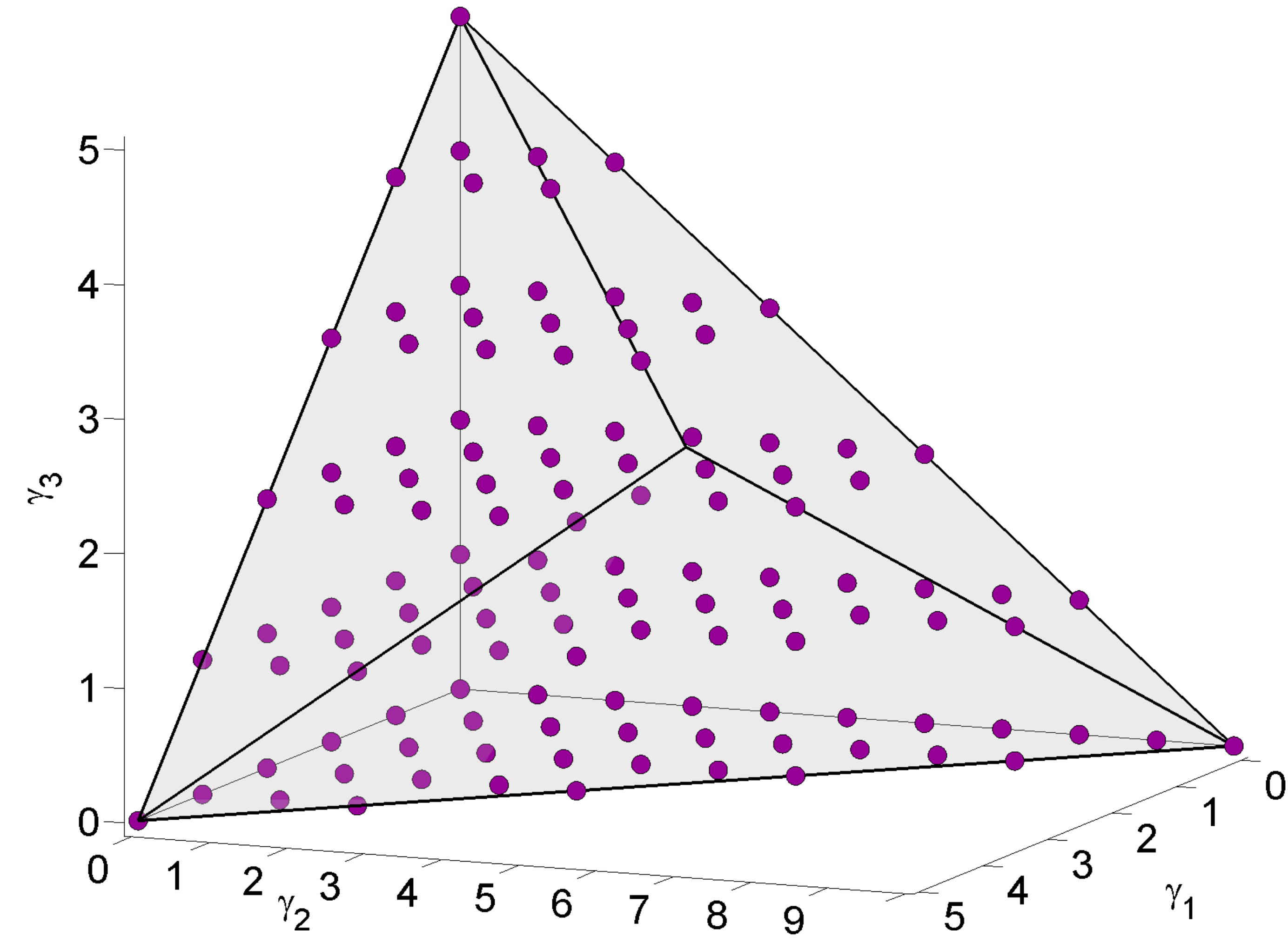}
	\hspace{0.5cm}	
\includegraphics[scale=0.25]{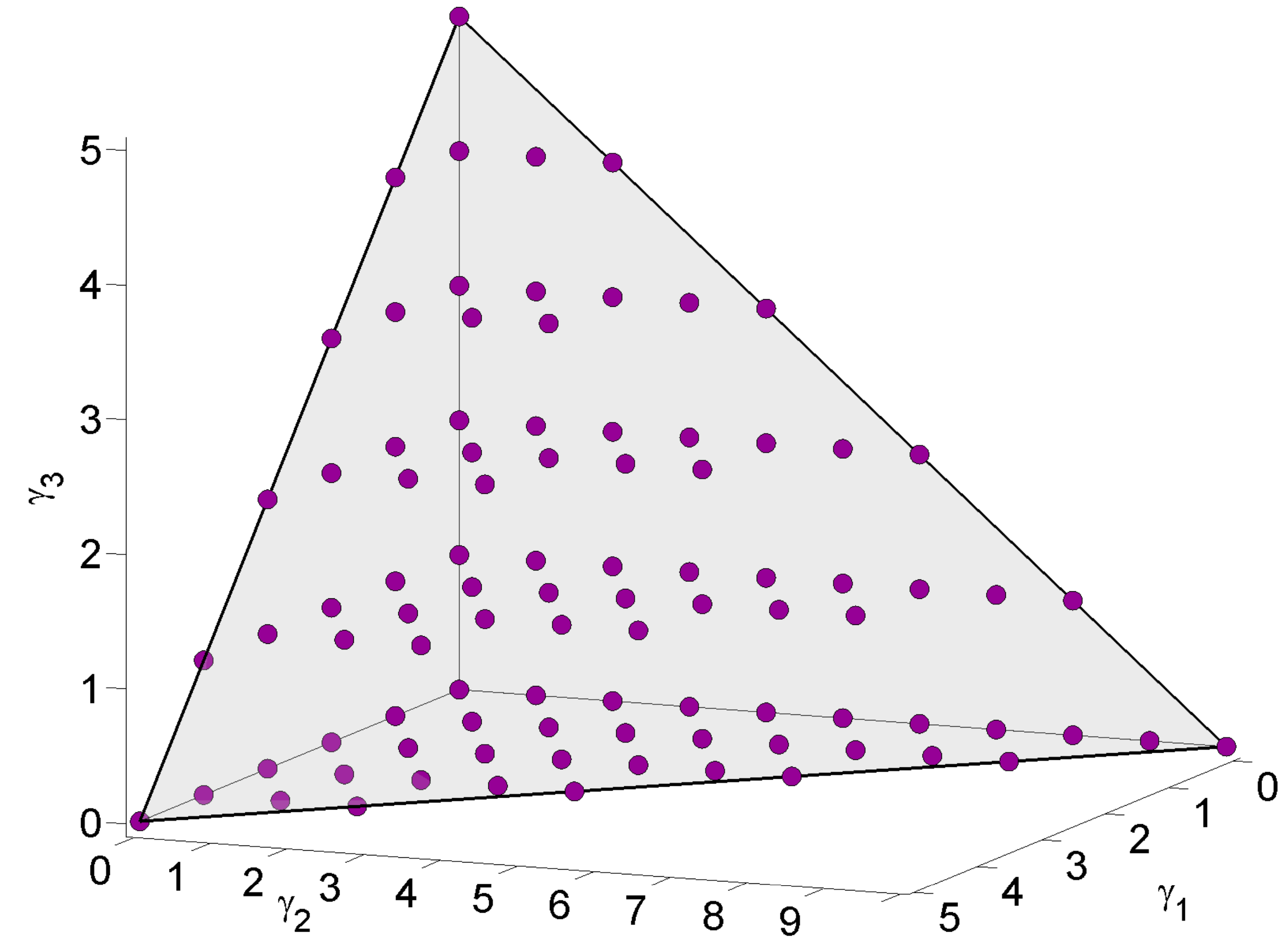}
  	\caption{Illustration of the sets $\overline{\vect{\Gamma}}^{(\vect{m})}$ (left) and $\vect{\Sigma}^{(\vect{m})}_1$ (right) for $\vect{m} = (5,10,5)$.
  	}
	\label{fig:1}
\end{figure}

\medskip

The technically more sophisticated part of the sketched program is the proof of \eqref{201609051530}.  The used methods are developed in Section \ref{1606171653}. Therein we consider the  sets
\[\overline{\vect{\Gamma}}^{(\vect{m})} =\left\{\,\vectgamma\in\mathbb{N}_0^{\mathsf{d}}\ \left|\begin {array}{rl}
\gamma_{\mathsf{i}}/m_{\mathsf{i}}\leq  1& \forall\,\mathsf{i}\in\{1,\ldots,\mathsf{d}\},\\
\gamma_{\mathsf{i}}/m_{\mathsf{i}}+\gamma_{\mathsf{j}}/m_{\mathsf{j}}\leq  1 & \forall \ \mathsf{i},\mathsf{j}\ \text{with}\ \mathsf{i}\neq\mathsf{j}
\end{array}\right.\right\}\]
and its symmetrizations $\overline{\vect{\Gamma}}^{(\vect{m}),\ast}$ according to \eqref{201606171550}.
  The used methods for these sets are templates for the corresponding methods for the sets $\vect{\Gamma}^{(\epsilon \vect{n})}_{\vect{\kappa}}$ and $\vect{\Gamma}^{(\epsilon \vect{n}),\ast}_{\kappa}$, respectively. It turns out that similar methods can be used to estimate for fixed rational $r>0$ the Lebesgue constants of families of sets $\vect{\Sigma}^{(\vect{m})}_{\,r}$ and its symmetrizations  $\vect{\Sigma}^{(\vect{m}),\ast}_{\,r}$, where
\[\vect{\Sigma}^{(\vect{m})}_{\,r} =\left\{\,\vectgamma\in \mathbb{N}_0^{\mathsf{d}}\,\left|\ \tsum_{\mathsf{i}=1}^{\mathsf{d}}\dfrac{\gamma_{\mathsf{i}}}{m_{\mathsf{i}}}\leq r\right.\right\}.\]
Note that $\overline{\vect{\Gamma}}^{(m_1,m_2),\ast}= \vect{\Sigma}^{(m_1,m_2),\ast}_{\,1}$ in dimension $\mathsf{d} = 2$.
Sets of this kind are illustrated in Figure \ref{fig:1} and are of interest since they might  be used as elementary building blocks for more complex polyhedra. Further, our results could be useful for the investigation of generalizations of the triangular partial Fourier series in \cite{Weisz2012}.

\medskip

Estimates of the Lebesgue constant $\FL\left(\vect{\Gamma}\right)$ for various types of sets $\vect{\Gamma}$ are extensively investigated in the literature. An overview about the state of the art  can be found in the survey article \cite{Liflyand2006}.
Since we are dealing with sets  having  a polyhedral structure, estimates of the Lebesgue constants for those sets  are particularly interesting for us. If $\vect{E}$ is a fixed $\mathsf{d}$-dimensional convex polyhedron  containing the origin, then it is well-known (see \cite{Baiborodov1982,Belinskii1977,Podkorytov1981,Travaglini1993,Yudin1979,YudinYudin1985}) that  for all real $m\geq 1$ we have
\[\FL\left(m \vect{E} \cap \mathbb{Z}^{\mathsf{d}}\right) \asymp \left(\ln(m+1)\right)^{\mathsf{d}}.\]

In this work, we want to refine this asymptotic result for special $\mathsf{d}$-dimensional polyhedra in which  integer-valued directional dilation parameters  $m_1,\ldots, m_{\mathsf{d}} \in \mathbb{N}$ are given.
An example for different directional parameters is the case of   rectangular sets $\vect{R}^{(\vect{m})}=[0,m_1] \times \cdots \times [0,m_{\mathsf{d}}]$.
In this case, for all $m_1,\ldots,m_{\mathsf{d}}\geq 1$, we have
\begin{equation}\label{1606171646}
\FL\left(\vect{R}^{(\vect{m})}\cap \mathbb{Z}^{\mathsf{d}}\right) \asymp \FL\left(\vect{R}^{(\vect{m}),\ast}\cap \mathbb{Z}^{\mathsf{d}}\right)
\asymp \prod_{\mathsf{i}=1}^{\mathsf{d}}\ln (m_{\mathsf{i}}+1).
\end{equation}
This immediately follows from the well-known one-dimensional case (see
\cite{Alimov1992}).

\medskip

The starting points for our investigations of $\FL\left(\vect{\Gamma}\right)$ are two estimates
of the Lebesgue constant given in \cite{YudinYudin1985} and \cite{Yudin1979}. In \cite[Theorem 2]{YudinYudin1985} it is stated that
for all polyhedra $\vect{E} \in \mathbb{R}^2$ with $n$ edges, we have the uniform upper bound
\begin{equation}\label{1502221953}\displaystyle \FL\left(\vect{E} \cap \mathbb{Z}^{2}\right) \lesssim n \left(\ln\mathrm{diam}(\vect{E})\right)^2.
\end{equation}
Further, it is shown in \cite{Yudin1979} that for all convex sets $\vect{E} \in \mathbb{R}^{\mathsf{d}}$ containing a ball with radius $r\geq 1$ we have the lower bound
\begin{equation}\label{1502221954}\displaystyle \FL\left(\vect{E} \cap \mathbb{Z}^{\mathsf{d}}\right) \gtrsim \left(\ln(r+1)\right)^{\mathsf{d}}.
 \end{equation}
Combining \eqref{1502221953} and \eqref{1502221954} yields that for all real $m_1,m_2\geq 1$  we have  the uniform upper and lower bound
\begin{equation}\label{1502221955} \left(\ln \left(\min(m_1,m_2)+1 \right)\right)^2 \lesssim \FL\left(\overline{\vect{\Gamma}}^{(m_1,m_2),\ast}\right) \lesssim \left(\ln \left(\max(m_1,m_2)+1 \right)\right)^2.
\end{equation}
A special case of our result (see Theorem \ref{1602041200}) is that for all positive integers $m_1,m_2$ we can improve \eqref{1502221955} to
$\FL\left(\overline{\vect{\Gamma}}^{(m_1,m_2),\ast}\right) \asymp \ln (m_1+1) \ln (m_2+1)$.
Under the  strongly restrictive condition that $m_2$ is a multiple of $m_1$ this result appears already in \cite{Kuznetsova1977}.

\medskip

In general, Theorem \ref{1602041200} states  that for  all $\vect{m}\in\mathbb{N}^{\mathsf{d}}$ we have
\[\FL\left(\overline{\vect{\Gamma}}^{(\vect{m})}\right)\asymp
\FL\left(\overline{\vect{\Gamma}}^{(\vect{m}),\ast}\right)\asymp\tprod_{\mathsf{i}=1}^{\mathsf{d}}\ln (m_{\mathsf{i}}+1).\]
Thus,
the magnitude of the uniform upper and lower bounds is the same as in the rectangular case \eqref{1606171646}.
Similarly, Theorem  \ref{1511231653}  states  that for a fixed $r\in\mathbb{Q}$, $r>0$, and all  $\vect{m}\in\mathbb{N}^{\mathsf{d}}$ we have
\[\FL\left(\vect{\Sigma}^{(\vect{m})}_r\right)\asymp
\FL\left(\vect{\Sigma}^{(\vect{m}),\ast}_r\right)\asymp\tprod_{\mathsf{i}=1}^{\mathsf{d}}\ln (m_{\mathsf{i}}+1).\]

\medskip

In Section  \ref{1606171653}, we consider also another type of  polyhedral sets given by
\begin{equation} \label{1501032000c}
\vect{\Xi}_{(r,s)}^{(\vect{m})}=\left\{\,\vectgamma\in \mathbb{Z}^{\mathsf{d}}\,\left|\ r\leq \dfrac{\gamma_{\mathsf{d}}}{m_{\mathsf{d}}}\leq \cdots
\leq  \dfrac{\gamma_2}{m_2}\leq \dfrac{\gamma_1}{m_1}\leq s\right.\right\}.
\end{equation}
For fixed   $r,s\in\mathbb{R}$ and all positive integers $m_1,\ldots,m_{\mathsf{d}}\in \mathbb{N}$, a uniform upper bound
$\FL\left(\vect{\Xi}_{(r,s)}^{(\vect{m})}\right)\lesssim\displaystyle \prod_{\mathsf{i}=1}^{\mathsf{d}}\ln (m_{\mathsf{i}}+1)$
is established for the  corresponding  Lebesgue constant in Theorem \ref{A1511231653}.
The proof of the upper bound of the Lebesgue constants for the  polyhedral sets $\overline{\vect{\Gamma}}^{(\vect{m})}$ and $\overline{\vect{\Gamma}}^{(\vect{m}),\ast}$ uses slightly generalized versions (see \eqref{1601032151})
of the polyhedral sets \eqref{1501032000c} as building blocks. The techniques presented in the proofs of Section  \ref{1606171653} are interesting in their own regard and might be as well useful for the consideration of other types of polyhedral sets.

\section*{General notation}
For $x\in\mathbb{R}$, we use  $\lfloor x\rfloor =\max\{n\in\mathbb{Z}\,|\,n\leq x\}$, $\lceil x\rceil =\min\{n\in\mathbb{Z}\,|\,n\geq x\}$ and denote
\[\llfloor x\rrfloor =x-\lfloor x\rfloor,\quad \llceil x\rrceil=\lceil x\rceil-x.\]

\medskip

Let $f$ and $g$ be real functions on a set $X$. The notation \[ \text{$f(x)\lesssim g(x)$ for  all $x\in X$}\]
has by definition the following meaning: \[\text{There exists a constant $C > 0$ such that $f(x)\leq C g(x)$
for all $x\in X$.}\] Furthermore, we write
\begin{equation}\label{1606181233}
\text{$f(x)\asymp g(x)$ for all $x\in X$},
\end{equation}

\vspace{2mm}

\noindent if for all $x\in X$ we have both $f(x)\lesssim g(x)$ and $g(x)\lesssim f(x)$.

\medskip

 We write $\vect{x}=(x_1,\ldots,x_{\mathsf{d}})$ for elements of the Euclidean space $\mathbb{R}^{\mathsf{d}}$ with fixed $\mathsf{d}\in\mathbb{N}$.\\
\noindent For $a,b\in\mathbb{R}$, $a<b$, $1\leq p<\infty$ and  Lebesgue-measurable $f:\,[a,b)^{\mathsf{d}}\to \mathbb{R}$, we set
\[\|f\|_{L^p([a,b)^{\mathsf{d}})} = \left( \dfrac1{(b-a)^{\mathsf{d}}}\int_{[a,b)^{\mathsf{d}}}\left|f(\vect{t})\right|^p\,\mathsf{d}\vect{t}
\right)^{1/p},\]
and for Lebesgue-measurable functions  $f: [-1,1]^{\mathsf{d}} \to \mathbb{R}$, and  $1 \leq p < \infty$, we define
\[  \|f\|_{w_{\mathsf{d}},\;\;\!\!\!p} = \left( \dfrac1{\pi^{\mathsf{d}}}\int_{[-1,1]^{\mathsf{d}}}\left|f(\vect{x})\right|^p\,w_{\mathsf{d}}(\vect{x}) \mathsf{d}\vect{x}
\right)^{1/p}, \quad w_{\mathsf{d}}(\vect{x}) = \tprod_{\mathsf{i} = 1}^{\mathsf{d}}  \displaystyle\dfrac{1}{\sqrt{1-x_{\mathsf{i}}^2}}.\]

\section{Lebesgue constants for polyhedral partial sums of Fourier series}\label{1606171653}

We summarize the main results of this section.

\begin{theorem}\label{1602041200}
For all $\vect{m}\in\mathbb{N}^{\mathsf{d}}$, we have
\[\FL\left(\overline{\vect{\Gamma}}^{(\vect{m})}\right)\asymp  \FL\left(\overline{\vect{\Gamma}}^{(\vect{m}),\ast}\right)\asymp \tprod_{\mathsf{i}=1}^{\mathsf{d}}\ln (m_{\mathsf{i}}+1).\]
\end{theorem}

In Section \ref{1502171626}, we will apply this theorem to obtain estimates of the Lebesgue constant for the interpolation problem on the Lissajous-Chebyshev nodes.
To  prove \mbox{Theorem~\ref{1602041200}} we will use the following statement which is also interesting by itself.

\begin{theorem} \label{A1511231653}  Let $r,s\in\mathbb{R}$, $0\leq r<s$, be fixed. For all $\vect{m}\in\mathbb{N}^{\mathsf{d}}$, we have
\begin{equation}\label{eqth1}
  \FL\left(\vect{\Xi}_{(r,s)}^{(\vect{m})}\right)\lesssim \tprod_{\mathsf{i}=1}^{\mathsf{d}}\ln (m_{\mathsf{i}}+1).
\end{equation}
\end{theorem}

\medskip

Further, let us consider the sets $\vect{\Sigma}^{(\vect{m})}_{\,r}$ and $\vect{\Sigma}^{(\vect{m})\ast}_{\,r}$. These sets  can be considered as  another possible generalization of the sets considered in  \cite{YudinYudin1985}  for $\vect{m}\in\mathbb{Z}^{\mathsf{d}}$, and they are interesting since they may be used as  building blocks for certain polyhedra.

\begin{theorem}\label{1511231653} Let $r\in\mathbb{Q}$, $r>0$, be fixed. For all $\vect{m}\in\mathbb{N}^{\mathsf{d}}$, we have
\begin{equation}\label{estimatesS}
  \FL\left( \vect{\Sigma}^{(\vect{m})}_{r}\right)\asymp  \FL\left( \vect{\Sigma}^{(\vect{m}),\ast}_{r}\right)\asymp \tprod_{\mathsf{i}=1}^{\mathsf{d}}\ln (m_{\mathsf{i}}+1).
\end{equation}
\end{theorem}

\noindent The proofs of these results are given in Subsections \ref{1602041256}, \ref{1601030918},  and \ref{1601031532}, respectively.

\subsection{Proof of Theorem \ref{A1511231653}}\label{1601030918}

Let us first formulate and prove several auxiliary statements.

\medskip
For  $\mathsf{d}\in\mathbb{N}$, $\vect{m}\in (0,\infty)^{\mathsf{d}}$, and $r,s \in \mathbb{R}$, $r<s$, we set
\[D^{(\vect{m})}_{(r,s)}(\vect{t})=\tsum_{\indexvectgamma\in \vect{\Xi}_{(r,s)}^{(\vect{m})}}\mathrm{e}^{\mathbf{i}(\indexvectgamma,\vect{t})}.\]
Let  $\mathsf{d}\geq 2$ everywhere below.
For $1\leq \mathsf{k}\leq \mathsf{d}$, we denote
\[D^{\circ,(\vect{m})}_{\mathsf{k},(r,s)}(\vect{t})=
D_{(r,s)}^{(m_1,\ldots,m_{\mathsf{k}-1},m_{\mathsf{k}+1},\ldots,m_{\mathsf{d}})}
(t_1,\ldots,t_{\mathsf{k}-1},t_{\mathsf{k}+1},\ldots,t_{\mathsf{d}}).\]
For  $2\leq \mathsf{k}\leq \mathsf{d}$, we introduce
\[
D^{\sharp,(\vect{m})}_{\mathsf{k},(r,s)}(\vect{t})=
D_{(r,s)}^{(m_1,\ldots,m_{\mathsf{k}-1},m_{\mathsf{k}+1},\ldots,m_{\mathsf{d}})}
(t_1,\ldots,t_{\mathsf{k}-2},t_{\mathsf{k}-1}+t_{\mathsf{k}}m_{\mathsf{k}}/m_{\mathsf{k}-1},t_{\mathsf{k}+1},\ldots,t_{\mathsf{d}}),
\]
\begin{equation}\label{star1}
  \Delta^{\sharp,(\vect{m})}_{\mathsf{k},(r,s)}(\vect{t})=D^{\sharp,(\vect{m})}_{\mathsf{k},(r,s)}(\vect{t})-
D^{\circ,(\vect{m})}_{\mathsf{k},(r,s)}(\vect{t}),
\end{equation}
and
\begin{equation*}
  \begin{split}
     F_{\mathsf{k},(r,s)}^{\sharp,(\vect{m})}(\vect{t})={ \frac{\mathrm{e}^{\mathbf{i}t_{\mathsf{k}}}}{\mathrm{e}^{\mathbf{i}t_{\mathsf{k}}}-1}}\sum\nolimits^{\circ\mathsf{k}}&
\mathrm{e}^{\mathbf{i}(\gamma_1t_1+\ldots+\gamma_{\mathsf{k}-2}t_{\mathsf{k}-2}
+\gamma_{\mathsf{k}+1}t_{\mathsf{k}+1}+\ldots+\gamma_{\mathsf{d}}t_{\mathsf{d}})}\\
&\times\mathrm{e}^{\mathbf{i}\gamma_{\mathsf{k}-1}(t_{\mathsf{k}-1}+t_{\mathsf{k}} m_{\mathsf{k}}/m_{\mathsf{k}-1})}(\mathrm{e}^{-\mathbf{i}
\llfloor \gamma_{\mathsf{k}-1}m_{\mathsf{k}}/m_{\mathsf{k}-1}\rrfloor t_{\mathsf{k}}}-1).
   \end{split}
\end{equation*}
Here and in the following, \begin{equation}\label{1505202138}
\displaystyle\sum\nolimits^{\circ\mathsf{k}}\ \text{means the sum over} \ (\gamma_1,\ldots,\gamma_{\mathsf{k}-1},\gamma_{\mathsf{k}+1},\ldots,\gamma_{\mathsf{d}})\in \vect{\Xi}_{(r,s)}^{(m_1,\ldots,m_{\mathsf{k}-1},m_{\mathsf{k}+1},\ldots,m_{\mathsf{d}})}.
\end{equation}
In the special case $\mathsf{k}=\mathsf{d}$, for simplicity, we denote
\[D^{\circ,(\vect{m})}_{(r,s)}(\vect{t})=D^{\circ,(\vect{m})}_{\mathsf{d},(r,s)}(\vect{t}),\quad D^{\sharp,(\vect{m})}_{(r,s)}(\vect{t})=D^{\sharp,(\vect{m})}_{\mathsf{d},(r,s)}(\vect{t}),\] \[\Delta^{\sharp,(\vect{m})}_{(r,s)}(\vect{t})=\Delta^{\sharp,(\vect{m})}_{\mathsf{d},(r,s)}(\vect{t}),\quad F_{(r,s)}^{\sharp,(\vect{m})}(\vect{t})=F_{\mathsf{d},(r,s)}^{\sharp,(\vect{m})}(\vect{t}),\]
and \[G^{(\vect{m})}_{(r,s)}(\vect{t})=\dfrac1{\mathrm{e}^{\mathbf{i}t_{\mathsf{d}}}-1}\left(\Delta^{\sharp,(\vect{m})}_{(r,s)}(\vect{t})-(\mathrm{e}^{\mathbf{i}\lceil rm_{\mathsf{d}}\rceil t_{\mathsf{d}}}-1)D^{\circ,(\vect{m})}_{(r,s)}(\vect{t})\right).\]

\begin{proposition}\label{1601051328}
Let   $\vect{m}\in(0,\infty)^{\mathsf{d}}$ and $r,s \in \mathbb{R}$, $r<s$. Then
\begin{equation}\label{1501042252}
D^{(\vect{m})}_{(r,s)}(\vect{t})
=G^{(\vect{m})}_{(r,s)}(\vect{t})+D^{\sharp,(\vect{m})}_{(r,s)}(\vect{t})+F^{\sharp,(\vect{m})}_{(r,s)}(\vect{t}).
\end{equation}
\end{proposition}

\begin{proof} First, we show that
\begin{equation}\label{1605180017}\begin{split}
D_{(r,s)}^{(m_{\mathsf{d}-1},m_{\mathsf{d}})}(t_{\mathsf{d}-1},t_{\mathsf{d}}) =&G_{(r,s)}^{(m_{\mathsf{d}-1},m_{\mathsf{d}})}(t_{\mathsf{d}-1},t_{\mathsf{d}})+D_{(r,s)}^{\sharp,(m_{\mathsf{d}-1},m_{\mathsf{d}})}(t_{\mathsf{d}-1},t_{\mathsf{d}})\\&+F_{(r,s)}^{\sharp,(m_{\mathsf{d}-1},m_{\mathsf{d}})}(t_{\mathsf{d}-1},t_{\mathsf{d}}).
\end{split}\end{equation}
Indeed, we have that $G_{(r,s)}^{(m_{\mathsf{d}-1},m_{\mathsf{d}})}(t_{\mathsf{d}-1},t_{\mathsf{d}})+D^{\sharp,(m_{\mathsf{d}-1},m_{\mathsf{d}})}_{(r,s)}(t_{\mathsf{d}-1},t_{\mathsf{d}})$ equals
\begin{align*}
&\dfrac1{\mathrm{e}^{\mathbf{i}t_{\mathsf{d}}}-1}\Delta^{\sharp,(m_{\mathsf{d}-1},m_{\mathsf{d}})}_{(r,s)}(t_{\mathsf{d}-1},t_{\mathsf{d}})-\dfrac{\mathrm{e}^{\mathbf{i}\lceil rm_{\mathsf{d}}\rceil t_{\mathsf{d}}}-1}{\mathrm{e}^{\mathbf{i}t_{\mathsf{d}}}-1}D_{(r,s)}^{(m_{\mathsf{d}-1})}(t_{\mathsf{d}-1})+D_{(r,s)}^{(m_{\mathsf{d}-1})}(t_{\mathsf{d}-1}+t_{\mathsf{d}} m_{\mathsf{d}}/m_{\mathsf{d}-1})\\
&=\dfrac1{\mathrm{e}^{\mathbf{i}t_{\mathsf{d}}}-1}\left(\mathrm{e}^{\mathbf{i}t_{\mathsf{d}}}D_{(r,s)}^{(m_{\mathsf{d}-1})}(t_{\mathsf{d}-1}+t_{\mathsf{d}} m_{\mathsf{d}}/m_{\mathsf{d}-1})-\mathrm{e}^{\mathbf{i}\lceil rm_{\mathsf{d}}\rceil t_{\mathsf{d}}}D_{(r,s)}^{(m_{\mathsf{d}-1})}(t_{\mathsf{d}-1})\right),
\end{align*}
and \[F_{(r,s)}^{\sharp,(m_{\mathsf{d}-1},m_{\mathsf{d}})}(t_{\mathsf{d}-1},t_{\mathsf{d}})=\dfrac{\mathrm{e}^{\mathbf{i}t_{\mathsf{d}}}}{\mathrm{e}^{\mathbf{i}t_{\mathsf{d}}}-1}\tsum_{\gamma_{\mathsf{d}-1}=\lceil rm_{\mathsf{d}-1}\rceil}^{\lfloor sm_{\mathsf{d}-1}\rfloor}\mathrm{e}^{\mathbf{i}\gamma_{\mathsf{d}-1}(t_{\mathsf{d}-1}+t_{\mathsf{d}} m_{\mathsf{d}}/m_{\mathsf{d}-1})}\left(\mathrm{e}^{-\mathbf{i}\llfloor \gamma_{\mathsf{d}-1}m_{\mathsf{d}}/m_{\mathsf{d}-1}\rrfloor t_{\mathsf{d}}}-1\right).\]
Now, \eqref{1605180017} follows from
\[\mathrm{e}^{\mathbf{i}\gamma_{\mathsf{d}-1}t_{\mathsf{d}-1}}\mathrm{e}^{\mathbf{i}(\lfloor \gamma_{\mathsf{d}-1}m_{\mathsf{d}}/m_{\mathsf{d}-1}\rfloor +1)t_{\mathsf{d}}}=\mathrm{e}^{\mathbf{i}t_{\mathsf{d}}}\mathrm{e}^{\mathbf{i}\gamma_{\mathsf{d}-1}(t_{\mathsf{d}-1}+t_{\mathsf{d}} m_{\mathsf{d}}/m_{\mathsf{d}-1})}\mathrm{e}^{-\mathbf{i}\llfloor \gamma_{\mathsf{d}-1}m_{\mathsf{d}}/m_{\mathsf{d}-1}\rrfloor t_{\mathsf{d}}}.\]

 For the functions corresponding to the symbols $S\in \{D,D^{\circ},D^{\sharp},\Delta^{\sharp},G,F^ {\sharp}\}$, we have  the descending recursive relation
\begin{equation}\label{1512180056}
S_{(r,s)}^{(m_{\mathsf{i}},\ldots,m_{\mathsf{d}})}(t_{\mathsf{i}},\ldots,t_{\mathsf{d}})
=\tsum_{\gamma_{\mathsf{i}}=\lceil rm_{\mathsf{i}}\rceil}^{\lfloor sm_{\mathsf{i}}\rfloor}\mathrm{e}^{\mathbf{i}\gamma_{\mathsf{i}}t_{\mathsf{i}}}
S_{(r,(\gamma_{\mathsf{i}}/m_{\mathsf{i}}))}^{(m_{\mathsf{i}+1},\ldots,m_{\mathsf{d}})}(t_{\mathsf{i}+1},\ldots,t_{\mathsf{d}}),\quad 1\leq \mathsf{i}\leq \mathsf{d}-2.
\end{equation}

Equality \eqref{1605180017} means that we have \eqref{1501042252} with $(m_{\mathsf{d}-1},m_{\mathsf{d}})$ in place of $\vect{m}$, and $(t_{\mathsf{d}-1},t_{\mathsf{d}})$ in place of $\vect{t}$.
Thus, induction argument using the relation \eqref{1512180056} for $S\in \{D,G,D^{\sharp},F^{\sharp}\}$ implies that for $\mathsf{i}\in\{\mathsf{d}-2,\ldots,2,1\}$ we have  \eqref{1501042252} with $(m_{\mathsf{i}},\ldots,m_{\mathsf{d}})$ in place of $\vect{m}$, and $(t_{\mathsf{i}},\ldots,t_{\mathsf{d}})$  in place of $\vect{t}$. In particular, for $\mathsf{i}=1$, we have \eqref{1501042252}.
\end{proof}

Next, for  $1\leq \mathsf{k}\leq \mathsf{d}-1$, we  introduce
\[D^{\flat,(\vect{m})}_{\mathsf{k},(r,s)}(\vect{t})=D_{(r,s)}^{(m_1,\ldots,m_{\mathsf{k}-1},m_{\mathsf{k}+1},\ldots,m_{\mathsf{d}})}(t_1,\ldots,t_{\mathsf{k}-1},t_{\mathsf{k}+1}+t_{\mathsf{k}}m_{\mathsf{k}}/m_{\mathsf{k}+1},t_{\mathsf{k}+2},\ldots,t_{\mathsf{d}}),\]
and
\[\Delta^{\flat,(\vect{m})}_{\mathsf{k},(r,s)}(\vect{t})
=D^{\flat,(\vect{m})}_{\mathsf{k},(r,s)}(\vect{t})-D^{\circ,(\vect{m})}_{\mathsf{k},(r,s)}(\vect{t}),\]
and, using \eqref{1505202138}, we set
\begin{equation*}
  \begin{split}
     F_{\mathsf{k},(r,s)}^{\flat,(\vect{m})}(\vect{t})=
\dfrac1{\mathrm{e}^{\mathbf{i}t_{\mathsf{k}}}-1}\displaystyle\sum\nolimits^{\circ\mathsf{k}}
&\mathrm{e}^{\mathbf{i}
(\gamma_1t_1+\ldots+\gamma_{\mathsf{k}-1}t_{\mathsf{k}-1}+\gamma_{\mathsf{k}+2}t_{\mathsf{k}+2}+\ldots+\gamma_{\mathsf{d}}t_{\mathsf{d}})}\\
&\times\mathrm{e}^{\mathbf{i}\gamma_{\mathsf{k}+1}(t_{\mathsf{k}+1}+t_{\mathsf{k}} m_{\mathsf{k}}/m_{\mathsf{k}+1})}\left(\mathrm{e}^{\mathbf{i}\llceil\gamma_{\mathsf{k}+1}m_{\mathsf{k}}/m_{\mathsf{k}+1}\rrceil t_{\mathsf{k}}}-1\right).
   \end{split}
\end{equation*}
We also denote
\[G^{(\vect{m})}_{\mathsf{k},(r,s)}(\vect{t})=\dfrac1{\mathrm{e}^{\mathbf{i}t_{\mathsf{k}}}-1}\left\{
  \begin{array}{cl} (\mathrm{e}^{\mathbf{i}(\lfloor sm_1\rfloor+1) t_1}-1)D^{\circ,(\vect{m})}_{1,(r,s)}(\vect{t})-\Delta^{\flat,(\vect{m})}_{1,(r,s)}(\vect{t})& \text{if}\quad \mathsf{k}=1,\\
\Delta^{\sharp,(\vect{m})}_{\mathsf{d},(r,s)}(\vect{t})-(\mathrm{e}^{\mathbf{i}\lceil rm_{\mathsf{d}}\rceil t_{\mathsf{d}}}-1)D^{\circ,(\vect{m})}_{\mathsf{d},(r,s)}(\vect{t})& \text{if} \quad \mathsf{k}=\mathsf{d},\\
\Delta^{\sharp,(\vect{m})}_{\mathsf{k},(r,s)}(\vect{t})-\Delta^{\flat,(\vect{m})}_{\mathsf{k},(r,s)}(\vect{t}) &\hspace*{-1.85em}\text{if} \quad 2\leq \mathsf{k}\leq \mathsf{d}-1,
\end{array}\right.\]

\[F^{(\vect{m})}_{\mathsf{k},(r,s)}(\vect{t})=\left\{
\begin{array}{cll} -F_{1,(r,s)}^{\flat,(\vect{m})}(\vect{t})&\text{if} & \hspace*{-0.5em}\mathsf{k}=1,\\
F_{\mathsf{d},(r,s)}^{\sharp,(\vect{m})}(\vect{t}) &\text{if} & \hspace*{-0.5em}\mathsf{k}=\mathsf{d},\\
F_{\mathsf{k},(r,s)}^{\sharp,(\vect{m})}(\vect{t})-F_{\mathsf{k},(r,s)}^{\flat,(\vect{m})}(\vect{t})&\text{if} \quad 2 \leq&\hspace*{-0.5em} \mathsf{k}\leq \mathsf{d}-1,
\end{array}\right.\]
and
\[H^{(\vect{m})}_{\mathsf{k},(r,s)}(\vect{t})=\left\{
\begin{array}{cll} 0&\text{if} & \hspace*{-0.5em}\mathsf{k}=1,\\
D_{\mathsf{k},(r,s)}^{\sharp,(\vect{m})}(\vect{t})&\text{if} \quad 2 \leq& \hspace*{-0.5em} \mathsf{k}\leq \mathsf{d}.
\end{array}\right.\]


\begin{proposition}\label{A1511231547}
Let  $\vect{m}\in(0,\infty)^{\mathsf{d}}$, $r,s \in \mathbb{R}$, $r<s$, and $\mathsf{k}\in\{1,\ldots,\mathsf{d}\}$. Then
\begin{equation}\label{1606182110}
D^{(\vect{m})}_{(r,s)}(\vect{t})=G^{(\vect{m})}_{\mathsf{k},(r,s)}(\vect{t})+H^{(\vect{m})}_{\mathsf{k},(r,s)}(\vect{t})
+F^{(\vect{m})}_{\mathsf{k},(r,s)}(\vect{t}).
\end{equation}
\end{proposition}

\begin{proof}  In the case  $\mathsf{k}=\mathsf{d}$, the equality \eqref{1606182110}  is proved in Proposition \ref{1601051328}. Let us consider the case $2\leq \mathsf{k}\leq \mathsf{d}-1$.
 By the  definitions of $G^{(\vect{m})}_{\mathsf{d},(r,s)}(\vect{t})$, $\Delta^{\sharp,(\vect{m})}_{\mathsf{d},(r,s)}(\vect{t})$ with $\mathsf{k}$ instead of $\mathsf{d}$, $(m_1,\ldots,m_{\mathsf{k}})$ instead of $\vect{m}$, and $(t_1,\ldots,t_{\mathsf{k}})$ instead of $\vect{t}$, we have
\begin{align*}
&G^{(m_1,\ldots,m_{\mathsf{k}})}_{\mathsf{k},(r,s)}(t_1,\ldots,t_{\mathsf{k}})\\
&\quad\quad\quad\quad=\dfrac1{\mathrm{e}^{\mathbf{i}t_{\mathsf{k}}}-1}\Delta^{\sharp,(m_1,\ldots,m_{\mathsf{k}})}_{\mathsf{k},(r,s)}(t_1,\ldots,t_{\mathsf{k}})-\dfrac{\mathrm{e}^{\mathbf{i}\lceil rm_{\mathsf{k}}\rceil t_{\mathsf{k}}}-1}{\mathrm{e}^{\mathbf{i}t_{\mathsf{k}}}-1}D^{\circ,(m_1,\dots,m_k)}_{\mathsf{k},(r,s)}(t_1,\ldots,t_{\mathsf{k}})\\
&\quad\quad\quad\quad=\dfrac{1}{\mathrm{e}^{\mathbf{i}t_{\mathsf{k}}}-1} D^{\sharp,(m_1,\ldots,m_{\mathsf{k}})}_{\mathsf{k},(r,s)}(t_1,\ldots,t_{\mathsf{k}})-\dfrac{\mathrm{e}^{\mathbf{i}\lceil rm_{\mathsf{k}}\rceil t_{\mathsf{k}}}}{\mathrm{e}^{\mathbf{i}t_{\mathsf{k}}}-1}D^{\circ,(m_1,\dots,m_k)}_{\mathsf{k},(r,s)}(t_1,\ldots,t_{\mathsf{k}}).
\end{align*}
At the same time, Proposition \ref{1601051328} with $(m_1,\ldots,m_{\mathsf{k}})$ instead of $\vect{m}$ and  $(t_1,\ldots,t_{\mathsf{k}})$ instead of $\vect{t}$ gives the equality
\begin{equation}\label{1605201330}\begin{split}
D_{(r,s)}^{(m_1,\ldots,m_{\mathsf{k}})}&(t_1,\ldots,t_{\mathsf{k}})+\dfrac{\mathrm{e}^{\mathbf{i}\lceil rm_{\mathsf{k}}\rceil t_{\mathsf{k}}}}{\mathrm{e}^{\mathbf{i}t_{\mathsf{k}}}-1}D^{\circ,(m_1,\dots,m_k)}_{\mathsf{k},(r,s)}(t_1,\ldots,t_{\mathsf{k}})\\
&=\dfrac{\mathrm{e}^{\mathbf{i}t_{\mathsf{k}}}}{\mathrm{e}^{\mathbf{i}t_{\mathsf{k}}}-1} D^{\sharp,(m_1,\ldots,m_{\mathsf{k}})}_{\mathsf{k},(r,s)}(t_1,\ldots,t_{\mathsf{k}})
+F^{\sharp,(m_1,\ldots,m_{\mathsf{k}})}_{\mathsf{k},(r,s)}(t_1,\ldots,t_{\mathsf{k}}).
\end{split}\end{equation}
 For the functions corresponding to  the symbols $S\in \{D,G,H,F\}$, we have the ascending recursion relation
\begin{equation}\label{A1512022134}S_{\mathsf{k},(r,s)}^{(m_1,\ldots,m_{\mathsf{i}})}(t_1,\ldots,t_{\mathsf{i}})=\tsum_{\gamma_{\mathsf{i}}=\lceil rm_{\mathsf{i}}\rceil}^{\lfloor sm_{\mathsf{i}}\rfloor}\mathrm{e}^{\mathbf{i}\gamma_{\mathsf{i}}t_{\mathsf{i}}}S_{\mathsf{k},((\gamma_{\mathsf{i}}/m_{\mathsf{i}}),s)}^{(m_1,\ldots,m_{\mathsf{i}-1})}(t_1,\ldots,t_{\mathsf{i}-1}),\quad \mathsf{k}+2\leq \mathsf{i}\leq \mathsf{d}.
\end{equation}
Below, we will show that \eqref{1606182110} is satisfied with $(m_1,\ldots,m_{\mathsf{k}+1})$ instead of $\vect{m}$ and $(t_1,\ldots,t_{\mathsf{k}+1})$ instead of $\vect{t}$, i.e.
\begin{equation}\label{1501131319}\begin{split}
D_{(r,s)}^{(m_1,\ldots,m_{\mathsf{k}+1})}(t_1,\ldots,&t_{\mathsf{k}+1}) =
G^{(m_1,\ldots,m_{\mathsf{k}+1})}_{\mathsf{k},(r,s)}(t_1,\ldots,t_{\mathsf{k}+1})\\
&+H^{(m_1,\ldots,m_{\mathsf{k}+1})}_{\mathsf{k},(r,s)}(t_1,\ldots,t_{\mathsf{k}+1})+F^{(m_1,\ldots,m_{\mathsf{k}+1})}_{\mathsf{k},(r,s)}(t_1,\ldots,t_{\mathsf{k}+1}).\end{split}
\end{equation}
If \eqref{1501131319} is shown, then by using induction arguments and the relation \eqref{A1512022134} for $S\in \{D,G,H,F\}$ we obtain \eqref{1606182110}  with $(m_1,\ldots,m_{\mathsf{i}})$ in place of $\vect{m}$, and $(t_1,\ldots,t_{\mathsf{i}})$  in place of $\vect{t}$ for $\mathsf{i}\in\{\mathsf{k}+2,\mathsf{k}+3,\ldots,\mathsf{d}\}$.
In particular, for $\mathsf{i}=\mathsf{d}$ we have formula \eqref{1606182110}.

\medskip

Thus, it remains to show \eqref{1501131319}. By the  definitions of $G^{(\vect{m})}_{\mathsf{k},(r,s)}(\vect{t})$ and $F^{(\vect{m})}_{\mathsf{k},(r,s)}(\vect{t})$
with $(m_1,\ldots,m_{\mathsf{k}+1})$ instead of $\vect{m}$ and $(t_1,\ldots,t_{\mathsf{k}+1})$ instead of $\vect{t}$, we have
\[\begin{split}G^{(m_1,\ldots,m_{\mathsf{k}+1})}_{\mathsf{k},(r,s)}(t_1,\ldots,t_{\mathsf{k}+1})=
\dfrac{D^{\sharp,(m_1,\ldots,m_{\mathsf{k}+1})}_{\mathsf{k},(r,s)}(t_1,\ldots,t_{\mathsf{k}+1})
- D^{\flat,(m_1,\ldots,m_{\mathsf{k}+1})}_{\mathsf{k},(r,s)}(t_1,\ldots,t_{\mathsf{k}+1})}{\mathrm{e}^{\mathbf{i}t_{\mathsf{k}}}-1} \end{split}\]
and
\[F^{(m_1,\ldots,m_{\mathsf{k}+1})}_{\mathsf{k},(r,s)}(t_1,\ldots,t_{\mathsf{k}+1})
=F^{\sharp,(m_1,\ldots,m_{\mathsf{k}+1})}_{\mathsf{k},(r,s)}(t_1,\ldots,t_{\mathsf{k}+1})-F^{\flat,(m_1,\ldots,m_{\mathsf{k}})}_{\mathsf{k},(r,s)}(t_1,\ldots,t_{\mathsf{k}+1}).\]
Therefore, \eqref{1501131319} is equivalent to
\begin{align}
\label{1605201027} D_{(r,s)}^{(m_1,\ldots,m_{\mathsf{k}+1})}&(t_1,\ldots,t_{\mathsf{k}+1})+\dfrac1{\mathrm{e}^{\mathbf{i}t_{\mathsf{k}}}-1}
D^{\flat,(m_1,\ldots,m_{\mathsf{k}+1})}_{\mathsf{k},(r,s)}(\ldots)
+F^{\flat,(m_1,\ldots,m_{\mathsf{k}+1})}_{\mathsf{k},(r,s)}(\ldots)\\
&\nonumber=\dfrac{\mathrm{e}^{\mathbf{i}t_{\mathsf{k}}}}{\mathrm{e}^{\mathbf{i}t_{\mathsf{k}}}-1}
D^{\sharp,(m_1,\ldots,m_{\mathsf{k}+1})}_{\mathsf{k},(r,s)}(t_1,\ldots,t_{\mathsf{k}+1})
+F^{\sharp,(m_1,\ldots,m_{\mathsf{k}+1})}_{\mathsf{k},(r,s)}(t_1,\ldots,t_{\mathsf{k}+1}).
\end{align}
Now, we observe that for   $S\in \{D,D^{\sharp},F^{\sharp}\}$ the equation in
\eqref{A1512022134} is satisfied also for $\mathsf{i}=\mathsf{k}+1$. Hence, \eqref{1605201330} implies that \eqref{1605201027} and, therefore, \eqref{1501131319} is  equivalent to
\begin{equation}\label{1601131350}
\begin{split}
\dfrac1{\mathrm{e}^{\mathbf{i}t_{\mathsf{k}}}-1}&\tsum_{\gamma_{\mathsf{k}+1}=\lceil rm_{\mathsf{k}+1}\rceil}^{\lfloor sm_{\mathsf{k}+1}\rfloor}\mathrm{e}^{\mathbf{i}\gamma_{\mathsf{k}+1}t_{\mathsf{k}+1}}\mathrm{e}^{\mathbf{i}\lceil \gamma_{\mathsf{k}+1}m_{\mathsf{k}}/m_{\mathsf{k}+1}\rceil t_{\mathsf{k}}}D^{\circ,(m_1,\ldots,m_{\mathsf{k}})}_{\mathsf{k},(\gamma_{\mathsf{k}+1}/m_{\mathsf{k}+1},s)}(t_1,\ldots,t_{\mathsf{k}})\\
&=\dfrac1{\mathrm{e}^{\mathbf{i}t_{\mathsf{k}}}-1}D^{\flat,(m_1,\ldots,m_{\mathsf{k}+1})}_{\mathsf{k},(r,s)}(t_1,\ldots,t_{\mathsf{k}+1})+F^{\flat,(m_1,\ldots,m_{\mathsf{k}+1})}_{\mathsf{k},(r,s)}(t_1,\ldots,t_{\mathsf{k}+1}).
\end{split}
\end{equation}
But \eqref{1601131350} easily follows from
\begin{equation}\label{1601022111}
\mathrm{e}^{\mathbf{i}\gamma_{\mathsf{k}+1}t_{\mathsf{k}+1}}\mathrm{e}^{\mathbf{i}\lceil \gamma_{\mathsf{k}+1}m_{\mathsf{k}}/m_{\mathsf{k}+1}\rceil t_{\mathsf{k}}}=\mathrm{e}^{\mathbf{i}\gamma_{\mathsf{k}+1}(t_{\mathsf{k}+1}+t_{\mathsf{k}} m_{\mathsf{k}}/m_{\mathsf{k}+1})}\mathrm{e}^{\mathbf{i}\llceil \gamma_{\mathsf{k}+1}m_{\mathsf{k}}/m_{\mathsf{k}+1}\rrceil t_{\mathsf{k}}}. \end{equation}
Thus, we get \eqref{1605201027} and therefore \eqref{1501131319}.

\medskip

Finally, we consider the case $\mathsf{k}=1$. Equation \eqref{1601022111} yields
\[D_{(r,s)}^{(m_1,m_2)}=G^{(m_1,m_2)}_{1,(r,s)}+F^{(m_1,m_2)}_{1,(r,s)}=G^{(m_1,m_2)}_{1,(r,s)}+H^{(m_1,m_2)}_{1,(r,s)}+F^{(m_1,m_2)}_{1,(r,s)}.\]
Thus, induction arguments and the relation \eqref{A1512022134} for $S\in \{D,G,H,F\}$ yield that for $\mathsf{i}\in\{3,4,\ldots,\mathsf{d}\}$ we have \eqref{1606182110}  with $(m_1,\ldots,m_{\mathsf{i}})$ in place of $\vect{m}$, and $(t_1,\ldots,t_{\mathsf{i}})$  in place of $\vect{t}$. In particular, for $\mathsf{i}=\mathsf{d}$ we have the assertion \eqref{1606182110}.
\end{proof}


\begin{proposition} \label{A1511231706} Let $r,s\in\mathbb{R}$, $r<s$, be fixed.
Then, for all
$\vect{m}\in[1,\infty)^{\mathsf{d}}$ and all $\mathsf{k}\in\{1,\ldots,\mathsf{d}\}$ we have
\[\|G_{\mathsf{k},(r,s)}^{(\vect{m})}\|_{L^1([-\pi,\pi)^{\mathsf{d}})}\lesssim \ln (m_{\mathsf{k}}+1) \|D_{(r,s)}^{(m_1,\ldots,m_{\mathsf{k}-1},m_{\mathsf{k}+1},\ldots,m_{\mathsf{d}})}\|_{L^1([-\pi,\pi)^{\mathsf{d}-1})}.\]
\end{proposition}
\begin{proof}
By using the inequality
\begin{equation}\label{star2}
   \frac1{|\mathrm{e}^{\mathbf{i}t}-1|}\lesssim \frac1{|t|},
   \quad \text{$t\in [-\pi,\pi)\setminus\{0\}$},
\end{equation}
it is easy to see that for  all $m_1,m_{\mathsf{d}}\in [1,\infty)$ we have
\begin{equation}\label{kv1}
  \int_{[-\pi,\pi)} \left|\dfrac{\mathrm{e}^{\mathbf{i}(\lfloor rm_{1}\rfloor+1) t_{1}}-1}{\mathrm{e}^{\mathbf{i}t_{1}}-1}\right|\mathrm{d}t_{1}\lesssim\ln (m_{1}+1)
\end{equation}
and
\begin{equation}\label{kv2}
 \int_{[-\pi,\pi)} \left|\dfrac{\mathrm{e}^{\mathbf{i}\lceil rm_{\mathsf{d}}\rceil t_{\mathsf{d}}}-1}{\mathrm{e}^{\mathbf{i}t_{\mathsf{d}}}-1}\right|\mathsf{d}t_{\mathsf{d}}\lesssim\ln (m_{\mathsf{d}}+1).
\end{equation}

Let $\mathsf{k}\in\{2,\ldots,\mathsf{d}\}$. Denoting
\begin{equation} \label{1502270544}
\vect{A}_{\mathsf{k}}(m_{\mathsf{k}})=\left\{\left.\,\vect{t}\in [-\pi,\pi)^{\mathsf{d}}\,\right|\,|t_{\mathsf{k}}|\leq \frac1{m_{\mathsf{k}}+1}\right\},\quad \vect{B}_{\mathsf{k}}(m_{\mathsf{k}})=[-\pi,\pi)^{\mathsf{d}}\setminus \vect{A}_{\mathsf{k}}(m_{\mathsf{k}}),
\end{equation}
we have
\begin{equation} \label{1505202552}\int_{[-\pi,\pi)^{\mathsf{d}}}\left|\dfrac{\Delta^{\sharp,(\vect{m})}_{\mathsf{k},(r,s)}(\vect{t})}{\mathrm{e}^{\mathbf{i}t_{\mathsf{k}}}-1}\right|\mathsf{d}\vect{t}=\int_{\vect{A}_{\mathsf{k}}(m_{\mathsf{k}})}\left|\dfrac{\Delta^{\sharp,(\vect{m})}_{\mathsf{k},(r,s)}(\vect{t})}{\mathrm{e}^{\mathbf{i}t_{\mathsf{k}}}-1}\right|\mathsf{d}\vect{t}+\int_{\vect{B}_{\mathsf{k}}(m_{\mathsf{k}})}\left|\dfrac{\Delta^{\sharp,(\vect{m})}_{\mathsf{k},(r,s)}(\vect{t})}{\mathrm{e}^{\mathbf{i}t_{\mathsf{k}}}-1}\right|\mathsf{d}\vect{t}=I+J.
\end{equation}
By using \eqref{star1} and \eqref{star2}, for all $\vect{m}\in[1,\infty)^{\mathsf{d}}$, we obtain
\begin{align}
\nonumber J&\leq \int_{\vect{B}_{\mathsf{k}}(m_{\mathsf{k}})}\dfrac{|D^{\sharp,(\vect{m})}_{\mathsf{k},(r,s)}(\vect{t})|+|D^{\circ,(\vect{m})}_{\mathsf{k},(r,s)}(\vect{t})|}{|\mathrm{e}^{\mathbf{i}t_{\mathsf{k}}}-1|}\mathsf{d}\vect{t}
\lesssim\int_{\vect{B}_{\mathsf{k}}(m_{\mathsf{k}})}\dfrac1{|t_{\mathsf{k}}|}|D^{\circ,(\vect{m})}_{\mathsf{k},(r,s)}(\vect{t})|\mathsf{d}\vect{t}\\
\label{1502272315}&\lesssim \ln (m_{\mathsf{k}}+1) \|D_{(r,s)}^{(m_1,\ldots,m_{\mathsf{k}-1},m_{\mathsf{k}+1},\ldots,m_{\mathsf{k}})}\|_{L^1([-\pi,\pi)^{\mathsf{d}-1})}
\end{align}
and
\begin{equation} \label{1505201601}
I\lesssim \int_{\vect{A}_{\mathsf{k}}(m_{\mathsf{k}})}\dfrac1{|t_{\mathsf{k}}|}\left|\Delta^{\sharp,(\vect{m})}_{\mathsf{k},(r,s)}
(\vect{t})\right|\mathsf{d}\vect{t}.
\end{equation}

In the following, we will use the next two well-known statements:\\
{\textit{For all continuously differentiable $2\pi$-periodic  $g:\,\mathbb{R}\to \mathbb{R}$ and $\delta\in\mathbb{R}$ (see \cite[p. 46]{DevoreLorentz1993}):
\begin{equation} \label{1601041903}
\|g(\,\cdot+\delta)-g\|_{L^1([-\pi,\pi))}\leq |\delta|\|g'\|_{L^1([-\pi,\pi))}.
\end{equation}
 For  all trigonometric polynomials $\tau_n$ of degree at most $n$, one has (see \cite[p. 102]{DevoreLorentz1993}):
\begin{equation} \label{1501041845}
\|\tau'_n\|_{L^1([-\pi,\pi))}\leq n\|\tau_n\|_{L^1([-\pi,\pi))}.
\end{equation}}}

Denoting $D^{\circ}_{(r,s)}=D_{(r,s)}^{(m_1,\ldots,m_{\mathsf{k}-1},m_{\mathsf{k}+1},\ldots,m_{\mathsf{k}})}$ and $\delta=t_{\mathsf{k}}m_{\mathsf{k}}/m_{\mathsf{k}-1}$, we can write
\[
\Delta^{\sharp,(\vect{m})}_{\mathsf{k},(r,s)}(\vect{t})=
D^{\circ}_{(r,s)}(t_1,\ldots,t_{\mathsf{k}-2},t_{\mathsf{k}-1}+
\delta,t_{\mathsf{k}+1},\ldots,t_{\mathsf{d}})-D^{\circ}_{(r,s)}(t_1,\ldots,,t_{\mathsf{k}-1},t_{\mathsf{k}+1},\ldots,t_{\mathsf{d}}).
\]
Since the  degree of the trigonometric polynomial
$D^{\circ}_{(r,s)}(t_1,\ldots,t_{\mathsf{k}-1},t_{\mathsf{k}+1},\ldots,t_{\mathsf{d}})$
in the variable $t_{\mathsf{k}-1}$
is at most $\lfloor sm_{\mathsf{k}-1}\rfloor$,
using \eqref{1601041903} and \eqref{1501041845},  we obtain
\[\int_{[-\pi,\pi)} |\Delta^{\sharp,(\vect{m})}_{\mathsf{k},(r,s)}(\vect{t})|\,\mathsf{d}t_{\mathsf{k}-1}\leq  |t_{\mathsf{k}}| \dfrac{m_{\mathsf{k}}}{m_{\mathsf{k}-1}} sm_{\mathsf{k}-1} \int_{[-\pi,\pi)}\left|D^{\circ}_{(r,s)}(t_1,\ldots,t_{\mathsf{k}-1},t_{\mathsf{k}+1},\ldots,t_{\mathsf{k}})\right|\mathrm{d}t_{\mathsf{k}-1}.\]
Thus, since \eqref{1505201601}, we have for  all $\vect{m}\in[1,\infty)^{\mathsf{d}}$ that
\[\displaystyle I\lesssim m_{\mathsf{k}}\displaystyle\int\nolimits_{-1/(m_{\mathsf{k}}+1)}^{1/(m_{\mathsf{k}}+1)}\mathrm{d}t_{\mathsf{k}}\,\|D^{\circ}_{(r,s)}\|_{L^1([-\pi,\pi)^{\mathsf{d}-1})}\lesssim \|D_{(r,s)}^{(m_1,\ldots,m_{\mathsf{k}-1},m_{\mathsf{k}+1},\ldots,m_{\mathsf{k}})}\|_{L^1([-\pi,\pi)^{\mathsf{d}-1})}.
\]
Combining this with \eqref{1502272315}, \eqref{1505202552} yields: For  $\vect{m}\in[1,\infty)^{\mathsf{d}}$, $\mathsf{k}\in\{2,\ldots,\mathsf{d}\}$, we have
\begin{equation}\label{1512180154}
\int_{[-\pi,\pi)^{\mathsf{d}}}\left|\dfrac{\Delta^{\sharp,(\vect{m})}_{\mathsf{k},(r,s)}(\vect{t})}{\mathrm{e}^{\mathbf{i}t_{\mathsf{k}}}-1}\right|\mathsf{d}\vect{t}\lesssim \ln (m_{\mathsf{k}}+1) \|D_{(r,s)}^{(m_1,\ldots,m_{\mathsf{k}-1},m_{\mathsf{k}+1},\ldots,m_{\mathsf{k}})}\|_{L^1([-\pi,\pi)^{\mathsf{d}-1})}.
\end{equation}

In analogy to \eqref{1512180154}, we derive that for all $\vect{m}\in[1,\infty)^{\mathsf{d}}$ and all $\mathsf{k}\in\{1,\ldots,\mathsf{d}-1\}$
\begin{equation} \label{1502272318}\int_{[-\pi,\pi)^{\mathsf{d}}}\left|\dfrac{\Delta^{\flat,(\vect{m})}_{\mathsf{k},(r,s)}(\vect{t})}{\mathrm{e}^{\mathbf{i}t_{\mathsf{k}}}-1}\right|\mathsf{d}\vect{t}\lesssim \ln (m_{\mathsf{k}}+1) \|D_{(r,s)}^{(m_1,\ldots,m_{\mathsf{k}-1},m_{\mathsf{k}+1},\ldots,m_{\mathsf{k}})}\|_{L^1([-\pi,\pi)^{\mathsf{d}-1})}.
\end{equation}
Finally, having in mind  the definition of $G_{\mathsf{k},(r,s)}^{(\vect{m})}$, we finish the proof by combining the inequalities \eqref{1512180154}, \eqref{1502272318}, \eqref{kv1}, and \eqref{kv2}.
\end{proof}

\begin{proposition}\label{A1511231707} Let $r,s\in\mathbb{R}$, $r<s$,  be fixed. Then,

a) for  all $\vect{m}\in\mathbb{N}^{\mathsf{d}}$ and all $\mathsf{k}\in\{2,\ldots,\mathsf{d}\}$, we have
\begin{equation}\label{1601022312}
\|F_{\mathsf{k},(r,s)}^{\sharp,(\vect{m})}\|_{L^1([-\pi,\pi)^{\mathsf{d}})}\lesssim \ln \left(m_{\mathsf{k}-1}+1\right) \|D_{(r,s)}^{(m_1,\ldots,m_{\mathsf{k}-1},m_{\mathsf{k}+1},\ldots,m_{\mathsf{d}})}\|_{L^1([-\pi,\pi)^{\mathsf{d}-1})},
\end{equation}

b) for all $\vect{m}\in\mathbb{N}^{\mathsf{d}}$ and all $\mathsf{k}\in\{1,\ldots,\mathsf{d}-1\}$, we have
\begin{equation}\label{1602280059}\|F_{\mathsf{k},(r,s)}^{\flat, (\vect{m})}\|_{L^1([-\pi,\pi)^{\mathsf{d}})}\lesssim \ln \left(m_{\mathsf{k}+1}+1\right) \|D_{(r,s)}^{(m_1,\ldots,m_{\mathsf{k}-1},m_{\mathsf{k}+1},\ldots,m_{\mathsf{d}})}\|_{L^1([-\pi,\pi)^{\mathsf{d}-1})}.
\end{equation}
\end{proposition}

\begin{proof} Let  $\mathsf{k}\in\{2,\ldots,\mathsf{d}\}$. We will show \eqref{1601022312} for all $\vect{m}\in\mathbb{N}^{\mathsf{d}}$. Denote
\begin{equation}\label{1606182059}
\begin{split}
   Q^{(\vect{m})}_{\mathsf{k},\nu}&(t_1,\ldots,t_{\mathsf{k}-1},t_{\mathsf{k}+1},\ldots,t_{\mathsf{d}})\\
   &=\sum\nolimits^{\circ\mathsf{k}}\mathrm{e}^{\mathbf{i}(\gamma_1t_1+\ldots+\gamma_{\mathsf{k}-1}t_{\mathsf{k}-1}+\gamma_{\mathsf{k}+1}t_{\mathsf{k}+1}+\ldots+\gamma_{\mathsf{d}}t_{\mathsf{d}})}\llfloor  \gamma_{\mathsf{k}-1}m_{\mathsf{k}}/m_{\mathsf{k}-1}\rrfloor^\nu,
\end{split}
\end{equation}
where $\sum\nolimits^{\circ\mathsf{k}}$ is given by \eqref{1505202138}. Using the equality
\[
\dfrac1{t_{\mathsf{k}}}\left(\mathrm{e}^{-\mathbf{i}\llfloor \gamma_{\mathsf{k}-1}m_{\mathsf{k}}/m_{\mathsf{k}-1}\rrfloor t_{\mathsf{k}}}-1\right)=\displaystyle\sum_{\nu=1}^{\infty}\dfrac1{\nu!}(-\mathbf{i})^{\nu}\llfloor \gamma_{\mathsf{k}-1}m_{\mathsf{k}}/m_{\mathsf{k}-1}\rrfloor^{\nu}t_{\mathsf{k}}^{\nu-1},
\]
and  \eqref{star2}, we obtain that for all $\vect{m}\in\mathbb{N}^{\mathsf{d}}$ and  all $\vect{t}\in [-\pi,\pi)^{\mathsf{d}}$
\[\left|F_{\mathsf{k},(r,s)}^{\sharp,(\vect{m})}(\vect{t})\right|\lesssim \sum_{\nu=1}^\infty \frac{\pi^{\nu-1}}{\nu!} \left|Q^{(\vect{m})}_{\mathsf{k},\nu}(t_1,\ldots,t_{\mathsf{k}-2},t_{\mathsf{k}-1}+t_{\mathsf{k}} m_{\mathsf{k}}/m_{\mathsf{k}-1},t_{\mathsf{k}+1},\ldots,t_{\mathsf{d}})\right|.\]
We conclude that for all $\vect{m}\in\mathbb{N}^{\mathsf{d}}$ the following inequality holds
\begin{align*}
\|F_{\mathsf{k},(r,s)}^{\sharp,(\vect{m})}\|_{L^1([-\pi,\pi)^{\mathsf{d}})}\lesssim
 \displaystyle\sum_{\nu=1}^{\infty}\dfrac{\pi^{\nu}}{\nu!}\|Q^{(\vect{m})}_{\mathsf{k},\nu}\|_{L^1([-\pi,\pi)^{\mathsf{d}-1})}.
\end{align*}
Thus, to prove \eqref{1601022312} it is sufficient to verify that for  all $\nu\geq 1$  and all $\vect{m}\in\mathbb{N}^{\mathsf{d}}$
\begin{equation}
\label{A1511231603}\|Q^{(\vect{m})}_{\mathsf{k},\nu}\|_{L^1([-\pi,\pi)^{\mathsf{d}-1})}\lesssim \ln (m_{\mathsf{k}-1}\nu+1)\|D_{(r,s)}^{(m_1,\ldots,m_{\mathsf{k}-1},m_{\mathsf{k}+1},\ldots,m_{\mathsf{d}})}\|_{L^1([-\pi,\pi)^{\mathsf{d}-1})}.
\end{equation}
For this we will use the  $1$-periodic function  $h_{\nu,m}$, $m\geq 1$, determined by
\begin{equation}\label{fuctionh}
  h_{\nu,m}(t)=\left\{ \begin {array}{cl} t^\nu  & \text{if}\quad 0\leq t\leq 1-m^{-1},\\
  m\left(1-m^{-1}\right)^\nu(1-t)& \text{if} \quad 1-m^{-1}\leq t<1.
\end{array}\right.
\end{equation}

Let us abbreviate $m=m_{\mathsf{k}-1}$. Since $\gamma_{\mathsf{k}-1}$, $m$, and $m_{\mathsf{k}}$ are integers, we have that
 \text{$\displaystyle 0\leq \llfloor \gamma_{\mathsf{k}-1}m_{\mathsf{k}}/m\rrfloor\leq 1-m^{-1}$}. Thus, taking into account that by $1$-periodicity of $h_{\nu,m}$ we have \text{$h_{\nu,m}(t)=h_{\nu,m}(\llfloor t\rrfloor)$}, $t\in\mathbb{R}$, we derive
 \begin{equation}\label{1606182100}
   h_{\nu,m}(\gamma_{\mathsf{k}-1}m_{\mathsf{k}}/m)=\llfloor \gamma_{\mathsf{k}-1}m_{\mathsf{k}}/m\rrfloor^\nu.
 \end{equation}

Next, by the Fourier inversion theorem, it holds
\begin{equation}\label{1606182105}
  h_{\nu,m}(t)=\sum_{\mu\in\mathbb{Z}}\widehat{h}_{\nu,m}(\mu)\mathrm{e}^{2\pi\mathbf{i}\mu t}\quad\text{in}\quad L^1([0,1)),
\end{equation} where
\[\widehat{h}_{\nu,m}(\nu)=\int_{[0,1)}h_{\nu,m}(t)\mathrm{e}^{-2\pi\mathbf{i}\nu t}\mathrm{d}t.\]
Combining \eqref{1606182059}, \eqref{1606182105}, \eqref{1606182100}, we get that $Q^{(\vect{m})}_{\mathsf{k},\nu}(t_1,\ldots,t_{\mathsf{k}-1},t_{\mathsf{k}+1},\ldots,t_{\mathsf{d}})$ equals
\begin{align*}
&\sum\nolimits^{\circ\mathsf{k}}
 \mathrm{e}^{\mathbf{i}(\gamma_1t_1+\ldots+\gamma_{\mathsf{k}-1}t_{\mathsf{k}-1}+\gamma_{\mathsf{k}+1}t_{\mathsf{k}+1}
 +\gamma_{\mathsf{d}}t_{\mathsf{d}})}\displaystyle\sum_{\mu\in\mathbb{Z}}\widehat{h}_{\nu,m}(\mu)\mathrm{e}^{2\pi\mathbf{i} \mu \gamma_{\mathsf{k}-1}m_{\mathsf{k}}/m}\\
&\quad =\sum_{\mu\in\mathbb{Z}}\widehat{h}_{\nu,m}(\mu)D^{(m_1,\ldots,m_{\mathsf{k}-1},m_{\mathsf{k}+1},\ldots,m_{\mathsf{d}})}_{(r,s)}
(t_1,\ldots,t_{\mathsf{k}-2},t_{\mathsf{k}-1}+2\pi\mu m_{\mathsf{k}}/m,t_{\mathsf{k}+1},\ldots,t_{\mathsf{d}})
\end{align*}
in $L^1([-\pi,\pi)^{\mathsf{d}-1})$. Hence, we have for all $\vect{m}\in\mathbb{N}^{\mathsf{d}}$ that
\begin{equation}
\label{A1511231603+}\|Q^{(\vect{m})}_{\mathsf{k},\nu}\|_{L^1([-\pi,\pi)^{\mathsf{d}-1})}\lesssim \|D_{(r,s)}^{(m_1,\ldots,m_{\mathsf{k}-1},m_{\mathsf{k}+1},\ldots,m_{\mathsf{d}})}\|_{L^1([-\pi,\pi)^{\mathsf{d}-1})}\sum_{\mu\in\mathbb{Z}}\left|\widehat{h}_{\nu,m}(\mu)\right|.
\end{equation}
In \cite{YudinYudin1985} (see also \cite{Ash2010}), it is shown that for all $\nu,m\geq 1$, we have
\begin{equation}\label{this}
  \sum_{\mu\in\mathbb{Z}}\left|\widehat{h}_{\nu,m}(\mu)\right|\lesssim \ln(m \nu+1).
\end{equation}
Combining this inequality and \eqref{A1511231603+}, we get \eqref{A1511231603} and, therefore, we have  \eqref{1601022312}.

By analogy, we can prove \eqref{1602280059}.
\end{proof}

\bigskip

\begin{proofof}{Theorem \ref{A1511231653}}
Inequality \eqref{eqth1} is well-known for $\mathsf{d}=1$, since
\begin{equation}\label{1505202327}
\FL\left(\vect{\Xi}_{(r,s)}\right)=\|D_{(r,s)}^{(\vect{m})}\|_{L^1([-\pi,\pi)^{\mathsf{d}})}.
\end{equation}
Let $\mathsf{d}\geq 2$.
For  all $\mathsf{k} \in \{1, \ldots, \mathsf{d}\}$, we get by Proposition \ref{A1511231547} that
\begin{equation}\label{eqth3}
\begin{split}
    \| D^{(\vect{m})}_{(r,s)}\|_{L^1([-\pi,\pi)^{\mathsf{d}})}\leq
    &\|G^{(\vect{m})}_{\mathsf{k},(r,s)}\|_{L^1([-\pi,\pi)^{\mathsf{d}})}+\|H^{(\vect{m})}_{\mathsf{k},(r,s)}\|_{L^1([-\pi,\pi)^{\mathsf{d}})}\\
    &+\|F^{(\vect{m})}_{\mathsf{k},(r,s)}\|_{L^1([-\pi,\pi)^{\mathsf{d}})}.
\end{split}
\end{equation}
Clearly, for all $\vect{m}\in\mathbb{N}^{\mathsf{d}}$ and all $\mathsf{k}\in\{1,\ldots,\mathsf{d}\}$, we have
\begin{equation}\label{eqth4}
  \| H^{(\vect{m})}_{\mathsf{k},(r,s)}\|_{L^1([-\pi,\pi)^{\mathsf{d}})}\leq \|D_{(r,s)}^{(m_1,\ldots,m_{\mathsf{k}-1},m_{\mathsf{k}+1},\ldots,m_{\mathsf{d}})}\|_{L^1([-\pi,\pi)^{\mathsf{d}-1})}
\end{equation}
and, by Proposition~\ref{A1511231706}, for all $\vect{m}\in\mathbb{N}^{\mathsf{d}}$ and all $\mathsf{k}\in\{1,\ldots,\mathsf{d}\}$ we have
\begin{equation}\label{eqth5}
  \| G^{(\vect{m})}_{\mathsf{k},(r,s)}\|_{L^1([-\pi,\pi)^{\mathsf{d}})}\lesssim \ln \left(m_{\mathsf{k}}+1\right) \|D_{(r,s)}^{(m_1,\ldots,m_{\mathsf{k}-1},m_{\mathsf{k}+1},\ldots,m_{\mathsf{d}})}\|_{L^1([-\pi,\pi)^{\mathsf{d}-1})}.
\end{equation}

Thus, we need to estimate only $\|F_{\mathsf{k},(r,s)}^{(\vect{m})}\|_{L^1([-\pi,\pi)^{\mathsf{d}})}$. This is done with a particular choice of the index $\mathsf{k}$.
Let $\mathsf{k}=\mathsf{k}(\vect{m})$ be such that $m_{\mathsf{i}}\leq m_{\mathsf{k}}$ for all $\mathsf{i}\in\{1,\ldots,\mathsf{d}\}$.
We consider the following three cases:

 (i) If $\mathsf{k}=\mathsf{k}(\vect{m})\in\{2,\ldots,\mathsf{d}-1\}$, we have
\[\|F_{\mathsf{k},(r,s)}^{(\vect{m})}\|_{L^1([-\pi,\pi)^{\mathsf{d}})}\leq \|F_{\mathsf{k},(r,s)}^{\sharp,(\vect{m})}\|_{L^1([-\pi,\pi)^{\mathsf{d}})}+\|F_{\mathsf{k},(r,s)}^{\flat,(\vect{m})}\|_{L^1([-\pi,\pi)^{\mathsf{d}})}\]

\quad\, and \[\ln \left(m_{\mathsf{k}-1}+1\right)+\ln \left(m_{\mathsf{k}+1}+1\right)\leq 2\ln \left(m_{\mathsf{k}}+1\right).\]

(ii) If $\mathsf{k}=\mathsf{k}(\vect{m})=1$, we have \[ \|F_{\mathsf{k},(r,s)}^{(\vect{m})}\|_{L^1([-\pi,\pi)^{\mathsf{d}})}=\|F_{\mathsf{k},(r,s)}^{\flat,(\vect{m})}\|_{L^1([-\pi,\pi)^{\mathsf{d}})}\quad \text{and}\quad \ln \left(m_{\mathsf{k}+1}+1\right)\leq \ln \left(m_{\mathsf{k}}+1\right).\]

(iii) If $\mathsf{k}=\mathsf{k}(\vect{m})=\mathsf{d}$, we have \[\|F_{\mathsf{k},(r,s)}^{(\vect{m})}\|_{L^1([-\pi,\pi)^{\mathsf{d}})}=\|F_{\mathsf{k},(r,s)}^{\sharp,(\vect{m})}\|_{L^1([-\pi,\pi)^{\mathsf{d}})}\quad \text{and}\quad \ln \left(m_{\mathsf{k}-1}+1\right)\leq \ln \left(m_{\mathsf{k}}+1\right).\]
Therefore, by  Proposition \ref{A1511231707}, we get that for all $\vect{m}\in\mathbb{N}^{\mathsf{d}}$ and $\mathsf{k}=\mathsf{k}(\vect{m})$ we have
\begin{equation}\label{eqth6}
  \|F_{\mathsf{k},(r,s)}^{(\vect{m})}\|_{L^1([-\pi,\pi)^{\mathsf{d}})}
\lesssim \ln \left(m_{\mathsf{k}}+1\right)
\|D_{(r,s)}^{(m_1,\ldots,m_{\mathsf{k}-1},m_{\mathsf{k}+1},\ldots,m_{\mathsf{d}})}\|_{L^1([-\pi,\pi)^{\mathsf{d}-1})}.
\end{equation}
Combining \eqref{eqth3}, \eqref{eqth4}, \eqref{eqth5}, and \eqref{eqth6}, we get that for all $\vect{m}\in\mathbb{N}^{\mathsf{d}}$ and $\mathsf{k}=\mathsf{k}(\vect{m})$
\[ \|D_{(r,s)}^{(\vect{m})}\|_{L^1([-\pi,\pi)^{\mathsf{d}})}\lesssim \ln \left(m_{\mathsf{k}}+1\right)
\|D_{(r,s)}^{(m_1,\ldots,m_{\mathsf{k}-1},m_{\mathsf{k}+1},\ldots,m_{\mathsf{d}})}\|_{L^1([-\pi,\pi)^{\mathsf{d}-1})}.\]
Because of \eqref{1505202327},
we get the assertion \eqref{eqth1} by a simple induction argument.
\end{proofof}

\subsection{Proof of Theorem \ref{1602041200}}\label{1602041256}

Let $r,s\in\mathbb{R}$, $s\geq r\geq 0$,  $\vect{m}=(m_1,\ldots,m_{\mathsf{d}})\in (0,\infty)^{\mathsf{d}}$, $\mathsf{d}\in\mathbb{N}$. Let  $S_{\mathsf{d}}$
be the set of all permutations  of $\{1,\ldots,\mathsf{d}\}$, i.e. the set of bijections from $\{1,\ldots,\mathsf{d}\}$ onto $\{1,\ldots,\mathsf{d}\}$.
For $\sigma\in S_{\mathsf{d}}$ and
$(\triangleleft_0,\ldots,\triangleleft_{\mathsf{d}})\in\{\leq,<,=\}^{\mathsf{d}+1}$, let
\begin{equation}\label{1601032151}
\vect{\Xi}_{(r,s),\sigma,(\triangleleft_0,\ldots,\triangleleft_{\mathsf{d}})}^{(\vect{m})}=\left\{\,\vectgamma\in \mathbb{Z}^{\mathsf{d}}\,\left|\ r\triangleleft_{\mathsf{d}} \dfrac{\gamma_{\sigma(\mathsf{d})}}{m_{\sigma(\mathsf{d})}}\triangleleft_{\mathsf{d}-1} \cdots \triangleleft_1  \dfrac{\gamma_{\sigma(1)}}{m_{\sigma(1)}}\triangleleft_0 s\right.\right\}.
\end{equation}

\begin{proposition} \label{1601031215} Let $r,s\in\mathbb{R}$, $s\geq r\geq 0$, be fixed.
Then, for  $\vect{m}\in\mathbb{N}^{\mathsf{d}}$, $\sigma\in S_{\mathsf{d}}$, and  $(\triangleleft_0,\ldots,\triangleleft_{\mathsf{d}})\in\{\leq,<,=\}^{\mathsf{d}+1}$, we have
\begin{equation}\label{1601032151121}
\FL\left(\vect{\Xi}_{(r,s),\sigma,(\triangleleft_0,\ldots,\triangleleft_{\mathsf{d}})}^{(\vect{m})}\right)\lesssim \tprod_{\mathsf{i}=1}^{\mathsf{d}}\ln (m_{\mathsf{i}}+1).
\end{equation}
\end{proposition}
\begin{proof} Note that if \eqref{1601032151121}  is proved for the identity permutation $\sigma=\id|_{\{1,\ldots,\mathsf{d}\}}$, then  \eqref{1601032151121} immediately follows for all $\sigma\in S_{\mathsf{d}}$. Furthermore, we can restrict the considerations to $(\triangleleft_0,\ldots,\triangleleft_{\mathsf{d}})\in\{\leq,<\}^{\mathsf{d}+1}$. Thus,
the proof follows the lines of the proof of Theorem~\ref{A1511231653} in an obvious way.
\end{proof}

We will use sets of the form \eqref{1601032151} as building blocks in order to prove the upper estimate in Theorem \ref{1602041200}. Let us formulate a technical auxiliary statement.

\begin{lemma} \label{1605182149}
Let $\mathfrak{X}^{(\vect{m})}$ be a set of subsets of $\mathbb{R}^{\mathsf{d}}$ and $\vect{m}\in \mathbb{N}^{\mathsf{d}}$. For   $N\in\mathbb{N}$ we denote
\begin{align}
\mathfrak{X}^{(\vect{m})}_{\cap, N}&=\left\{\,\vect{\Xi}_1\cap\ldots\cap\vect{\Xi}_j\,\left|\,\vect{\Xi}_1,\ldots,\vect{\Xi}_j\in \mathfrak{X}^{(\vect{m})},\,j\in\{1,\ldots,N\}\,\right.\right\}, \label{1605182200}\\
\mathfrak{X}^{(\vect{m})}_{\cup, N}&=\left\{\,\vect{\Xi}_1\cup\ldots\cup\vect{\Xi}_j\,\left|\,\vect{\Xi}_1,\ldots,\vect{\Xi}_j\in \mathfrak{X}^{(\vect{m})},\,j\in\{1,\ldots,N\}\,\right.\right\}.  \nonumber
\end{align}
Assume that $N\in\mathbb{N}$ is fixed and that for $\vect{m}\in \mathbb{N}^{\mathsf{d}}$ and all $\vect{\Xi}\in \mathfrak{X}^{(\vect{m})}_{\cap, N}$ we have
\begin{equation}
\label{1605182207}
\FL\left(\vect{\Xi}\right)\lesssim \tprod_{\mathsf{i}=1}^{\mathsf{d}}\ln (m_{\mathsf{i}}+1).
\end{equation}
Then, for the fixed $N\in\mathbb{N}$, the estimate \eqref{1605182207} holds also for all
$\vect{\Xi}\in \mathfrak{X}^{(\vect{m})}_{\cup, N}$.
\end{lemma}
\begin{proof}
The well-known inclusion–exclusion principle yields
\[\FL\left(\vect{\Xi}_1\cup\ldots\cup\vect{\Xi}_j\right) \leq  \tsum_{k = 1}^j \left(\tsum_{1 \leq l_{1} < \cdots < l_{k} \leq j} \FL\left(\vect{\Xi}_{l_1} \cap \cdots \cap \vect{\Xi}_{l_k}\right)\right). \]
Since $N$ is fixed, we conclude the assertion.
\end{proof}

For  $\vect{m}\in\mathbb{N}^{\mathsf{d}}$, we  consider the sets \[\vect{\Gamma}^{(\vect{m})}_0=\left\{\,\vectgamma\in \mathbb{N}_0^{\mathsf{d}}\,\left|\,\forall\,\mathsf{i}:\ 2\gamma_{\mathsf{i}}\leq m_{\mathsf{i}}\right.\right\},\quad
\vect{\Gamma}^{(\vect{m})}_1=\left\{\,\vectgamma\in \mathbb{N}_0^{\mathsf{d}}\,\left|\,\forall\,\mathsf{i}:\ 2\gamma_{\mathsf{i}}<m_{\mathsf{i}}\right.\right\},\]
and we  use the notation
\begin{equation}\label{1601031444} \mathsf{K}^{(\vect{m})}[\vectgamma]=\{\,\mathsf{i}\in\{1,\ldots,\mathsf{d}\}\,|\,\gamma_{\mathsf{i}}/m_{\mathsf{i}}=\max\nolimits^{(\vect{m})}[\vectgamma]\,\}
\end{equation}
with $\max\nolimits^{(\vect{m})}[\vectgamma]=\max\{\gamma_{\mathsf{i}}/m_{\mathsf{i}}\,|\,\mathsf{i}\in\{1,\ldots,\mathsf{d}\}\}$, and for $\emptyset\neq\mathsf{K}\subseteq \{1,\ldots,\mathsf{d}\}$, we denote
\[\vect{\Gamma}^{(\vect{m}),\mathsf{K}}_1=\left\{\,\vectgamma\in \vect{\Gamma}^{(\vect{m})}_1\,\left|\,\mathsf{K}^{(\vect{m})}[\vectgamma]=\mathsf{K}\right. \right\}.\]

\begin{proposition}\label{1605201512} Let $\mathsf{d}\geq 2$ and $\emptyset \neq \mathsf{K}=\{\mathsf{k}_1,\ldots,\mathsf{k}_h\} \subsetneqq \{1,\ldots,\mathsf{d}\}$, $\mathsf{k}_1<\ldots<\mathsf{k}_h$.
Then $\vect{\Gamma}^{(\vect{m}),\mathsf{K}}_1$ is equal to
\begin{equation}\label{1605201918}\bigcup_{\sigma\in S_{\mathsf{d},\mathsf{K}}}\left\{\,\vectgamma\in \mathbb{N}_0^{\mathsf{d}}\,\left|\ 0\leq \dfrac{\gamma_{\sigma(\mathsf{d})}}{m_{\sigma(\mathsf{d})}}\leq\ldots\leq  \dfrac{\gamma_{\sigma(h+1)}}{m_{\sigma(h)}}
< \dfrac{\gamma_{\sigma(h)}}{m_{\sigma(h)}}=\ldots=\dfrac{\gamma_{\sigma(1)}}{m_{\sigma(1)}}<\dfrac12\right.\right\},
\end{equation}
where $S_{\mathsf{d},\mathsf{K}}=\{\,\sigma\in S_{\mathsf{d}}\,|\,\sigma(1)=\mathsf{k}_1,\ldots,\sigma(h)=\mathsf{k}_h\,\}$.
\end{proposition}
\begin{proof}
By the definition, we have
\begin{equation}\label{1605201922}\vect{\Gamma}^{(\vect{m}),\mathsf{K}}_1=\left\{\,\vectgamma\in \mathbb{N}_0^{\mathsf{d}}\,\left|\ \forall\, \mathsf{j}\notin \mathsf{K}: \,0\leq \dfrac{\gamma_{\mathsf{j}}}{m_{\mathsf{j}}}
< \dfrac{\gamma_{\mathsf{k}_h}}{m_{\mathsf{k}_h}}=\ldots=\dfrac{\gamma_{\mathsf{k}_1}}{m_{\mathsf{k}_1}}<\dfrac12\right.\right\}.
\end{equation}
Since for $\sigma\in S_{\mathsf{d},\mathsf{K}}$ we have $\sigma(1)=\mathsf{k}_1,\ldots,\sigma(h)=\mathsf{k}_h$ and $\sigma(h+1),\ldots,\sigma(\mathsf{d})\notin \mathsf{K}$, we conclude that \eqref{1605201918} is a subset of \eqref{1605201922}.

Now, let $\vectgamma$ be an element of \eqref{1605201922}. Then, there exist $\mathsf{j}_{h+1},\ldots,\mathsf{j}_{\mathsf{d}}$ such that
\[ \{\,\mathsf{j}_{h+1},\ldots,\mathsf{j}_{\mathsf{d}}\}=\{1,\ldots,\mathsf{d}\}\setminus\mathsf{K}\quad \text{and}\quad \dfrac{\gamma_{\mathsf{j}_{\mathsf{d}}}}{m_{\mathsf{j}_{\mathsf{d}}}}\leq\ldots\leq  \dfrac{\gamma_{\mathsf{j}_{h+1}}}{m_{\mathsf{j}_{h+1}}}.\]
We set $\sigma(1)=\mathsf{k}_1,\ldots,\sigma(h)=\mathsf{k}_h$ and $\sigma(h+1)=\mathsf{j}_{h+1},\ldots,\sigma(\mathsf{d})=\mathsf{j}_{\mathsf{d}}$. Then, $\sigma\in S_{\mathsf{d},\mathsf{K}}$ and $\vectgamma$ is an element of the
corresponding set in the union \eqref{1605201918}.
\end{proof}

\begin{corollary}\label{1605182040} Let $\emptyset\neq \mathsf{K}\subseteq \{1,\ldots,\mathsf{d}\}$. Then,
for all $\vect{m}\in\mathbb{N}^{\mathsf{d}}$,  we have
\begin{equation}\label{1602280228}\FL\left(\vect{\Gamma}^{(\vect{m}),\mathsf{K}}_1\right)\lesssim \tprod_{\mathsf{i}=1}^{\mathsf{d}}\ln (m_{\mathsf{i}}+1).
\end{equation}
\end{corollary}
\begin{proof} For $\mathsf{K}=\{1,\ldots,\mathsf{d}\}$, we have
\[\vect{\Gamma}^{(\vect{m}),\{1,\ldots,\mathsf{d}\}}_1=\left\{\,\vectgamma\in \mathbb{N}_0^{\mathsf{d}}\,\left|\ 0\leq  \dfrac{\gamma_{\mathsf{d}}}{m_{\mathsf{d}}}=\ldots=\dfrac{\gamma_1}{m_1}<\dfrac12\right.\right\}=\vect{\Xi}_{(0,1/2),\id,(\leq,=,\ldots,=,<)}^{(\vect{m})}.\]
Thus, Proposition  \ref{1601031215} implies \eqref{1602280228}.

Let us now consider the case $\mathsf{d}\geq 2$ with a non-empty set $\emptyset\neq \mathsf{K}\subsetneqq \{1,\ldots,\mathsf{d}\}$. Let
$h \in \{1, \ldots, \mathsf{d}\}$ be the same as in Proposition \ref{1605201512}.
Let $\triangleleft_0$, $\triangleleft_{h}$ be the relation $<$.
If $h\geq 2$, let further $(\triangleleft_1,\ldots,\triangleleft_{h-1})=(=,\ldots,=)$ and
$(\triangleleft_{h+1},\ldots,\triangleleft_{\mathsf{d}})=(\leq,\ldots,\leq)$. With this notation Proposition \ref{1605201512} implies that
 \begin{equation}\label{1605201706}
\vect{\Gamma}^{(\vect{m}),\mathsf{K}}_1=\displaystyle \bigcup_{\sigma\in S_{\mathsf{d},\mathsf{K}}}\vect{\Xi}_{(0,1/2),\sigma,(\triangleleft_0,\ldots,\triangleleft_{\mathsf{d}})}^{(\vect{m})}.
\end{equation}
Let $\mathfrak{X}^{(\vect{m})}=\left\{\,\vect{\Xi}_{(0,1/2),\sigma,(\triangleleft_0,\ldots,\triangleleft_{\mathsf{d}})}^{(\vect{m})}\,\left|\,\sigma\in S_{\mathsf{d},\mathsf{K}}\right.\,\right\}$ and $N=(\mathsf{d}-h)!$.  Then,  \eqref{1605182200} equals
\[\mathfrak{X}^{(\vect{m})}_{\cap, N} = \left\{\,\vect{\Xi}_{(0,1/2),\sigma,(\triangleleft'_0,\ldots,\triangleleft'_{\mathsf{d}})}^{(\vect{m})}\,\left|\,\sigma\in S_{\mathsf{d},\mathsf{K}}\right., \text{$\triangleleft_j'\in\{\leq,=\}$ if $h+1\leq j\leq \mathsf{d}-1$,
$\triangleleft_j'=\triangleleft_j$ else}\,\right\}.\]
Since Proposition  \ref{1601031215} implies \eqref{1605182207} for all sets in $\mathfrak{X}^{(\vect{m})}_{\cap, N}$, Lemma \ref{1605182149} yields \eqref{1605182207} for all sets in $\mathfrak{X}^{(\vect{m})}_{\cup, N}$. Now, taking into account
that by \eqref{1605201706} we have $\vect{\Gamma}^{(\vect{m}),\mathsf{K}}_1\in \mathfrak{X}^{(\vect{m})}_{\cup, N}$, we obtain the assertion \eqref{1602280228}.\end{proof}

For  $\mathsf{k}\in\{1,\ldots,\mathsf{d}\}$, we define
\begin{equation}\label{eq:201607121320}\mathfrak{s}^{(\vect{m})}_{\mathsf{k}}(\vectgamma)= \left(
\gamma_1, \ldots, \gamma_{\mathsf{k}-1}, m_{\mathsf{k}}-\gamma_{\mathsf{k}},\gamma_{\mathsf{k}+1}, \ldots, \gamma_{\mathsf{d}}
 \right).\end{equation}

\begin{proposition} \label{1605182041} For $\vect{m}\in\mathbb{N}^{\mathsf{d}}$, we have
\begin{equation}\label{1602011534}
\displaystyle \overline{\vect{\Gamma}}^{(\vect{m})}=
\vect{\Gamma}^{(\vect{m})}_0\cup\bigcup_{\emptyset\neq \mathsf{K}\subseteq
\{1,\ldots,\mathsf{d}\}}\bigcup_{\mathsf{k}\in \mathsf{K}}\mathfrak{s}^{(\vect{m})}_{\mathsf{k}}\left(\vect{\Gamma}^{(\vect{m}),\mathsf{K}}_1\right)
\end{equation}
and, furthermore, the right hand side of \eqref{1602011534}  is  a  union of pairwise disjoint sets.
\end{proposition}

\begin{proof}
Let $\vectgamma\in\overline{\vect{\Gamma}}^{(\vect{m})}$. We will show that  $\vectgamma$ belongs to the right hand side of \eqref{1602011534}. Since this is clear if $\vectgamma\in\vect{\Gamma}^{(\vect{m})}_0$, we assume that $\vectgamma\notin\vect{\Gamma}^{(\vect{m})}_0$. Since $\vectgamma\notin\vect{\Gamma}^{(\vect{m})}_0$, there exists $\mathsf{k}$ such that $\gamma_{\mathsf{k}}/m_{\mathsf{k}}> 1/2$. Therefore, by the definition of $\overline{\vect{\Gamma}}^{(\vect{m})}$, we have
\begin{equation}\label{1603131453}
\forall\,\mathsf{i}\in \{1,\ldots,\mathsf{d}\}\setminus\{\mathsf{k}\}:\quad  \gamma_{\mathsf{i}}/m_{\mathsf{i}}<1/2.
\end{equation}
Let $\vectgamma'=\mathfrak{s}^{(\vect{m})}_{\mathsf{k}}(\vectgamma)$. Since  $\gamma_{\mathsf{k}}/m_{\mathsf{k}}> 1/2$, we have
$\gamma'_{\mathsf{k}}/m_{\mathsf{k}}<1/2$. Thus, since for $\mathsf{i}\neq \mathsf{k}$ we have $\gamma'_{\mathsf{i}}=\gamma_{\mathsf{i}}$, we get from \eqref{1603131453} that $\vectgamma'\in\vect{\Gamma}^{(\vect{m})}_1$.
By the definition of  $\overline{\vect{\Gamma}}^{(\vect{m})}$, we have
\[\forall\,\mathsf{i}\in \{1,\ldots,\mathsf{d}\}\setminus\{\mathsf{k}\}:\quad \gamma'_{\mathsf{i}}/m_{\mathsf{i}}-\gamma'_{\mathsf{k}}/m_{\mathsf{k}}=\gamma_{\mathsf{i}}/m_{\mathsf{i}}+\gamma_{\mathsf{k}}/m_{\mathsf{k}}-1\leq 1-1=0.\]
Thus, by the definition in \eqref{1601031444}, we have $\mathsf{k}\in \mathsf{K}^{(\vect{m})}[\vectgamma']$. Obviously $\vectgamma'=\mathfrak{s}^{(\vect{m})}_{\mathsf{k}}(\vectgamma)$ implies $\vectgamma=\mathfrak{s}^{(\vect{m})}_{\mathsf{k}}(\vectgamma')$. Thus,  we have $\vectgamma\in \mathfrak{s}^{(\vect{m})}_{\mathsf{k}}\left(\vect{\Gamma}^{(\vect{m}),\mathsf{K}}_1\right)$ with $\mathsf{K}= \mathsf{K}^{(\vect{m})}[\vectgamma']$ and $\mathsf{k}\in \mathsf{K}$.

Now, let $\vectgamma$ belong to the right hand side of \eqref{1602011534}. We will show that $\vectgamma\in\overline{\vect{\Gamma}}^{(\vect{m})}$. This is clear if $\vectgamma\in\vect{\Gamma}^{(\vect{m})}_0$. Suppose $\vectgamma\in \mathfrak{s}^{(\vect{m})}_{\mathsf{k}}\left(\vect{\Gamma}^{(\vect{m}),\mathsf{K}}_1\right)$ with $\mathsf{K}\subseteq \{1,\ldots,\mathsf{d}\}$ and $\mathsf{k}\in \mathsf{K}$.
There is $\vectgamma'\in \vect{\Gamma}^{(\vect{m}),\mathsf{K}}_1$ with $\vectgamma=\mathfrak{s}^{(\vect{m})}_{\mathsf{k}}(\vectgamma')$ and,
by the definition of $\vect{\Gamma}^{(\vect{m}),\mathsf{K}}_1$, we have  $\mathsf{K}= \mathsf{K}^{(\vect{m})}[\vectgamma']$.
Since $\mathsf{k}\in \mathsf{K}= \mathsf{K}^{(\vect{m})}[\vectgamma']$, we have $\gamma'_{\mathsf{i}}/m_{\mathsf{i}}\leq \gamma'_{\mathsf{k}}/m_{\mathsf{k}}$, $\mathsf{i} \in \{1,\ldots,\mathsf{d}\}$, thus
\begin{equation}\label{1603131454}
\forall\,\mathsf{i}\in \{1,\ldots,\mathsf{d}\}\setminus\{\mathsf{k}\}:\quad \gamma_{\mathsf{i}}/m_{\mathsf{i}}+\gamma_{\mathsf{k}}/m_{\mathsf{k}}= \gamma'_{\mathsf{i}}/m_{\mathsf{i}}-\gamma'_{\mathsf{k}}/m_{\mathsf{k}}+1\leq 1.
\end{equation}
Since  for $\mathsf{j}\neq \mathsf{k}$ we have $\gamma_{\mathsf{j}}=\gamma'_{\mathsf{j}}$, and since $\vectgamma'\in \vect{\Gamma}^{(\vect{m})}_1$, we have \eqref{1603131453}, and therefore
\begin{equation}\label{1603131503}
\forall\,\mathsf{i},\mathsf{j}\in \{1,\ldots,\mathsf{d}\}\setminus\{\mathsf{k}\}:\quad\gamma_{\mathsf{i}}/m_{\mathsf{i}}+\gamma_{\mathsf{j}}/m_{\mathsf{j}}<1.
\end{equation}
Combining \eqref{1603131454} and  \eqref{1603131503} yields $\vectgamma\in \overline{\vect{\Gamma}}^{(\vect{m})}$.

Finally, to  complete the proof, we show that the right hand side of  \eqref{1602011534} is the union of pairwise disjoint sets. Let $\vectgamma\in\mathfrak{s}^{(\vect{m})}_{\mathsf{k}}\left(\vect{\Gamma}^{(\vect{m}),\mathsf{K}}_1\right)$ and $\mathsf{k}\in\mathsf{K}$. Then, $\gamma_{\mathsf{k}}>m_{\mathsf{k}}/2$ and therefore $\vectgamma\notin \vect{\Gamma}^{(\vect{m})}_0$. Let furthermore $\vectgamma\in\mathfrak{s}^{(\vect{m})}_{\mathsf{k}'}\left(\vect{\Gamma}^{(\vect{m}),\mathsf{K}'}_1\right)$. Then, $\gamma_{\mathsf{k}'}>m_{\mathsf{k}'}/2$. Therefore, since $\vectgamma\in \overline{\vect{\Gamma}}^{(\vect{m})}$,  we have $\mathsf{k}'=\mathsf{k}$, for otherwise $1<\gamma_{\mathsf{k}}/m_{\mathsf{k}}+\gamma_{\mathsf{k}'}/m_{\mathsf{k}'}\leq 1$. We have $\vectgamma=\mathfrak{s}^{(\vect{m})}_{\mathsf{k}}(\vectgamma')$ for some $\vectgamma'$ that is uniquely determined by $\vectgamma'=\mathfrak{s}^{(\vect{m})}_{\mathsf{k}}(\vectgamma)$. Therefore,  $\vectgamma'\in \vect{\Gamma}^{(\vect{m}),\mathsf{K}}_1$ and $\vectgamma'\in \vect{\Gamma}^{(\vect{m}),\mathsf{K}'}_1$, and we conclude $\mathsf{K}'=\mathsf{K}^{(\vect{m})}[\vectgamma']=\mathsf{K}$.
\end{proof}
 \begin{corollary}\label{1602280259}
For  all $\vect{m}\in\mathbb{N}^{\mathsf{d}}$, we have
 \[\FL\left(\overline{\vect{\Gamma}}^{(\vect{m})}\right)\lesssim \tprod_{\mathsf{i}=1}^{\mathsf{d}}\ln (m_{\mathsf{i}}+1). \]
\end{corollary}
\begin{proof} By Proposition \ref{1605182041}, the  right hand side of \eqref{1602011534}  is  a  union of pairwise disjoint sets. Therefore, \eqref{1602011534}  implies
\begin{equation}\label{1605201434}\FL\left(\overline{\vect{\Gamma}}^{(\vect{m})}\right)\leq \FL\left(\vect{\Gamma}^{(\vect{m})}_0\right)+\tsum_{\emptyset\neq \mathsf{K}\subseteq \{1,\ldots,\mathsf{d}\}}\tsum_{\mathsf{k}\in\mathsf{K}}\FL\left(\mathfrak{s}^{(\vect{m})}_{\mathsf{k}}\left(\vect{\Gamma}^{(\vect{m}),\mathsf{K}}_1\right)\right).
\end{equation}
Clearly
\begin{equation}\label{1605201435}\FL\left(\mathfrak{s}^{(\vect{m})}_{\mathsf{k}}\left(\vect{\Gamma}^{(\vect{m}),\mathsf{K}}_1\right)\right)=\FL\left(\vect{\Gamma}^{(\vect{m}),\mathsf{K}}_1\right)
\end{equation}
and the cross product structure of $\vect{\Gamma}^{(\vect{m})}_0$ implies
\begin{equation}\label{1602280227}
\FL\left(\vect{\Gamma}^{(\vect{m})}_0\right)\lesssim \tprod_{\mathsf{i}=1}^{\mathsf{d}}\ln (m_{\mathsf{i}}+1).
\end{equation}
Combining  \eqref{1605201434}, \eqref{1605201435}, Corollary \ref{1605182040}, and  \eqref{1602280227}
 yields the assertion.
\end{proof}

 \begin{corollary} \label{1501051402}
For  all $\vect{m}\in\mathbb{N}^{\mathsf{d}}$, we have
\[\FL\left(\overline{\vect{\Gamma}}^{(\vect{m}),\ast}\right)\lesssim \tprod_{\mathsf{i}=1}^{\mathsf{d}}\ln (m_{\mathsf{i}}+1).\]
\end{corollary}
\begin{proof}
For $\vect{u}\in \{-1,1\}^{\mathsf{d}}$, we denote $\overline{\vect{\Gamma}}^{(\vect{m})}_{\vect{u}}=\left\{\, (u_1\gamma_1,\ldots,u_{\mathsf{d}}\gamma_{\mathsf{d}})\,\left|\,\vectgamma\in\overline{\vect{\Gamma}}^{(\vect{m})}\right.\right\}$.
Consider $\mathfrak{X}^{(\vect{m})}=\left\{\,\overline{\vect{\Gamma}}^{(\vect{m})}_{\vect{u}}\,|\,\vect{u}\in\{-1,1\}^{\mathsf{d}}\,\right\}$ and $N=2^{\mathsf{d}}$. Then, it is clear that
\begin{equation}\label{eq:201607121208}
\overline{\vect{\Gamma}}^{(\vect{m}),\ast}=\!\!\!\!\displaystyle \bigcup_{\vect{u}\in\{-1,1\}^{\mathsf{d}}}\overline{\vect{\Gamma}}^{(\vect{m})}_{\vect{u}}\in \mathfrak{X}^{(\vect{m})}_{\cup, N}.
\end{equation}

Let $\vect{u}^{(1)},\ldots, \vect{u}^{(j)}\in \{-1,1\}^{\mathsf{d}}$, and \mbox{$\mathsf{M}=\{\,\mathsf{i}\in \{1,\dots,\mathsf{d}\}\,|\,u^{(1)}_{\mathsf{i}}= u^{(2)}_{\mathsf{i}}=\ldots=u^{(j)}_{\mathsf{i}}\,\}$}.
We have $\displaystyle \bigcap_{l=1}^j\overline{\vect{\Gamma}}^{(\vect{m})}_{\vect{u}^{(l)}}=\left\{\,\vectgamma\in \overline{\vect{\Gamma}}^{(\vect{m})}_{\vect{u}^{(1)}}\,\left|\,\gamma_{\mathsf{i}}=0 \ \text{for all}\ \mathsf{i}\notin\mathsf{M} \right.\right\}$.
If  $\emptyset \neq \mathsf{M}=\{\mathsf{i}_1,\ldots,\mathsf{i}_h\}$, \mbox{$\mathsf{i}_1<\ldots<\mathsf{i}_h$},
$(m'_1,\ldots,m'_h)=(m_{\mathsf{i}_1},\ldots,m_{\mathsf{i}_h})$, $(u'_1,\ldots,u'_h)=(u^{(1)}_{\mathsf{i}_1},\ldots,u^{(1)}_{\mathsf{i}_h})$, then
\begin{equation}\label{17052016222}
\FL\left(\,\bigcap_{l=1}^j\overline{\vect{\Gamma}}^{(\vect{m})}_{\vect{u}^{(l)}}\right)=\FL\left(\overline{\vect{\Gamma}}_{(u'_1,\ldots,u'_h)}^{(m'_1,\ldots,m'_h)}\right)=\FL\left(\overline{\vect{\Gamma}}^{(m'_1,\ldots,m'_h)}\right).
\end{equation}
At the same time, if $\mathsf{M}=\emptyset$, then the left hand side in \eqref{17052016222} is $\FL\left(\left\{\vect{0}\right\}\right) = 1$.

Note that $\tprod_{l=1}^{h}\ln (m'_l+1)\leq \tprod_{\mathsf{i}=1}^{\mathsf{d}}\ln (m_{\mathsf{i}}+1)$ for $\mathsf{M}\neq \emptyset$.
Thus, using Corollary  \ref{1602280259} we conclude that for all $\vect{m}\in\mathbb{N}^{\mathsf{d}}$ we have
\begin{equation}\label{1605182243}
\FL\left(\,\bigcap_{l=1}^j\overline{\vect{\Gamma}}^{(\vect{m})}_{\vect{u}^{(l)}}\right)\lesssim \tprod_{\mathsf{i}=1}^{\mathsf{d}}\ln (m_{\mathsf{i}}+1).
\end{equation}
Now, \eqref{1605182243} implies that the assumption \eqref{1605182207} is satisfied and, therefore, taking into account \eqref{eq:201607121208} and Lemma \ref{1605182149} we get the assertion.
\end{proof}
\begin{lemma} \label{1508071608} For $z>0$ and $a\in(0,1]$, we have
\begin{equation}  \label{15080907}
\max\{\ln(az),1\}\geq \max\{a\ln z,1\}\geq a\max\{\ln z,1\}\geq a\ln z.
\end{equation}
\end{lemma}
\begin{proof} The assertion is trivial for $a=1$. Let $a\in(0,1)$. The function $h:\,(0,1]\to\mathbb{R}$, $h(u)= u(1-\ln u)$, $u\in(0,1]$, is increasing, thus $a(1-\ln a)< h(1)=1$.

  We conclude $\ln a> (a-1)(1-\ln a)$, i.e. $\ln a> (a-1)\ln (\mathrm{e}/a)$. Thus, since $a-1<0$, for $z\geq \mathrm{e}/a$ we have  $\ln a> (a-1)\ln z$, i.e. $\ln(az)> a\ln z$.
For $0<z< \mathrm{e}/a$, we conclude  $a\ln z<a\ln(\mathrm{e}/a)=a(1-\ln a)< 1$. We have shown: if $\ln(az)\geq 1$, then we have $\ln(az)> a\ln z$, and
if $\ln(az)<1$, then we have also $a\ln z< 1$.
\end{proof}
\begin{proposition} \label{1602171132} There are $\alpha_{\mathsf{d}}$, $\beta_{\mathsf{d}}>0$ such that for all
$\vect{m}\in\mathbb{N}^{\mathsf{d}}$, we have
\begin{equation}\label{1602291718}
  \FL\left(\overline{\vect{\Gamma}}^{(\vect{m})}\right)\geq \alpha_{\mathsf{d}}\tprod_{\mathsf{i}=1}^{\mathsf{d}}\ln (m_{\mathsf{i}}+1),\qquad
\FL\left(\overline{\vect{\Gamma}}^{(\vect{m}),\ast}\right)\geq \beta_{\mathsf{d}}\tprod_{\mathsf{i}=1}^{\mathsf{d}}\ln (m_{\mathsf{i}}+1).
\end{equation}
\end{proposition}
\begin{proof} We use the   following Hardy-Littlewood  inequality, see\cite[p. 286]{Zygmund}:
\begin{equation}\label{1512202341}
\dfrac1{2}\int_{[-\pi,\pi)}\left|\tsum_{\gamma=0}^{N}c_{\gamma}\mathrm{e}^{\mathbf{i}\gamma t}\right|\,\mathrm{d}t \geq \tsum_{\gamma=0}^{N}\dfrac{|c_{\gamma}|}{\gamma+1},\quad N\in \mathbb{N}_0,\ c_0,\ldots,c_N\in\mathbb{C}.
\end{equation}
By the induction argument from  \cite[p. 69]{Rudin1969}, we get for  $N_1,\ldots,N_{\mathsf{d}}\in\mathbb{N}_0$, $c_{\indexvectgamma}\in\mathbb{C}^{\mathsf{d}}$:
\begin{equation}\label{Hardy-Littlewood}
  \dfrac1{2^{\mathsf{d}}}\int_{[-\pi,\pi)^{\mathsf{d}}}
\left|\tsum_{\gamma_1=0}^{N_1}\cdots \tsum_{\gamma_{\mathsf{d}}=0}^{N_{\mathsf{d}}}c_{\indexvectgamma}\,
\mathrm{e}^{\mathbf{i}(\indexvectgamma,\vect{t})}\right|\,\mathrm{d}\vect{t}\geq \tsum_{\gamma_1=0}^{N_1}\cdots \tsum_{\gamma_{\mathsf{d}}=0}^{N_{\mathsf{d}}}\dfrac{|c_{\indexvectgamma}|}{(\gamma_1+1)\cdot\ldots\cdot(\gamma_{\mathsf{d}}+1)}.
\end{equation}
Using an appropriate shifting and orthogonality, we obtain
\begin{equation}\label{1608070954}
\FL\left(\vect{\Gamma}\right)\geq \dfrac1{\pi^{\mathsf{d}}}\quad \text{for all finite $\emptyset\neq \vect{\Gamma}\subset \mathbb{Z}^{\mathsf{d}}$}.
\end{equation}
By \eqref{Hardy-Littlewood}, we get for $\FL\left(\overline{\vect{\Gamma}}^{(\vect{m})}\right)$ the lower bounds
\[
 \dfrac1{\pi^{\mathsf{d}}}\tsum_{\indexvectgamma\in \overline{\vect{\Gamma}}^{(\vect{m})}}\dfrac1{(\gamma_1+1)\cdot\ldots\cdot(\gamma_{\mathsf{d}}+1)}\geq
\dfrac1{\pi^{\mathsf{d}}} \tsum_{\gamma_1=0}^{\lfloor m_1/2\rfloor}\cdots \tsum_{\gamma_{\mathsf{d}}=0}^{\lfloor m_{\mathsf{d}}/2\rfloor}\dfrac1{(\gamma_1+1)\cdot\ldots\cdot(\gamma_{\mathsf{d}}+1)}.\]
Since
\begin{equation}\label{1605201209}
 \tsum_{\gamma=0}^{\lfloor x\rfloor}\dfrac1{\gamma+1}\geq \max\{\ln (x+1),1\},\qquad x\geq 0,\end{equation}
we  have
  \[\FL\left(\overline{\vect{\Gamma}}^{(\vect{m})}\right)\geq \dfrac1{\pi^{\mathsf{d}}} \tprod_{\mathsf{i}=1}^{\mathsf{d}}\max\{\ln\left(m_{\mathsf{i}}/2+1\right),1\}\geq \dfrac1{\pi^{\mathsf{d}}} \tprod_{\mathsf{i}=1}^{\mathsf{d}}\max\{\ln\left((m_{\mathsf{i}}+1)/2\right),1\},\]
and now Lemma \ref{1508071608} implies the first inequality in \eqref{1602291718} with $\alpha_{\mathsf{d}}=(2\pi)^{-\mathsf{d}}$.

\medskip

Since $\FL\left(\{\gamma\in\mathbb{Z}\,|\,|\gamma|\leq x\}\right)=\FL\left(\{\gamma\in\mathbb{Z}\,|\,0\leq \gamma \leq 2\lfloor x\rfloor\}\right)$, $x\geq 0$, the Hardy-Littlewood inequality \eqref{1512202341} and \eqref{1605201209} imply
\begin{equation}\label{1602161746}
  \FL\left(\{\gamma\in\mathbb{Z}\,|\,|\gamma|\leq x\}\right)\geq \dfrac1{\pi}\tsum_{\gamma=0}^{2\lfloor x\rfloor}\dfrac1{\gamma+1}\geq  \dfrac1{\pi}\max(\ln(x+1),1),\qquad x\geq 0.
\end{equation}

Thus, for $\mathsf{d}=1$,  \eqref{1602161746}  yields  the second inequality in \eqref{1602291718}  with $\beta_1=\pi^{-1}$.
Let us prove this inequality for $\mathsf{d}\geq 2$. We adapt the decomposition approach from \cite{Yudin1979}.  For $\mathsf{j}\in \{1,\ldots,\mathsf{d}\}$, we denote
 \[
 \overline{\vect{\Gamma}}^{(\vect{m}),\ast}_{\mathsf{j},(\gamma)}=\left\{(\gamma_1,\ldots,\gamma_{\mathsf{j}-1},
 \gamma_{\mathsf{j}+1},\ldots,\gamma_{\mathsf{d}})\,|\,(\gamma_1,\ldots,\gamma_{\mathsf{j}-1},\gamma,\gamma_{\mathsf{j}+1},\ldots,\gamma_{\mathsf{d}})\in \overline{\vect{\Gamma}}^{(\vect{m}),\ast}\right\},
 \]
\[
a_{\mathsf{j},(\gamma)}^{(\vect{m})}(t_1,\ldots,t_{\mathsf{j}-1},t_{\mathsf{j}+1},\ldots,t_{\mathsf{d}})
=\!\!\!\!\!\!\!
\tsum_{(\gamma_1,\ldots,\gamma_{\mathsf{j}-1},\gamma_{\mathsf{j}+1},\ldots,\gamma_{\mathsf{d}})\in \overline{\vect{\Gamma}}^{(\vect{m}),\ast}_{\mathsf{j},(\gamma)}}\mathrm{e}^{\gamma_1t_1+\ldots+\gamma_{\mathsf{j}-1}
t_{\mathsf{j}-1}+\gamma_{\mathsf{j}+1}t_{\mathsf{j}+1}+\ldots+\gamma_{\mathsf{d}}t_{\mathsf{d}}}.
\]
 Using the following equality $$
 \tsum_{\indexvectgamma\in\overline{\vect{\Gamma}}^{(\vect{m}),\ast}}\mathrm{e}^{\mathbf{i}(\indexvectgamma,\vect{t})}=\mathrm{e}^{-\mathbf{i} m_{\mathsf{j}}t_{\mathsf{j}}}\tsum_{\gamma=0}^{2m_{\mathsf{j}}}\mathrm{e}^{\mathbf{i}\gamma t_{\mathsf{j}}}a_{\mathsf{j},(\gamma-m_{\mathsf{j}})}^{(\vect{m})}(t_1,\ldots,t_{\mathsf{j}-1},t_{\mathsf{j}+1},\ldots,t_{\mathsf{d}})
 $$
and  \eqref{1512202341}, we get
\[
\dfrac1{2}\int_{[-\pi,\pi)} \left|\tsum_{\indexvectgamma\in\overline{\vect{\Gamma}}^{(\vect{m}),\ast}}
\mathrm{e}^{\mathbf{i}(\indexvectgamma,\vect{t})}\right|\mathrm{d}t_{\mathsf{j}}\geq \tsum_{\gamma=0}^{2m_{\mathsf{j}}}\dfrac1{\gamma+1}|a_{\mathsf{j},(\gamma-m_{\mathsf{j}})}^{(\vect{m})}(t_1,\ldots,t_{\mathsf{j}-1},t_{\mathsf{j}+1},\ldots,t_{\mathsf{d}})|
\]
and, therefore, we derive
\begin{equation}\label{1602270506}
\FL\left(\overline{\vect{\Gamma}}^{(\vect{m}),\ast}\right)\geq \dfrac1{\pi} \tsum_{\gamma=0}^{\lfloor m_{\mathsf{j}}/2\rfloor}\dfrac1{\gamma+1}\FL\left(\overline{\vect{\Gamma}}^{(\vect{m}),\ast}_{\mathsf{j},(m_{\mathsf{j}}-\gamma)}\right).
\end{equation}
Denote $\mathsf{K}=\{\,\mathsf{i}\in\{1,\ldots,\mathsf{d}\}\,|\,\ln((m_{\mathsf{i}}+1)/(4\mathrm{e}^{\mathsf{d}}))\geq 1\}$. Having in mind \eqref{1608070954}, we can assume without restriction  $\mathsf{K}\neq\emptyset$, since we can ensure $\beta_{\mathsf{d}}\in(0,\pi^{-\mathsf{d}}]$.

Let $\mathsf{j}\in\mathsf{K}$.  For all $\gamma\in \{0,\ldots,\lfloor m_{\mathsf{j}}/2\rfloor\}$, we have  the cross product structure
\[\overline{\vect{\Gamma}}^{(\vect{m}),\ast}_{\mathsf{j},(m_{\mathsf{j}}-\gamma)}
=\left\{\,(\gamma_1,\ldots,\gamma_{\mathsf{j}-1},\gamma_{\mathsf{j}+1},\ldots,\gamma_{\mathsf{d}})\,\left|\;\forall\,\mathsf{i}\in\{1,\ldots,\mathsf{d}\}\setminus\{\;\!\mathsf{j}\;\!\}:\;\!|\gamma_{\mathsf{i}}|\leq \dfrac{m_{\mathsf{i}}}{m_{\mathsf{j}}}\gamma\right.\,\right\}.\]
Thus, for all $\gamma\in \{0,\ldots,\lfloor m_{\mathsf{j}}/2\rfloor\}$ the inequality \eqref{1602161746} implies
\begin{equation}\label{1608071001}
\displaystyle\FL\left(\overline{\vect{\Gamma}}^{(\vect{m}),\ast}_{\mathsf{j},(m_{\mathsf{j}}-\gamma)}\right) \geq \dfrac1{\pi^{\mathsf{d}-1}}\tprod_{\mathsf{i}\in\mathsf{K}\setminus\{\mathsf{j}\}}
\ln\left(\dfrac{m_{\mathsf{i}}}{m_{\mathsf{j}}}\gamma+1\right).
\end{equation}
Note that  the product over the empty set $\mathsf{K}\setminus\{\mathsf{j}\}= \emptyset$  is considered as $1$. In  this case, by \eqref{1608070954}, the inequality \eqref{1608071001} is satisfied.
Now, \eqref{1602270506} yields
\[\FL\left(\overline{\vect{\Gamma}}^{(\vect{m}),\ast}\right)\geq \dfrac1{\pi^{\mathsf{d}}}\tsum_{\gamma=1}^{\lfloor m_{\mathsf{j}}/2\rfloor}\dfrac1{\gamma+1}\tprod_{\mathsf{i}\in\mathsf{K}\setminus\{\mathsf{j}\}}
\ln\left(\dfrac12\dfrac{m_{\mathsf{i}}}{m_{\mathsf{j}}}(2\gamma)+1\right).\]
For $\gamma\geq 1$, we have $2\gamma\geq \gamma+1$. Further,  we have $4\leq 2m_{\mathsf{j}}$. We conclude
\begin{align}
\nonumber \pi^{\mathsf{d}}\FL\left(\overline{\vect{\Gamma}}^{(\vect{m}),\ast}\right)&+\tprod_{\mathsf{i}\in\mathsf{K}\setminus\{\mathsf{j}\}}\ln\left(\dfrac{m_{\mathsf{i}}}{4}+1\right)
 \geq \displaystyle\dfrac1{\pi^{\mathsf{d}}}\tsum_{\gamma=0}^{\lfloor m_{\mathsf{j}}/2\rfloor}\dfrac1{\gamma+1}\tprod_{\mathsf{i}\in\mathsf{K}\setminus\{\mathsf{j}\}}
\ln\left(\dfrac12\dfrac{m_{\mathsf{i}}}{m_{\mathsf{j}}}(\gamma+1)+1\right)\\
\label{1608051912}\geq &\displaystyle\int_{0}^{m_{\mathsf{j}}/2} \dfrac{1}{v+1} \tprod_{\mathsf{i}\in\mathsf{K}\setminus\{\mathsf{j}\}} \ln \left( \dfrac12\dfrac{m_{\mathsf{i}}}{m_{\mathsf{j}}}v+1\right)\mathrm{d} v.
\end{align}
Next, for $r>0$, we derive
\begin{align*}
\tsum_{\mathsf{j}\in\mathsf{K}}\displaystyle\int_0^{rm_{\mathsf{j}}}&\dfrac{1}{v+1}\tprod_{\mathsf{i}\in\mathsf{K}\setminus\{\mathsf{j}\}}\ln\left(\dfrac12\dfrac{m_{\mathsf{i}}}{m_{\mathsf{j}}}v +1\right)\mathrm{d}v=\tsum_{\mathsf{j}\in\mathsf{K}}\displaystyle \int_0^{r/2} \dfrac{m_{\mathsf{j}}}{m_{\mathsf{j}}\tau +\frac12}\tprod_{\mathsf{i}\in\mathsf{K}\setminus\{\mathsf{j}\}}\ln\left(m_{\mathsf{i}}\tau+1\right)\mathrm{d}\tau\\
&\geq \displaystyle \int_0^{r/2}\tsum_{\mathsf{j}\in\mathsf{K}}\dfrac{m_{\mathsf{j}}}{m_{\mathsf{j}}\tau +1}\tprod_{\mathsf{i}\in\mathsf{K}\setminus\{\mathsf{j}\}}\ln\left(m_{\mathsf{i}}\tau+1\right)\mathrm{d}\tau=\tprod_{\mathsf{i}\in\mathsf{K}}\ln\left(\dfrac12rm_{\mathsf{i}} +1\right)
\end{align*}
and, therefore, there exists  $\mathsf{k}\in\mathsf{K}$ such that
\begin{equation}\label{1608051847}
\displaystyle\int_0^{rm_{\mathsf{k}}}\dfrac{1}{v+1}\tprod_{\mathsf{i}\in\mathsf{K}\setminus\{\mathsf{k}\}}\ln\left(\dfrac12\dfrac{m_{\mathsf{i}}}{m_{\mathsf{k}}}v +1\right)\mathrm{d}v \geq \dfrac1{|\mathsf{K}|}\tprod_{\mathsf{i}\in\mathsf{K}}\ln\left(\dfrac12rm_{\mathsf{i}} +1\right).
\end{equation}
Using \eqref{1608051847} with $r=1/2$ and \eqref{1608051912} with $\mathsf{j}=\mathsf{k}$ and taking into account the definition of $\mathsf{K}$,  we obtain
\begin{align}\label{1608061356}\pi^{\mathsf{d}}\FL\left(\overline{\vect{\Gamma}}^{(\vect{m}),\ast}\right)&\geq \dfrac1{\mathsf{d}}\left(\ln\left(\,\dfrac14 m_{\mathsf{k}}+1\right)-\ln \left(\mathrm{e}^{\mathsf{d}}\right)\right)\tprod_{\mathsf{i}\in\mathsf{K}\setminus\{\mathsf{k}\}}\ln\left(\,\dfrac14 m_{\mathsf{i}}+1\right)\\
\nonumber &\geq \dfrac1{\mathsf{d}}\tprod_{\mathsf{i}\in\mathsf{K}}\ln((m_{\mathsf{i}}+1)/(4\mathrm{e}^{\mathsf{d}}))= \dfrac1{\mathsf{d}}\tprod_{\mathsf{i}=1}^{\mathsf{d}}\max\{\ln((m_{\mathsf{i}}+1)/(4\mathrm{e}^{\mathsf{d}})),1\}.
\end{align}
Now, Lemma  \ref{1508071608} implies  the assertion with $\beta_{\mathsf{d}}=\mathsf{d}^{-1}\pi^{-\mathsf{d}}(4\mathrm{e}^{\mathsf{d}})^{-\mathsf{d}}\in(0,\pi^{-\mathsf{d}}]$.
\end{proof}

\medskip

\begin{proofof}{Theorem \ref{1602041200}}
The statement follows  immediately  by combining Corollary~\ref{1602280259}, Corollary~\ref{1501051402}, and Proposition~\ref{1602171132}.
\end{proofof}

\subsection{Proof of Theorem \ref{1511231653}}\label{1601031532}

Let $\mathsf{d}\geq 2$, $\vect{m}\in (0,\infty)^{\mathsf{d}}$, and $r > 0$. Denote $D^{(\vect{m})}_{\Sigma,r}(\vect{t})=\tsum_{\indexvectgamma\in  \vect{\Sigma}^{(\vect{m})}_{r}}\mathrm{e}^{\mathbf{i}(\indexvectgamma,\vect{t})}$
and
$$
\displaystyle \lambda_r^{(m_1,\ldots,m_{\mathsf{j}})}(\gamma_1,\ldots,\gamma_{\mathsf{j}-1})
=m_{\mathsf{j}}\left(r-\tsum_{\mathsf{i}=1}^{\mathsf{j}-1}\dfrac{\gamma_{\mathsf{i}}}{m_{\mathsf{i}}}\right),\quad \mathsf{j}=2,\ldots,\mathsf{d}.
$$
It is easy to see that
\[D^{(\vect{m})}_{\Sigma,r}(\vect{t})=\tsum_{\gamma_1=0}^{\lfloor rm_1\rfloor}\mathrm{e}^{\mathbf{i}\gamma_1t_1}\tsum_{\gamma_2=0}^{\lfloor\lambda_r^{(m_1,m_2)}(\gamma_1)\rfloor}\mathrm{e}^{\mathbf{i}\gamma_2t_2}\ \ \cdots\quad \tsum_{\gamma_{\mathsf{d}}=0}^{\lfloor\lambda_r^{(m_1,\ldots,m_{\mathsf{d}})}(\gamma_1,\ldots,\gamma_{\mathsf{d}-1})\rfloor}\mathrm{e}^{\mathbf{i}\gamma_{\mathsf{d}}t_{\mathsf{d}}}.\]

In what follows, we will need several auxiliary functions given by
 \begin{equation}\label{1511231605}
 \begin{split}
F^{(\vect{m})}_{\Sigma,r}(\vect{t})=&
\tsum_{\gamma_1=0}^{\lfloor rm_1\rfloor}\mathrm{e}^{\mathbf{i}\gamma_1(t_1-m_{\mathsf{d}} t_{\mathsf{d}}/m_1)}\tsum_{\gamma_2=0}^{\lfloor\lambda_r^{(m_1,m_2)}(\gamma_1)\rfloor}\mathrm{e}^{\mathbf{i}\gamma_2(t_2-m_{\mathsf{d}} t_{\mathsf{d}}/m_2)}\ \ \cdots\\
&\cdots \tsum_{\gamma_{\mathsf{d}-1}=0}^{\lfloor\lambda_r^{(m_1,\ldots,m_{\mathsf{d}-1})}(\gamma_1,\ldots,\gamma_{\mathsf{d}-2})\rfloor} \mathrm{e}^{\mathbf{i}\gamma_{\mathsf{d}-1}(t_{\mathsf{d}-1}-m_{\mathsf{d}} t_{\mathsf{d}}/m_{\mathsf{d}-1})}f^{(\vect{m})}_{\Sigma,r}(\gamma_1,\ldots,\gamma_{\mathsf{d}-1},t_{\mathsf{d}}),
\end{split}
\end{equation}
 \begin{equation}\label{1505202044}
f^{(\vect{m})}_{\Sigma,r}(\gamma_1,\ldots,\gamma_{\mathsf{d}-1},t_{\mathsf{d}})=\mathrm{e}^{\mathbf{i}(rm_{\mathsf{d}}+1)
t_{\mathsf{d}}}
\dfrac{\mathrm{e}^{-\mathbf{i}\llfloor \lambda_r^{(m_1,\ldots,m_{\mathsf{d}})}(\gamma_1,\ldots,\gamma_{\mathsf{d}-1})\rrfloor t_{\mathsf{d}}}-1}{\mathrm{e}^{\mathbf{i}t_{\mathsf{d}}}-1},
\end{equation}
and
\begin{align}
\nonumber D^{\circ,(\vect{m})}_{\Sigma,\mathsf{d},r}(\vect{t})&=D^{(m_1,\ldots,m_{\mathsf{d}-1})}_{\Sigma,r}(t_1,\ldots,t_{\mathsf{d}-1}),\\
\nonumber D^{(\vect{m})}_{\Sigma,\mathsf{d},r}(\vect{t})&=D^{(\vect{m})}_{\Sigma,r}
(t_1-m_{\mathsf{d}}t_{\mathsf{d}}/m_1,\ldots,t_{\mathsf{d}-1}-m_{\mathsf{d}}t_{\mathsf{d}}/m_{\mathsf{d}-1}),\\
\displaystyle G^{(\vect{m})}_{\Sigma,r}(\vect{t})&=\dfrac1{\mathrm{e}^{\mathbf{i}t_{\mathsf{d}}}-1}
\label{1605202157}\left(\mathrm{e}^{\mathbf{i}(rm_{\mathsf{d}}+1)t_{\mathsf{d}}}D^{(\vect{m})}_{\Sigma,\mathsf{d},r}(\vect{t})
-D^{\circ,(\vect{m})}_{\Sigma,\mathsf{d},r}(\vect{t})\right).
\end{align}

\begin{proposition} \label{1608081106} We have
$D_{\Sigma,r}^{(\vect{m})}(\vect{t})=G_{\Sigma,r}^{(\vect{m})}(\vect{t})+F_{\Sigma,r}^{(\vect{m})}(\vect{t})$.
\end{proposition}
\begin{proof}
We have for  $\mathsf{i}\in\{1,\ldots,\mathsf{d}-2\}$ the recursive relation
\begin{equation}\label{1511230927}
S_{\Sigma,r}^{(m_{\mathsf{i}},\ldots,m_{\mathsf{d}})}(t_{\mathsf{i}},\ldots,t_{\mathsf{d}})
=\tsum_{\gamma_{\mathsf{i}}=0}^{\lfloor rm_{\mathsf{i}}\rfloor}\mathrm{e}^{\mathbf{i}\gamma_{\mathsf{i}}t_{\mathsf{i}}}
S_{\Sigma,(r-\gamma_{\mathsf{i}}/m_{\mathsf{i}})}^{(m_{\mathsf{i}+1},\ldots,m_{\mathsf{d}})}(t_{\mathsf{i}+1},\ldots,t_{\mathsf{d}}),
\end{equation}
where  $S\in \{D,F\}$. Note that
\begin{align*}
G_{\Sigma,r}^{(m_{\mathsf{d}-1},m_{\mathsf{d}})}(t_{\mathsf{d}-1},t_{\mathsf{d}})&=\dfrac1{\mathrm{e}^{\mathbf{i}t_{\mathsf{d}}}-1}\left(\mathrm{e}^{\mathbf{i}(rm_{\mathsf{d}}+1)t_{\mathsf{d}}}D_{\Sigma,r}^{(m_{\mathsf{d}-1})}(t_{\mathsf{d}-1}-m_{\mathsf{d}} t_{\mathsf{d}}/m_{\mathsf{d}-1})-D_{\Sigma,r}^{(m_{\mathsf{d}-1})}(t_{\mathsf{d}-1})\right),\\
F_{\Sigma,r}^{(m_{\mathsf{d}-1},m_{\mathsf{d}})}(t_{\mathsf{d}-1},t_{\mathsf{d}})&=\dfrac{\mathrm{e}^{\mathbf{i}(m_{\mathsf{d}}+1)t_{\mathsf{d}}}}{\mathrm{e}^{\mathbf{i}t_{\mathsf{d}}}-1}\tsum_{\gamma_{\mathsf{d}-1}=0}^{\lfloor rm_{\mathsf{d}-1}\rfloor}\mathrm{e}^{\mathbf{i}\gamma_{\mathsf{d}-1}(t_{\mathsf{d}-1}-m_{\mathsf{d}} t_{\mathsf{d}}/m_{\mathsf{d}-1})}\left(\mathrm{e}^{-\mathbf{i}\llfloor m_{\mathsf{d}}(r-\gamma_{\mathsf{d}-1}/m_{\mathsf{d}-1})\rrfloor t_{\mathsf{d}}}-1\right).
\end{align*}
Thus, from \eqref{1511230927}  for $S=D$, we  immediately get  the same  recursive relation \eqref{1511230927}   for the function corresponding to the symbol $S=G$.

Next, using the equality
\[\mathrm{e}^{\mathbf{i}\gamma_{\mathsf{d}-1}t_{\mathsf{d}-1}}\mathrm{e}^{\mathbf{i}(\lfloor m_{\mathsf{d}}(r-\gamma_{\mathsf{d}-1}/m_{\mathsf{d}-1})\rfloor+1) t_{\mathsf{d}}}=\mathrm{e}^{\mathbf{i}(rm_{\mathsf{d}}+1)t_{\mathsf{d}}}\mathrm{e}^{\mathbf{i}\gamma_{\mathsf{d}-1}(t_{\mathsf{d}-1}-m_{\mathsf{d}} t_{\mathsf{d}}/m_{\mathsf{d}-1})}\mathrm{e}^{-\mathbf{i}\llfloor m_{\mathsf{d}}(r-\gamma_{\mathsf{d}-1}/m_{\mathsf{d}-1})\rrfloor t_{\mathsf{d}}},\]
we conclude  that $D_{\Sigma,r}^{(m_{\mathsf{d}-1},m_{\mathsf{d}})}=G_{\Sigma,r}^{(m_{\mathsf{d}-1},m_{\mathsf{d}})}+F_{\Sigma,r}^{(m_{\mathsf{d}-1},m_{\mathsf{d}})}$.
Thus, applying the relations \eqref{1511230927} to  $S\in \{D,G,F\}$, we obtain the assertion.
\end{proof}
\begin{proposition} \label{1511231706}
Let $r\in (0,\infty)$. For all $\vect{m}\in [1,\infty)^{\mathsf{d}}$, we have
\[\|G^{(\vect{m})}_{\Sigma,r}\|_{L^1([-\pi,\pi)^{\mathsf{d}})}\lesssim \ln (m_{\mathsf{d}}+1) \|D_{\Sigma,r}^{(m_1,\ldots,m_{\mathsf{d}-1})}\|_{L^1([-\pi,\pi)^{\mathsf{d}-1})}.\]
\end{proposition}
\begin{proof}
We have
\[
G_{\Sigma,r}^{(\vect{m})}(\vect{t})=\dfrac{\Delta_{\Sigma,r}^{(\vect{m})}
(\vect{t})}{\mathrm{e}^{\mathbf{i}t_{\mathsf{d}}}-1}
+L_r^{(m_{\mathsf{d}})}(t_{\mathsf{d}})D_{\Sigma,r}^{(m_1,\ldots,m_{\mathsf{d}-1})}
(t_1-m_{\mathsf{d}}t_{\mathsf{d}}/m_1,\ldots,t_{\mathsf{d}-1}-m_{\mathsf{d}}t_{\mathsf{d}}/m_{\mathsf{d}-1}),
\]
where
$L_r^{(m)}(t)=\dfrac{\mathrm{e}^{\mathbf{i}(rm+1)t}-1}{\mathrm{e}^{\mathbf{i}t}-1}$
and
$\Delta_{\Sigma,r}^{(\vect{m})}(\vect{t})=D^{(\vect{m})}_{\Sigma,\mathsf{d},r}(\vect{t})-D^{\circ,(\vect{m})}_{\Sigma,\mathsf{d},r}(\vect{t})$.

Moreover, by the telescoping sum decomposition, we derive
\begin{equation}\label{DD}
  \Delta_{\Sigma,r}^{(\vect{m})}(\vect{t})=\tsum_{\mathsf{i}=1}^{\mathsf{d}-1}\Delta^{(\vect{m})}_{\Sigma,r,\mathsf{i}}(\vect{t}),
\end{equation}
where
\begin{align*}
\Delta^{(\vect{m})}_{\Sigma,r,\mathsf{i}}(\vect{t})
=&D^{(m_1,\ldots,m_{\mathsf{d}-1})}_{\Sigma,r}
(t_1,\ldots,t_{\mathsf{i}-1},t_{\mathsf{i}}-m_{\mathsf{d}}t_{\mathsf{d}}/m_{\mathsf{i}},\ldots,t_{\mathsf{d}-1}-m_{\mathsf{d}}t_{\mathsf{d}}/m_{\mathsf{d}-1})\\
&-D^{(m_1,\ldots,m_{\mathsf{d}-1})}_{\Sigma,r}(t_1,\ldots,t_{\mathsf{i}-1},t_{\mathsf{i}},t_{\mathsf{i}+1}-m_{\mathsf{d}}t_{\mathsf{d}}/m_{\mathsf{i}+1},\ldots,t_{\mathsf{d}-1}-m_{\mathsf{d}}t_{\mathsf{d}}/m_{\mathsf{d}-1}).
\end{align*}
Using \eqref{DD}, \eqref{1605202157},  and the sets  \eqref{1502270544},  we get
\[\int_{[-\pi,\pi)} |G_{\Sigma,r}^{(\vect{m})}|\,\mathsf{d}\vect{t}\leq \tsum_{\mathsf{i}=1}^{\mathsf{d}} I_{\mathsf{i}} + J,\]
where
\[\displaystyle I_{\mathsf{i}}=\int_{\vect{A}_{\mathsf{d}}(m_{\mathsf{d}})}\dfrac{|\Delta^{(\vect{m})}_{\Sigma,r,\mathsf{i}}(\vect{t})|}{\left|\mathrm{e}^{\mathbf{i}t_{\mathsf{d}}}-1\right|}\,\mathsf{d}\vect{t}\lesssim\int_{\vect{A}_{\mathsf{d}}(m_{\mathsf{d}})}\dfrac1{|t_{\mathsf{d}}|}|\Delta^{(\vect{m})}_{\Sigma,r,\mathsf{i}}(\vect{t})|\,\mathsf{d}\vect{t},\]
and
\[J=\int_{\vect{B}_{\mathsf{d}}(m_{\mathsf{d}})}\dfrac{\left|D^{(\vect{m})}_{\Sigma,\mathsf{d},r}(\vect{t})\right|+\left|D^{\circ,(\vect{m})}_{\Sigma,\mathsf{d},r}(\vect{t})\right|}{\left|\mathrm{e}^{\mathbf{i}t_{\mathsf{d}}}-1\right|}\,\mathsf{d}\vect{t}+\int_{\vect{A}_{\mathsf{d}}(m_{\mathsf{d}})}\left|L_r^{m_{\mathsf{d}}}(t_{\mathsf{d}})D^{(\vect{m})}_{\Sigma,\mathsf{d},r}(\vect{t})\right|\,\mathsf{d}\vect{t}.\]
By \eqref{star2}, we  easily get
$J \lesssim \ln (m_{\mathsf{d}}+1) \|D_{\Sigma,r}^{(m_1,\ldots,m_{\mathsf{d}-1})}\|_{L^1([-\pi,\pi)^{\mathsf{d}-1})}$. Further, we have
\begin{equation}\label{1602281253}
\displaystyle I_{\mathsf{i}}\lesssim m_{\mathsf{d}}\displaystyle\int\nolimits_{-1/(m_{\mathsf{d}}+1)}^{1/(m_{\mathsf{d}}+1)}\mathrm{d}t_{\mathsf{d}}\,
\|D_{\Sigma,r}^{(m_1,\ldots,m_{\mathsf{d}-1})}\|_{L^1([-\pi,\pi)^{\mathsf{d}-1})}
\lesssim\|D_{\Sigma,r}^{(m_1,\ldots,m_{\mathsf{d}-1})}\|_{L^1([-\pi,\pi)^{\mathsf{d}-1})}.
\end{equation}
Indeed, using \eqref{1601041903} and \eqref{1501041845}, we obtain  for $\mathsf{i}\in\{1,\ldots,\mathsf{d}-1\}$ that
\[\displaystyle \int_{[-\pi,\pi)} |\Delta^{(\vect{m})}_{\Sigma,r,\mathsf{i}}(\vect{t})|\,\mathsf{d}t_{\mathsf{i}}\leq  |t_{\mathsf{i}}| \dfrac{m_{\mathsf{d}}}{m_{\mathsf{i}}} rm_{\mathsf{i}} \int_{[-\pi,\pi)}\left|D^{(m_1,\ldots,m_{\mathsf{d}-1})}_{\Sigma,r}(t_1,\ldots,t_{\mathsf{d}-1})\right|\mathrm{d}t_{\mathsf{i}},\]
and, therefore \eqref{1602281253}.
\end{proof}

\begin{proposition}\label{1511231707} Let $r=p/q$ with $p,q\in\mathbb{N}$. For all $\vect{m}\in \mathbb{N}^{\mathsf{d}}$, we have
\[\|F_{\Sigma,r}^{(\vect{m})}\|_{L^1([-\pi,\pi)^{\mathsf{d}})}\lesssim \ln \left(\lcm(q,m_1,\ldots,m_{\mathsf{d}-1})+1\right) \|D^{(m_1,\ldots,m_{\mathsf{d}-1})}_{\Sigma,r}\|_{L^1([-\pi,\pi)^{\mathsf{d}-1})},\]
where  $\lcm(q,m_1,\ldots,m_{\mathsf{d}-1})$ denotes the least common multiple of $q,m_1,\ldots,m_{\mathsf{d}-1}$.
\end{proposition}

\begin{proof}
The proposition can be proved by repeating the proof of Proposition \ref{A1511231707}.
Thus, let us present the sketch of the proof.

Using \eqref{1511231605}, \eqref{1505202044}, and  \eqref{star2}, we get as in the proof of Proposition \ref{A1511231707} that
\[
\|F_{\Sigma,r}^{(\vect{m})}\|_{L^1([-\pi,\pi)^{\mathsf{d}})}
\lesssim \sum_{\nu=1}^{\infty}\dfrac{\pi^{\nu}}{\nu!}
\|Q^{(\vect{m})}_{\Sigma,\nu}\|_{L^1([-\pi,\pi)^{\mathsf{d}-1})},
\]
where
\[Q^{(\vect{m})}_{\Sigma,\nu}(t_1,\ldots,t_{\mathsf{d}-1})=\tsum_{\gamma_1=0}^{\lfloor rm_1\rfloor}\tsum_{\gamma_2=0}^{\lfloor\lambda(\gamma_1)\rfloor}\ldots
\tsum_{\gamma_{\mathsf{d}-1}=0}^{\lfloor\lambda(\gamma_1,\ldots,\gamma_{\mathsf{d}-2})\rfloor}\mathrm{e}^{\mathbf{i}
(t_1\gamma_1+\ldots+t_{\mathsf{d}-1}\gamma_{\mathsf{d}-1})}\llfloor\lambda(\gamma_1,\ldots,\gamma_{\mathsf{d}-1})\rrfloor^{\nu}\]
and $\lambda(\gamma_1,\ldots,\gamma_{\mathsf{j}-1})=\lambda_r^{(m_1,\ldots,m_{\mathsf{j}})}(\gamma_1,\ldots,\gamma_{\mathsf{j}-1})$.

Thus, to finish the proof it is sufficient to verify that for all $\nu\geq 1$ we have
\begin{equation}
\label{1511231603}
\|Q^{(\vect{m})}_{\Sigma,\nu}\|_{L^1([-\pi,\pi)^{\mathsf{d}-1})}
\lesssim \ln ( M\nu+1)\|D_{\Sigma,r}^{(m_1,\ldots,m_{\mathsf{d}-1})}\|_{L^1([-\pi,\pi)^{\mathsf{d}-1})},
\end{equation}
where $M=\lcm(q,m_1,\ldots,m_{\mathsf{d}-1})$.

Taking into account that \text{$0\leq \llfloor\lambda(\gamma_1,\ldots,\gamma_{\mathsf{d}-1})\rrfloor\leq 1-M^{-1}$},
we get in the same way as in the proof of Proposition \ref{A1511231707} that
in $L^1([-\pi,\pi)^{\mathsf{d}-1})$
\[  \begin{split}
&Q^{(\vect{m})}_{\Sigma,\nu}(t_1,\ldots,t_{\mathsf{d}-1})\\
&=\tsum_{\gamma_1=0}^{\lfloor rm_1\rfloor}\ldots\tsum_{\gamma_{\mathsf{d}-1}=0}^{\lfloor\lambda(\gamma_1,\ldots,\gamma_{\mathsf{d}-2})\rfloor}\mathrm{e}^{\mathbf{i}(t_1\gamma_1+\ldots+t_{\mathsf{d}-1}\gamma_{\mathsf{d}-1})}\displaystyle\tsum_{\mu\in\mathbb{Z}}\widehat{h}_{\nu,M}(\mu)\mathrm{e}^{2\pi\mathbf{i} \mu \lambda(\gamma_1,\ldots,\gamma_{\mathsf{d}-1})}\\
&=\sum_{\mu\in\mathbb{Z}}\widehat{h}_{\nu,M}(\mu)\mathrm{e}^{2\pi\mathbf{i}\mu m_{\mathsf{d}}}D_{\Sigma,r}^{(m_1,\ldots,m_{\mathsf{d}-1})}(t_1-2\pi\mu m_{\mathsf{d}}/m_1,\ldots,t_{\mathsf{d}-1}-2\pi\mu m_{\mathsf{d}} /m_{\mathsf{d}-1}),
  \end{split}\]
where the function ${h}_{\nu,M}$ is given by \eqref{fuctionh}.
Thus, using \eqref{this}, we get \eqref{1511231603}.
\end{proof}

\bigskip

\begin{proofof}{Theorem \ref{1511231653}}
The statement of the theorem is well-known for $\mathsf{d}=1$. Remark also that the case $\mathsf{d}=2$ with $r = 1$ is already considered in Theorem~\ref{1602041200}.

Let us prove the upper estimates for $\mathsf{d}\geq 2$ in \eqref{estimatesS}.
Without loss of generality we can assume that $m_1\leq \ldots\leq m_{\mathsf{d}}$. The  upper estimate for $\FL\left( \vect{\Sigma}^{(\vect{m})}_{r}\right)=\|D_{\Sigma,r}^{(\vect{m})}\|_{L^1([-\pi,\pi)^{\mathsf{d}})}$ can now be easily obtained by
using Proposition \ref{1608081106}, Proposition~\ref{1511231706}, Proposition~\ref{1511231707} and the induction argument. Using this, we can conclude  the  upper estimate for  $\FL\left(\vect{\Sigma}^{(\vect{m}),\ast}_r\right)$ in  the same way as in the proof of Corollary \ref{1501051402}.

Let us consider the lower bounds.
As in  the proof of Proposition \ref{1602171132}, we  get
\begin{equation*}
  \begin{split}
\FL\left(\vect{\Sigma}^{(\vect{m})}_{r}\right)&\geq\dfrac1{\pi^{\mathsf{d}}}
 \tsum_{\gamma_1=0}^{\lfloor rm_1/\mathsf{d}\rfloor}\cdots \tsum_{\gamma_{\mathsf{d}}=0}^{\lfloor rm_{\mathsf{d}}/\mathsf{d}\rfloor}\dfrac1{(\gamma_1+1)\cdot\ldots\cdot(\gamma_{\mathsf{d}}+1)}\\
 &\geq \dfrac1{\pi^{\mathsf{d}}} (\min\{r/\mathsf{d},1\})^{\mathsf{d}} \tprod_{\mathsf{i}=1}^{\mathsf{d}}\ln\left(m_{\mathsf{i}}+1\right).
   \end{split}
\end{equation*}

To show the lower bounds for the sets $\vect{\Sigma}^{(\vect{m}),\ast}_{r}$, it is sufficient to prove that there exists  $\kappa_{\mathsf{d}}\in(0,\pi^{-\mathsf{d}}]$ such that for  all $r>0$ and all $\vect{m}\in\mathbb{N}^{\mathsf{d}}$,  we have
\begin{equation}
\label{1602270715}
\FL\left(\vect{\Sigma}^{(\vect{m}),\ast}_{r}\right)\geq \kappa_{\mathsf{d}}  \tprod_{\mathsf{i}=1}^{\mathsf{d}}\max\left\{\ln\left(rm_{\mathsf{i}}+1\right),1\right\},
\end{equation}
since by Lemma  \ref{1508071608} we will
have $\FL\left(\vect{\Sigma}^{(\vect{m}),\ast}_{r}\right)\geq \kappa_{\mathsf{d}} (\min\{r,1\})^{\mathsf{d}} \tprod_{\mathsf{i}=1}^{\mathsf{d}}\ln\left(m_{\mathsf{i}}+1\right)$.

We use the induction argument.  By \eqref{1602161746}, we can choose $\kappa_1=\pi^{-1}$. Let $\mathsf{d}\geq 2$. By analogy with the proof  of  the second inequality in \eqref{1602291718},  we get
\begin{equation}
\label{1602292033}\FL\left(\vect{\Sigma}^{(\vect{m}),\ast}_{r}\right)\geq \dfrac1{\pi}\tsum_{\gamma=0}^{\lfloor rm_{\mathsf{j}}\rfloor}\dfrac1{\gamma+1}\FL\left(\vect{\Sigma}^{(\vect{m}),\ast}_{r,\;\!\mathsf{j},(\lfloor rm_{\mathsf{j}}\rfloor-\gamma)}\right),
\end{equation}
where
 $\displaystyle \vect{\Sigma}^{(\vect{m}),\ast}_{r,\;\!\mathsf{j},(\gamma)}=\left\{(\gamma_1,\ldots,\gamma_{\mathsf{j}-1},\gamma_{\mathsf{j}+1},\ldots,\gamma_{\mathsf{d}})\,|\,(\gamma_1,\ldots,\gamma_{\mathsf{j}-1},\gamma,\gamma_{\mathsf{j}+1},\ldots,\gamma_{\mathsf{d}})\in \vect{\Sigma}^{(\vect{m}),\ast}_{r}\right\}$.
Denote $\mathsf{K}=\{\,\mathsf{i}\in\{1,\ldots,\mathsf{d}\}\,|\,\ln((rm_{\mathsf{i}}+1)/(4\mathrm{e}^{\mathsf{d}}))\geq 1\}$.  We can   assume  without restriction that $\mathsf{K}\neq \emptyset$, since by \eqref{1608070954} the number $\kappa_{\mathsf{d}}$ can be chosen from $(0,\pi^{-\mathsf{d}}]$.

 Let $\mathsf{j}\in \mathsf{K}$. For $\gamma\in \{0,\ldots,\lfloor rm_{\mathsf{j}}\rfloor\}$, we have
\[\vect{\Sigma}^{(\vect{m}),\ast}_{r,\;\!\mathsf{j},(\lfloor rm_{\mathsf{j}}\rfloor-\gamma)}=
\vect{\Sigma}^{(m_1,\ldots,m_{\mathsf{j}-1},m_{\mathsf{j}+1},\ldots,m_{\mathsf{d}}),\ast}_{r-r/m_{\mathsf{j}}+\gamma/m_{\mathsf{j}}}.\] Thus, since $r-r/m_{\mathsf{j}}+\gamma/m_{\mathsf{j}}\geq\gamma/m_{\mathsf{j}}$, using  the induction argument yields
\[\FL\left(\vect{\Sigma}^{(\vect{m}),\ast}_{r,\;\!\mathsf{j},(\lfloor rm_{\mathsf{j}}\rfloor-\gamma)}\right) \geq \kappa_{\mathsf{d}-1}  \tprod_{\substack{\mathsf{i}=1\\\mathsf{i}\neq\mathsf{j}}}^{\mathsf{d}}\max\{\ln\left((\gamma/m_{\mathsf{j}})m_{\mathsf{i}}+1\right),1\}
 \geq \kappa_{\mathsf{d}-1}  \tprod_{\mathsf{i}\in\mathsf{K}\setminus\{\mathsf{j}\}}  \ln((\gamma/m_{\mathsf{j}})m_{\mathsf{i}}+1).\]
Note again that  the product over the empty set  $\mathsf{K}\setminus\{\mathsf{j}\}= \emptyset$  is considered as $1$, and that in this case by  \eqref{1608070954} the last inequality  is satisfied with $\kappa_{\mathsf{d}-1}\in(0,\pi^{-(\mathsf{d}-1)}]$.

We have $r\geq 1/m_{\mathsf{j}}$. By analogy with \eqref{1608051912}, using  \eqref{1602292033} implies
\begin{align*}
\FL\left(\vect{\Sigma}^{(\vect{m}),\ast}_{r}\right)+&\dfrac{\kappa_{\mathsf{d}-1}}{\pi} \tprod_{\mathsf{i}\in\mathsf{K}\setminus\{\mathsf{j}\}}\ln\left(\dfrac12rm_{\mathsf{i}}+1\right)\\
&\geq\displaystyle \dfrac{\kappa_{\mathsf{d}-1}}{\pi}\int_0^{rm_{\mathsf{j}}}\dfrac{1}{v+1}\tprod_{\mathsf{i}\in\mathsf{K}\setminus\{\mathsf{j}\}}\ln\left(\dfrac12\dfrac{m_{\mathsf{i}}}{m_{\mathsf{j}}}v +1\right)\mathrm{d}v.
\end{align*}
There is  $\mathsf{k}\in\mathsf{K}$ satisfying \eqref{1608051847}. Using the first and second inequality in \eqref{15080907}, by analogy with \eqref{1608061356}, we get   \eqref{1602270715} for $\kappa_{\mathsf{d}}=\mathsf{d}^{-1}\pi^{-1}(4\mathrm{e}^{\mathsf{d}})^{-\mathsf{d}}\kappa_{\mathsf{d}-1}\in (0,\pi^{-\mathsf{d}}]$.
\end{proofof}


\section{Interpolation on Lissajous-Chebyshev nodes}\label{1502171626}

We first describe the solution of the interpolation problem \eqref{1502182305} in more detail and collect some notation from \cite{DenckerErb2015a}.

\medskip

Let us consider for $\vectgamma\in\mathbb{N}_0^{\mathsf{d}}$
the  $\mathsf{d}$-variate Chebyshev polynomials
\[T_{\indexvectgamma}(\vect{x})=  T_{\gamma_1}(x_1) \cdot \ldots \cdot T_{\gamma_{\mathsf{d}}}(x_{\mathsf{d}}),\quad \vect{x}\in[-1,1]^{\mathsf{d}},\]
where $T_{\gamma}(x) = \cos (\gamma \arccos x)$.
The  Chebyshev polynomials form an orthogonal basis of the polynomial space
$\Pi^{\mathsf{d}} = \mathrm{span} \{T_{\indexvectgamma}\,|\,\vectgamma\in\mathbb{N}_0^{\mathsf{d}} \}$
with respect to the inner product
\[  \langle f,g \rangle_{w_{\mathsf{d}}} = \dfrac{1}{\pi^{\mathsf{d}}} \int_{[-1,1]^{\mathsf{d}}} f(\vect{x}) \overline{g(\vect{x})} w_{\mathsf{d}}(\vect{x})\,\mathrm{d}\vect{x}, \quad w_{\mathsf{d}}(\vect{x})  = \tprod_{\mathsf{i} = 1}^{\mathsf{d}}  \displaystyle\dfrac{1}{\sqrt{1-x_{\mathsf{i}}^2}}.\]
The corresponding norms of these basis elements are
\[\|T_{\indexvectgamma}\|^2_{w_{\mathsf{d}},2}  = 2^{-\mathfrak{e}(\indexvectgamma)}, \quad \text{where}\quad \mathfrak{e}(\vectgamma)=\#\{\,\mathsf{i}\in\{1,\ldots,\mathsf{d}\}\,|\,\gamma_{\mathsf{i}}\neq 0\,\}.\]

We define
\[\mathcal{N}_{\mathsf{d}}=\{\,\vect{n}=(n_1,\ldots,n_{\mathsf{d}})\in\mathbb{N}^{\mathsf{d}}\,|\,n_1,\ldots,n_{\mathsf{d}}\;  \text{are pairwise relatively prime}\,\}.\]
In the following, let $\vect{n}\in \mathcal{N}_{\mathsf{d}}$, $\epsilon\in\{1,2\}$, and $\vect{\kappa}\in\mathbb{Z}^{\mathsf{d}}$.
For the index set $\vect{\Gamma}^{(\epsilon\vect{n})}_{\vect{\kappa}}$ from the introduction, we define the polynomial vector space
\[\Pi^{(\epsilon\vect{n})}_{\vect{\kappa}} = \vspan \left\{\, T_{\indexvectgamma}\,\left|\, \vectgamma \in \vect{\Gamma}^{(\epsilon\vect{n})}_{\vect{\kappa}} \right. \right\}.\]
Clearly, the system $\{ T_{\indexvectgamma} \left| \, \vectgamma \in \vect{\Gamma}^{(\epsilon\vect{n})}_{\vect{\kappa}} \right.\}$ forms an orthogonal basis of $\Pi^{(\epsilon\vect{n})}_{\vect{\kappa}}$.\\

To introduce the sets $\LC^{(\epsilon\vect{n})}_{\vect{\kappa}}$, we define
\[
\I^{(\epsilon\vect{n})}_{\vect{\kappa}}=\I^{(\epsilon\vect{n})}_{\vect{\kappa},0}\cup \I^{(\epsilon\vect{n})}_{\vect{\kappa},1},
\]
where the sets $\I^{(\epsilon\vect{n})}_{\vect{\kappa},\mathfrak{r}}$, $\mathfrak{r}\in\{0,1\}$, are given by
\[
\I^{(\epsilon\vect{n})}_{\vect{\kappa},\mathfrak{r}}=\left\{\,\vect{i}\in\mathbb{N}_0^{\mathsf{d}}\,\left|\, 0\leq i_{\mathsf{i}}\leq \epsilon n_{\mathsf{i}} \ \text{and} \  i_{\mathsf{i}}\equiv \kappa_{\mathsf{i}}+\mathfrak{r} \tmod 2 \ \text{for all} \ \mathsf{i} \in \{1, \ldots, \mathsf{d}\} \right.\,\right\}.\]
Then, using the notation
\[\vect{z}^{(\epsilon\vect{n})}_{\vect{i}}=\left(z^{(\epsilon n_1)}_{i_1},\ldots,z^{(\epsilon n_{\mathsf{d}})}_{i_{\mathsf{d}}}\right),\quad  z^{(\epsilon n)}_{i}=\cos\left(i\pi/(\epsilon n)\right),\]
the Lissajous-Chebyshev node sets are defined as
\begin{equation} \label{eq:201605170810}
\LC^{(\epsilon\vect{n})}_{\vect{\kappa}}=\left\{\, \vect{z}^{(\epsilon\vect{n})}_{\vect{i}}\,\left|\,\vect{i}\in \I^{(\epsilon\vect{n})}_{\vect{\kappa}} \right.\right\}.
\end{equation}
Note that the mapping $\vect{i}\mapsto \vect{z}^{(\epsilon\vect{n})}_{\vect{i}}$ is
a bijection from $\I^{(\epsilon\vect{n})}_{\vect{\kappa}}$ onto $\LC^{(\epsilon\vect{n})}_{\vect{\kappa}}$.

Further, for $\vect{i}\in \I^{(\epsilon\vect{n})}_{\vect{\kappa}}$, we introduce  the weight $\mathfrak{w}^{(\epsilon\vect{n})}_{\vect{i}}$ by
\begin{equation}\label{1601131435}
\mathfrak{w}^{(\epsilon\vect{n})}_{\vect{i}}=\displaystyle 2^{\#\{\,\mathsf{i}\,|\,0<i_{\mathsf{i}}<\epsilon n_{\mathsf{i}}\,\}}\Bigl/\left(2\,\epsilon^{\mathsf{d}}\! \tprod_{\mathsf{i}=1}^{\mathsf{d}} n_{\mathsf{i}}\right),
\end{equation}
and for $\vectgamma\in \vect{\Gamma}^{(\epsilon\vect{n})}_{\vect{\kappa}}$, we use the notation
\[\mathfrak{f}^{(\epsilon\vect{n})}(\vectgamma)= \max \big\{
\#\{\,\mathsf{i} \in \{1,\ldots, \mathsf{d} \} \,|\,2\gamma_{\mathsf{i}}=\epsilon n_{\mathsf{i}}\,\}-1, 0 \big\}. \]
Note that in the case $\epsilon=1$, we have $\mathfrak{f}^{(\vect{n})}(\vectgamma)=0$ for all $\vectgamma\in \vect{\Gamma}^{(\vect{n})}_{\vect{\kappa}}$.

Finally, for $\vect{i} \in \I^{(\epsilon\vect{n})}_{\vect{\kappa}}$, we introduce on $[-1,1]^{\mathsf{d}}$ the polynomials
\begin{equation}\label{1601161944}
  L^{(\epsilon\vect{n})}_{\vect{\kappa},\vect{i}}(\vect{x}) = \mathfrak{w}^{(\epsilon\vect{n})}_{\vect{i}} \Bigl( \
 \tsum_{\substack{\indexvectgamma \in \vect{\Gamma}^{(\epsilon\vect{n})}_{\vect{\kappa}}}}\!\!2^{\mathfrak{e}(\indexvectgamma)-\mathfrak{f}^{(\epsilon\vect{n})}(\indexvectgamma)} T_{\indexvectgamma}(\vect{z}_{\vect{i}}^{(\epsilon\vect{n})})T_{\indexvectgamma}(\vect{x})  \ - \ T_{\epsilon n_{\mathsf{d}}}(z^{(\epsilon n_{\mathsf{d}})}_{i_{\mathsf{d}}}) \;\!  T_{\epsilon n_{\mathsf{d}}}(x_{\mathsf{d}})\Bigr)
\end{equation}
that by definition belong to the space $\Pi^{(\epsilon\vect{n})}_{\vect{\kappa}}$.

The existence and uniqueness of a solution of the interpolation problem \eqref{1502182305} are guaranteed by the following theorem.
\begin{theorem} \label{thm:201605090601}
For  $f:\,[-1,1]^{\mathsf{d}}\to\mathbb{R}$ the unique solution to the interpolation problem \eqref{1502182305} in the space $\Pi^{(\epsilon\vect{n})}_{\vect{\kappa}}$ is given by the polynomial
\[P^{(\epsilon\vect{n})}_{\!\vect{\kappa}}f(\vect{x}) = \tsum_{\vect{i} \in \I^{(\epsilon\vect{n})}_{\vect{\kappa}}}
f(\vect{z}^{(\epsilon\vect{n})}_{\vect{i}}) L^{(\epsilon\vect{n})}_{\vect{\kappa},\vect{i}}(\vect{x}).
\]
\end{theorem}

The proof of this result is given in \cite{DenckerErb2015a}. Note that for $\epsilon=1$ only the case $\vect{\kappa}=\vect{0}$ was treated.
However, since the general node sets $\LC^{(\vect{n})}_{\vect{\kappa}}$ differ from $\LC^{(\vect{n})}_{\vect{0}}$ only in terms of reflections with
respect to the coordinate axis, the corresponding results can be transferred immediately.

By Theorem \ref{thm:201605090601}, the discrete Lebesgue constant $\Lambda^{(\epsilon\vect{n})}_{\vect{\kappa}}$ introduced in \eqref{eq:201605281730} can be reformulated as
\begin{equation}\label{1602171546}
\Lambda^{(\epsilon\vect{n})}_{\vect{\kappa}} = \max_{\vect{x} \in [-1,1]^{\mathsf{d}}}\tsum_{\vect{i} \in \I^{(\epsilon\vect{n})}_{\vect{\kappa}}} \left|L^{(\epsilon\vect{n})}_{\vect{\kappa},\vect{i}}(\vect{x})\right|.
\end{equation}
As a first auxiliary result to estimate this constant, we prove the following Marcinkiewicz-Zygmund-type inequality.

\begin{proposition} \label{1602162324}
Let $\vect{\kappa} \in \mathbb{Z}^{\mathsf{d}}$, $\epsilon \in \{1,2\}$,
and  $0 < p < \infty$ be fixed.
For all $\vect{n} \in \mathcal{N}_{\mathsf{d}}$ and all  $P\in \Pi^{(\epsilon\vect{n})}_{\vect{\kappa}}$, we have
\begin{equation}\label{1501062106}
\tsum_{\vect{i}\in \I^{(\epsilon\vect{n})}_{\vect{\kappa}}}\mathfrak{w}^{(\epsilon\vect{n})}_{\vect{i}}\left|P\left(\vect{z}^{(\epsilon\vect{n})}_{\vect{i}}\right)\right|^p \displaystyle
\lesssim \|P\|_{w_{\mathsf{d}},\;\;\!\!\!p}^p=\frac{1}{\pi^{\mathsf{d}}} \int_{[-1,1]^{\mathsf{d}}} |P(\vect{x})|^p w_{\mathsf{d}}(\vect{x})\,\mathrm{d}\vect{x}.
\end{equation}
\end{proposition}

\begin{proof}
The proof is based on the idea given in \cite{Xu1996}. We proceed as in the proof of \cite[Lemma 3]{Erb2015} and use the following one-dimensional result from \cite[Theorem 2]{LubinskyMateNevai1987}:

\vspace{1mm}

\noindent \emph{For all  $M\in\mathbb{N}$, $0 \leq \theta_1  < \cdots < \theta_M < 2  \pi$, and for all
univariate trigonometric polynomials $q_m$ of degree at most $m \in \mathbb{N}$, we have the inequality
\begin{equation} \label{eq:201605231009}
\tsum_{{\nu} =1}^M |q_m(\theta_{\nu})|^p \leq \displaystyle\left(m+\dfrac1{2 \eta }\right) \dfrac{(p+1)e}{2\pi} \int_{0}^{2\pi} |q_m(\theta)|^p \mathrm{d} \theta,\end{equation}
where $\eta =\min(\theta_2-\theta_1,\ldots,\theta_M-\theta_{M-1},2\pi-(\theta_M-\theta_1))$}.

\vspace{1mm}

\noindent For  $m\in\mathbb{N}$, $\mathfrak{r} \in \{0,1\}$, we consider the sets $J_{\mathfrak{r}}^{(m)} = \{\,i \in \mathbb{N}_0\,|\ i < 2m,\ i \equiv \mathfrak{r} \tmod 2\,\}$. Suppose that
$i_1,\ldots,i_m \in J_{\mathfrak{r}}^{(m)}$ with $0\leq i_1<\ldots<i_m < 2 m$. Setting $M = m$ and
$\theta_{\nu} = i_{\nu}\pi/m$, we obtain $\eta = 2\pi/m$. Using \eqref{eq:201605231009}, we get for all univariate polynomials $Q$ of degree at most $m$ the inequality
\begin{align} \label{1603091815}
\dfrac1m &\tsum_{i \in J_{\mathfrak{r}}^{(m)}:\, i \leq m}
\hspace{-4mm} \left( 2 - \delta_{0,i} - \delta_{0,m}\right) \left|Q\left(z^{(m)}_{i}\right)\right|^p
= \dfrac1m\tsum_{i \in J_{\mathfrak{r}}^{(m)}} \hspace{-1mm}  \left|Q\left(\textstyle \cos \frac{i \pi}{m}\right)\right|^p
= \dfrac1m\tsum_{\nu = 1}^m \left|Q(\cos \theta_{\nu})\right|^p \notag  \\
&\qquad \leq \left( \textstyle 1+\dfrac1{4\pi }\right) \frac{(p+1)e}{2\pi} \int_{0}^{2\pi} |Q(\cos \theta)|^p \mathrm{d} \theta \leq \frac{3(p+1)}{\pi}\int_{-1}^1  \frac{|Q(x)|^p}{\sqrt{1-x^2}} \mathrm{d}x,
\end{align}
where $\delta_{i,j}$ denotes the Kronecker delta.  

Let $P\in \Pi^{(\epsilon\vect{n})}_{\vect{\kappa}}$. The degree of the univariate polynomial $z_i\mapsto P(z_1,\ldots,z_{\mathsf{d}})$ is at most $\epsilon n_i$.
Now, taking into account the cross product structure of $\I^{(\epsilon\vect{n})}_{\vect{\kappa},\mathfrak{r}}$, $\mathfrak{r}\in \{0,1\}$, the weights defined in \eqref{1601131435},
and applying $\mathsf{d}$ times inequality \eqref{1603091815},
we obtain
\begin{equation} \label{eq:201605291044}
\tsum_{\vect{i}\in \I^{(\epsilon\vect{n})}_{\vect{\kappa},\mathfrak{r}}}\mathfrak{w}^{(\epsilon\vect{n})}_{\vect{i}}
\left|P\left(\vect{z}^{(\epsilon\vect{n})}_{\vect{i}}\right)\right|^p \lesssim \|P\|_{w_{\mathsf{d}},\;\;\!\!\!p}^p\,, \quad \mathfrak{r} \in \{0,1\},
\end{equation}
for $\vect{n} \in \mathcal{N}_{\mathsf{d}}$ and $P\in \Pi^{(\epsilon\vect{n})}_{\vect{\kappa}}$. Since
$\I^{(\epsilon\vect{n})}_{\vect{\kappa}} = \I^{(\epsilon\vect{n})}_{\vect{\kappa},0} \cup \I^{(\epsilon\vect{n})}_{\vect{\kappa},1}$, inequality \eqref{eq:201605291044} yields \eqref{1501062106}.
\end{proof}

\medskip

A slight adaption of the proof of Theorem \ref{1602041200} gives the following result.
 \begin{corollary} \label{1502171206} Let $\epsilon\in\{1,2\}$ and $\vect{\kappa}\in \mathbb{Z}^{\mathsf{d}}$ be fixed.
For  all $\vect{n}\in\mathcal{N}_{\mathsf{d}}$, we have
 \[\FL\left(\vect{\Gamma}^{(\epsilon \vect{n})}_{\vect{\kappa}}\right)\asymp  \FL\left(\vect{\Gamma}^{(\epsilon\vect{n}),\ast}_{\vect{\kappa}}\right)\asymp \tprod_{\mathsf{i}=1}^{\mathsf{d}}\ln (n_{\mathsf{i}}+1).\]
\end{corollary}
\begin{proof}
We use the notation
\[\vect{\Gamma}^{(\epsilon\vect{n})}_{\vect{\kappa},\mathfrak{r}} =
\left\{\vectgamma\in\mathbb{N}_0^{\mathsf{d}}\left|\begin {array}{l}
\text{
$\gamma_{\mathsf{i}}/n_{\mathsf{i}}\leq \epsilon /2$\quad $\forall\,\mathsf{i} \in \{1, \ldots, \mathsf{d}\}$ with $\kappa_{\mathsf{i}}\equiv \mathfrak{r}\tmod 2$,} \\
\text{ $\gamma_{\mathsf{i}}/n_{\mathsf{i}} < \epsilon /2$\quad $\forall\,\mathsf{i} \in \{1, \ldots, \mathsf{d}\}$ with  $\kappa_{\mathsf{i}}\not \equiv \mathfrak{r}\tmod 2$}
\end{array}\right.
\right\}. \]
 Further, using \eqref{1601031444}, let $\vect{\Gamma}^{(\epsilon\vect{n}),\mathsf{K}}_{\vect{\kappa},1}=\left\{\,\vectgamma\in \vect{\Gamma}^{(\epsilon\vect{n})}_{\vect{\kappa},1}\,\left|\,\mathsf{K}^{(\epsilon\vect{n})}[\vectgamma]=\mathsf{K}\right. \right\}$ for $\emptyset\neq \mathsf{K}\subseteq\{1,\ldots,\mathsf{d}\}$, and, using  \eqref{eq:201607121320}, define the mapping $\mathfrak{s}^{(\epsilon\vect{n})}$ by
$\mathfrak{s}^{(\epsilon\vect{n})}(\vectgamma) = \mathfrak{s}^{(\epsilon\vect{n})}_{\mathsf{k}}(\vectgamma)$, $\mathsf{k} = \max \mathsf{K}^{(\epsilon\vect{n})}[\vectgamma]$.  Employing the statements from the proof of \cite[Proposition 2.6 and Proposition 3.8]{DenckerErb2015a} we have
$\vect{\Gamma}^{(\epsilon\vect{n})}_{\vect{\kappa}}=\vect{\Gamma}^{(\epsilon\vect{n})}_{\vect{\kappa},0}\cup  \left\{\,\mathfrak{s}^{(\epsilon\vect{n})}(\vectgamma)\,\left|\,\vectgamma\in\vect{\Gamma}^{(\epsilon\vect{n})}_{\vect{\kappa},1}\right. \right\}$. This equality can be written as 
\[\displaystyle \vect{\Gamma}^{(\epsilon\vect{n})}_{\vect{\kappa}}=
\vect{\Gamma}^{(\epsilon\vect{n})}_{\vect{\kappa},0}\cup\bigcup_{\emptyset\neq \mathsf{K}\subseteq
\{1,\ldots,\mathsf{d}\}}\mathfrak{s}^{(\epsilon\vect{n})}_{\max \mathsf{K}}\left(\vect{\Gamma}^{(\epsilon\vect{n}),\mathsf{K}}_{\vect{\kappa},1}\right)\]
and can be considered as an  analog of \eqref{1602011534}.
Thus, substituting  in Subsection \ref{1602041256} the symbols $\overline{\vect{\Gamma}}^{(\vect{m})}$, $\vect{\Gamma}^{(\vect{m})}_0$, $\vect{\Gamma}^{(\vect{m})}_1$
by $\vect{\Gamma}^{(\epsilon\vect{n})}_{\vect{\kappa}}$, $\vect{\Gamma}^{(\epsilon\vect{n})}_{\vect{\kappa},0}$, $\vect{\Gamma}^{(\epsilon\vect{n})}_{\vect{\kappa},1}$,
respectively, the proof of Corollary~\ref{1502171206} follows by the same lines of argumentation as the proof of Theorem~\ref{1602041200}.
\end{proof}

We obtain the following estimates of the discrete Lebesgue constants.

\begin{theorem} \label{1503120230}  Let $\epsilon\in\{1,2\}$ and $\vect{\kappa}\in \mathbb{Z}^{\mathsf{d}}$ be fixed.
For all $\vect{n}\in\mathcal{N}_{\mathsf{d}}$, we have
\begin{equation}\label{201607121419}
\Lambda^{(\epsilon\vect{n})}_{\vect{\kappa}} \asymp \tprod_{\mathsf{i}=1}^{\mathsf{d}}\ln (n_{\mathsf{i}}+1).
\end{equation}
\end{theorem}

\begin{proof} We  introduce \[
\widetilde{K}^{(\epsilon\vect{n})}_{\vect{\kappa}}(\vect{x},\vect{x}')= \!\! \tsum_{\substack{\indexvectgamma \in \vect{\Gamma}^{(\epsilon\vect{n})}_{\vect{\kappa}}}}\!\!
2^{\mathfrak{e}(\indexvectgamma)-\mathfrak{f}^{(\epsilon\vect{n})}(\indexvectgamma)} T_{\indexvectgamma}(\vect{x})\;\!T_{\indexvectgamma}(\vect{x}'),\quad
K^{(\epsilon\vect{n})}_{\vect{\kappa}}(\vect{x},\vect{x}') = \!\! \tsum_{\substack{\indexvectgamma \in \vect{\Gamma}^{(\epsilon\vect{n})}_{\vect{\kappa}}}}\!\!2^{\mathfrak{e}(\indexvectgamma)} T_{\indexvectgamma}(\vect{x})\;\!T_{\indexvectgamma}(\vect{x}').\]
From \eqref{1601161944}, \eqref{1602171546}, and Proposition \ref{1602162324}, we get for all $\vect{n}\in\mathcal{N}_{\mathsf{d}}$ that
\begin{align} \label{1502180020}
\Lambda^{(\epsilon\vect{n})}_{\vect{\kappa}} &=
\max_{\vect{x} \in [-1,1]^{\mathsf{d}}}\tsum_{\vect{i}\in \I^{(\epsilon\vect{n})}_{\vect{\kappa}}}\mathfrak{w}^{(\epsilon\vect{n})}_{\vect{i}}
\left|\widetilde{K}^{(\epsilon\vect{n})}_{\vect{\kappa}}\left(\vect{x},\vect{z}^{(\epsilon\vect{n})}_{\vect{i}}\right)
- T_{\epsilon n_{\mathsf{d}}}(z^{(\epsilon n_{\mathsf{d}})}_{i_{\mathsf{d}}}) \;\!  T_{\epsilon n_{\mathsf{d}}}(x_{\mathsf{d}})\right| \\
\nonumber & \leq \max_{\vect{x} \in [-1,1]^{\mathsf{d}}}\tsum_{\vect{i}\in \I^{(\epsilon\vect{n})}_{\vect{\kappa}}}\mathfrak{w}^{(\epsilon\vect{n})}_{\vect{i}}
\left|\widetilde{K}^{(\epsilon\vect{n})}_{\vect{\kappa}}\left(\vect{x},\vect{z}^{(\epsilon\vect{n})}_{\vect{i}}\right)\right|+1 \lesssim \max\limits_{\vect{x} \in [-1,1]^{\mathsf{d}}}\|\widetilde{K}^{(\epsilon\vect{n})}_{\vect{\kappa}}(\vect{x},\,\cdot\,)\|_{w_{\mathsf{d}},1}+1.
\end{align}
Using the  well-known  relation $\tprod_{\mathsf{i}=1}^{\mathsf{r}}\cos(\vartheta_{\mathsf{i}})=
\dfrac1{2^{\mathsf{r}}}\tsum_{\vect{v}\in\{-1,1\}^{\mathsf{r}}}\cos(v_1\vartheta_1+\cdots+v_{\mathsf{r}}\vartheta_{\mathsf{r}})$, $\mathsf{r} \in \mathbb{N}$,
we get $\tprod_{\mathsf{i}=1}^{\mathsf{d}} \cos(\gamma_{\mathsf{i}}s_{\mathsf{i}})\cos(\gamma_{\mathsf{i}}t_{\mathsf{i}})
= \dfrac1{2^{2\mathsf{d}}}\tsum_{\vect{v},\vect{w}\in \{-1,1\}^{\mathsf{d}}}\cos\left(\tsum_{\mathsf{i} = 1 }^{\mathsf{d}}(v_{\mathsf{i}}\gamma_{\mathsf{i}}s_{\mathsf{i}}+w_{\mathsf{i}}\gamma_{\mathsf{i}}t_{\mathsf{i}})\right)$.

\medskip

\medskip Then, for all $\vect{x}=(\cos s_1,\ldots,\cos s_{\mathsf{d}})$, we get
\begin{align}
&\|K^{(\epsilon\vect{n})}_{\vect{\kappa}}(\vect{x},\,\cdot\,)\|_{w_{\mathsf{d}},1}=\dfrac1{(2\pi)^{\mathsf{d}}}\int_{[-\pi,\pi)^{\mathsf{d}}}\left|\tsum_{\indexvectgamma \in \vect{\Gamma}^{(\epsilon\vect{n})}_{\vect{\kappa}}} 2^{\mathfrak{e}(\indexvectgamma)}\tprod_{\mathsf{i}=1}^{\mathsf{d}} \cos(\gamma_{\mathsf{i}}s_{\mathsf{i}})\cos(\gamma_{\mathsf{i}}t_{\mathsf{i}})\right|\mathsf{d}\vect{t}\label{1602171145}\\
&=\dfrac1{(2\pi)^{\mathsf{d}}}\int_{[-\pi,\pi)^{\mathsf{d}}}\left|\dfrac1{2^{2\mathsf{d}}}\tsum_{\vect{w} \in \{-1,1\}^{\mathsf{d}}}\tsum_{\indexvectgamma \in \vect{\Gamma}^{(\epsilon\vect{n})}_{\vect{\kappa}}} 2^{\mathfrak{e}(\indexvectgamma)} \tsum_{\vect{v}\in \{-1,1\}^{\mathsf{d}}}\cos\left(\tsum_{\mathsf{i}=1}^{\mathsf{d}}(v_{\mathsf{i}}\gamma_{\mathsf{i}}s_{\mathsf{i}}+w_{\mathsf{i}}v_{\mathsf{i}}
\gamma_{\mathsf{i}}t_{\mathsf{i}})\right)\right|\mathsf{d}\vect{t}\nonumber\\
&\leq \dfrac1{2^{\mathsf{d}}} \tsum_{\vect{w}\in \{-1,1\}^{\mathsf{d}}}\dfrac1{(2\pi)^{\mathsf{d}}} \displaystyle \int_{[-\pi,\pi)^{\mathsf{d}}}\left|\tsum_{\indexvectgamma \in \vect{\Gamma}^{(\epsilon\vect{n})}_{\vect{\kappa}}} 2^{\mathfrak{e}(\indexvectgamma)-\mathsf{d}} \tsum_{\vect{v}\in \{-1,1\}^{\mathsf{d}}}\cos\left(\tsum_{\mathsf{i}=1}^{\mathsf{d}}v_{\mathsf{i}}(\gamma_{\mathsf{i}}s_{\mathsf{i}}+w_{\mathsf{i}}
\gamma_{\mathsf{i}}t_{\mathsf{i}})\right)\right|\mathsf{d}\vect{t}\nonumber\\
&= \dfrac1{(2\pi)^{\mathsf{d}}} \displaystyle \int_{[-\pi,\pi)^{\mathsf{d}}}\left|\tsum_{\indexvectgamma \in \vect{\Gamma}^{(\epsilon\vect{n})}_{\vect{\kappa}}} 2^{\mathfrak{e}(\indexvectgamma)-\mathsf{d}} \tsum_{\vect{v}\in \{-1,1\}^{\mathsf{d}}}\cos\left(\tsum_{\mathsf{i}=1}^{\mathsf{d}}v_{\mathsf{i}}\gamma_{\mathsf{i}}(s_{\mathsf{i}}+ t_{\mathsf{i}})\right)\right|\mathsf{d}\vect{t}
=\FL\left(\vect{\Gamma}^{(\epsilon\vect{n}),\ast}_{\vect{\kappa}}\right).\nonumber
\end{align}
Note that $\FL\left(\vect{\Gamma}^{(\epsilon\vect{n}),\ast}_{\vect{\kappa}}
\right) = \|K^{(\epsilon\vect{n})}_{\vect{\kappa}}(\vect{1},\,\cdot\,)\|_{w_{\mathsf{d}},1}$.
In the same way as in \eqref{1602171145}, we get
\begin{equation*}
\|\widetilde{K}^{(\epsilon\vect{n})}_{\vect{\kappa}}(\vect{x},\,\cdot\,)\|_{w_{\mathsf{d}},1} \leq
\|\widetilde{K}^{(\epsilon\vect{n})}_{\vect{\kappa}}(\vect{1},\,\cdot\,)\|_{w_{\mathsf{d}},1} \quad \text{for all}\ \vect{x} \in [-1,1]^{\mathsf{d}}.
\end{equation*}
Thus, by \eqref{1502180020}, we obtain for all $\vect{n}\in\mathcal{N}_{\mathsf{d}}$ the upper estimate
\begin{equation} \label{1602171146} \Lambda^{(\epsilon\vect{n})}_{\vect{\kappa}}\lesssim \|\widetilde{K}^{(\epsilon\vect{n})}_{\vect{\kappa}}(\vect{1},\,\cdot\,)\|_{w_{\mathsf{d}},1}+1.\end{equation}

For $\mathsf{K} \subseteq \{1, \ldots, \mathsf{d} \}$, we denote $\vect{\Gamma}^{(\epsilon\vect{n})}_{\vect{\kappa},\mathsf{K}}=
\left\{\,\vectgamma\in \vect{\Gamma}^{(\epsilon\vect{n})}_{\vect{\kappa}}\,\left|\,2\gamma_{\mathsf{i}}=\epsilon n_{\mathsf{i}}\
\Leftrightarrow \mathsf{i}\in\mathsf{K}\,\right. \right\}$.
Then, for all $\vect{x}=(\cos t_1,\ldots,\cos t_{\mathsf{d}})$, we have
\begin{equation} \label{1602171147}
\widetilde{K}^{(\epsilon\vect{n})}_{\vect{\kappa}}(\vect{1},\,\vect{x})=
K^{(\epsilon\vect{n})}_{\vect{\kappa}}(\vect{1},\,\vect{x}) -
\tsum_{\emptyset\neq \mathsf{K}\subseteq \{1,\ldots,\mathsf{d}\}}(1-2^{-\#\mathsf{K}+1})
\displaystyle\tsum_{\indexvectgamma \in \vect{\Gamma}^{(\epsilon\vect{n})}_{\vect{\kappa},\mathsf{K}}}2^{\mathfrak{e}(\indexvectgamma)}\cos(\gamma_{\mathsf{i}}t_{\mathsf{i}}).
\end{equation}
If $\mathsf{K}\neq \emptyset$ and $\vect{\Gamma}^{(\epsilon\vect{n})}_{\vect{\kappa},\mathsf{K}}\neq\emptyset$, then $\vect{\Gamma}^{(\epsilon\vect{n})}_{\vect{\kappa},\mathsf{K}}\subseteq \vect{\Gamma}^{(\epsilon\vect{n})}_{\vect{\kappa},\mathfrak{r}}$
for some $\mathfrak{r} \in \{0,1\}$. Therefore, the sets $\vect{\Gamma}^{(\epsilon\vect{n})}_{\vect{\kappa},\mathsf{K}}$ have a
cross product structure and we get
 \[\left|\tsum_{\indexvectgamma \in \vect{\Gamma}^{(\epsilon\vect{n})}_{\vect{\kappa},\mathsf{K}}}2^{\mathfrak{e}(\indexvectgamma)}
 \tprod_{\mathsf{i}=1}^{\mathsf{d}}\cos(\gamma_{\mathsf{i}}t_{\mathsf{i}})\right| \lesssim
 \tprod_{\mathsf{i}\in\{1,\ldots,\mathsf{d}\}\setminus \mathsf{K}}
 \left|\tsum_{ \gamma_{\mathsf{i}} = - \lceil \epsilon n_{\mathsf{i}}/2 \rceil+1}^{\lceil \epsilon n_{\mathsf{i}}/2 \rceil-1}
 \mathrm{e}^{ \mathbf{i} \gamma_{\mathsf{i}} t_{\mathsf{i}}} \right|.\]
Thus, for  $\mathsf{K}\neq \emptyset$ and all $\vect{n}\in\mathcal{N}_{\mathsf{d}}$, we have
\begin{equation} \label{201607121444}
\int_{[-\pi,\pi)^{\mathsf{d}}}\left|\tsum_{\indexvectgamma \in \vect{\Gamma}^{(\epsilon\vect{n})}_{\vect{\kappa},\mathsf{K}}}2^{\mathfrak{e}(\indexvectgamma)}
\tprod_{\mathsf{i}=1}^{\mathsf{d}}\cos(\gamma_{\mathsf{i}}t_{\mathsf{i}})\right|\mathsf{d}\vect{t}\lesssim \tprod_{\mathsf{i}\in\{1,\ldots,\mathsf{d}\}\setminus \mathsf{K}} \ln(n_{\mathsf{i}}+1).
\end{equation}
Now, combining \eqref{1602171145}, \eqref{1602171146}, \eqref{1602171147}, and \eqref{201607121444} gives
the first inequality in \eqref{1606180824}. Finally, Corollary \ref{1502171206} implies for  $\vect{n}\in\mathcal{N}_{\mathsf{d}}$ the estimate from above in \eqref{201607121419}.

\medskip

We turn to the lower bound in \eqref{201607121419}.
Let $\epsilon\in \{1,2\}$, $\vect{\kappa}\in\mathbb{Z}^{\mathsf{d}}$, and $\vect{n}\in\mathcal{N}_{\mathsf{d}}$. By \cite[Theorem 1]{Subbotin1983}, we have

\vspace{-4mm}

\begin{equation} \label{1606191403}  \Lambda^{(\epsilon\vect{n})}_{\vect{\kappa}} \geq \dfrac1{3^{\mathsf{d}}} \FL\left(\vect{\Gamma}^{(\epsilon\vect{n}),\ast}_{\vect{\kappa}}\right).
\end{equation}
Note that in \cite{Subbotin1983}
the $L^\infty$-$L^\infty$-operator norm of the partial Fourier sum operator with respect to the set $\vect{\Gamma}^{(\epsilon\vect{n})}_{\vect{\kappa}}$
is used to characterize the Lebesgue constant related to Fourier sums. This characterization is identical to the definition given in this article, see \cite{Liflyand2006} and the
references therein. The relation \eqref{1606191403} and Corollary \ref{1502171206} immediately imply that for  $\vect{n}\in\mathcal{N}_{\mathsf{d}}$ we have the estimate
from below in \eqref{201607121419}.
\end{proof}

\medskip

To estimate the approximation error $\|f - P^{(\epsilon\vect{n})}_{\!\vect{\kappa}} f \|_\infty$ for a continuous function $f \in C([-1,1]^{\mathsf{d}})$, let us consider the error of the best approximation given by
\[E^{(\epsilon\vect{n})}_{\!\vect{\kappa}}(f)=\min\limits_{P\in \Pi^{(\epsilon\vect{n})}_{\vect{\kappa}}}\|f-P\|_{\infty}.\]
Further, let $P^*\in \Pi^{(\epsilon\vect{n})}_{\!\vect{\kappa}}$ be such that  $E^{(\epsilon\vect{n})}_{\!\vect{\kappa}}(f)=\|f-P^*\|_\infty$.  By Theorem
\ref{1503120230}, we get
\begin{align} \notag \|f - P^{(\epsilon\vect{n})}_{\!\vect{\kappa}} f \|_\infty
&\leq  \| P^{(\epsilon\vect{n})}_{\!\vect{\kappa}} (P^*-f) \|_\infty +\|f - P^* \|_\infty\\
&\leq (\Lambda^{(\epsilon\vect{n})}_{\vect{\kappa}} +1)E^{(\epsilon\vect{n})}_{\!\vect{\kappa}}(f) \lesssim
\left( \tprod_{\mathsf{i}=1}^{\mathsf{d}}\ln (n_{\mathsf{i}}+1) \right) E^{(\epsilon\vect{n})}_{\!\vect{\kappa}}(f).
\label{eq:000001}
\end{align}

Now, using \eqref{eq:000001} and a multivariate version of Jackson's inequality (see \cite[Section 5.3.2]{Timan1960}) to estimate $E^{(\epsilon\vect{n})}_{\!\vect{\kappa}}(f)$, we obtain
the following result.

\begin{corollary} \label{cor-dinilipschitz}
Let $\epsilon\in\{1,2\}$ and $\vect{\kappa}\in \mathbb{Z}^{\mathsf{d}}$ be fixed.
Let also $\vect{s} \in \mathbb{N}_0^{\mathsf{d}}$ and \[\dfrac{ \partial^{\!\;s_{\mathsf{j}}} f}{{\partial x_{\mathsf{j}}}^{\!\!s_{\mathsf{j}}}} \in C([-1,1]^{\mathsf{d}}),
\quad \mathsf{j} \in\{1, \ldots, \mathsf{d}\}.\] Then, for $\vect{n}\in\mathcal{N}_{\mathsf{d}}$, we have
\[ \|f - P^{(\epsilon\vect{n})}_{\vect{\kappa}} f \|_\infty \lesssim \left(
\tprod_{\mathsf{i}=1}^{\mathsf{d}}\ln (n_{\mathsf{i}}+1)\right) \, \tsum_{\mathrm{j} = 1}^{\mathsf{d}}
\dfrac1{(n_{\mathsf{j}}+1)^{s_{\mathsf{j}}}} \, \omega
\left(\dfrac{ \partial^{\!\;s_{\mathsf{j}}} f}{{\partial x_{\mathsf{j}}}^{\!\!s_{\mathsf{j}}}} ; 0, \ldots, 0, \dfrac{1}{n_{\mathsf{j}}+1},0, \ldots, 0 \right), \]
where
\vspace{-1mm}
\[\omega(f; \vect{u}) = \sup_{\substack{\vect{x}, \vect{x}'\in [-1,1]^{\mathsf{d}}\\\forall\,
\mathsf{i} \;\!\in\{1, \ldots, \mathsf{d} \}:\,|x'{}_{\;\!\!\!\!\mathsf{i}}-x_{\mathsf{i}}| \leq u_{\mathsf{i}}}} |f(\vect{x}')-f(\vect{x})| \]

\noindent denotes the modulus of continuity of $f$ on $[-1,1]^{\mathsf{d}}$ (see  \cite[Section 6.3]{Timan1960}).
\end{corollary}

\begin{proof}
In view of \eqref{eq:000001}, we only need to give a proper estimate of the best approximation $E^{(\epsilon\vect{n})}_{\!\vect{\kappa}}(f)$.
Since $\vect{\Gamma}^{(\epsilon\vect{n})}_{\vect{\kappa},0} \subseteq \vect{\Gamma}^{(\epsilon\vect{n})}_{\vect{\kappa}}$, we have
$E^{(\epsilon\vect{n})}_{\!\vect{\kappa}}(f) \leq E^{(\epsilon\vect{n})}_{\!\vect{\kappa},0}(f)$, where $E^{(\epsilon\vect{n})}_{\!\vect{\kappa},0}(f)$
denotes the error of the best approximation in the space spanned by  $T_{\indexvectgamma}$, $\vectgamma \in \vect{\Gamma}^{(\epsilon\vect{n})}_{\vect{\kappa},0}$.
Since $\vect{\Gamma}^{(\epsilon\vect{n})}_{\vect{\kappa},0}$ has a tensor-product structure, we  obtain
\begin{align*}
E^{(\epsilon\vect{n})}_{\!\vect{\kappa}}(f)  \leq E^{(\epsilon\vect{n})}_{\!\vect{\kappa},0}(f) &\lesssim
\tsum_{\mathrm{j} = 1}^{\mathsf{d}}
\dfrac{2^{s_{\mathsf{j}}} \omega
\left(\frac{ \partial^{\;\!s_{\mathsf{j}}} f}{{\partial x_{\mathsf{j}}}^{\!\!s_{\mathsf{j}}}} ; 0, \ldots, 0, \frac{2}{\epsilon n_{\mathsf{j}}+1},0\ldots, 0 \right)}{(\epsilon n_{\mathsf{j}}+1)^{s_{\mathsf{j}}}} \\
&\lesssim
\tsum_{\mathrm{j} = 1}^{\mathsf{d}}
\dfrac{\omega
\left(\frac{ \partial^{\!\;s_{\mathsf{j}}} f}{{\partial x_{\mathsf{j}}}^{\!\!s_{\mathsf{j}}}} ; 0, \ldots, 0, \frac{1}{n_{\mathsf{j}}+1},0\ldots, 0 \right)}{(n_{\mathsf{j}}+1)^{s_{\mathsf{j}}}}
\end{align*}

\noindent by using  the estimates from \cite[Section 5.3.2]{Timan1960}.
\end{proof}

\medskip

Similar as stated in \cite[Theorem 4.1]{Mason1980} for the tensor-product case, we can also give a Dini-Lipschitz criterion for the uniform convergence of the error
$\|f - P^{(\epsilon\vect{n})}_{\!\vect{\kappa}} f \|_\infty$.

 If $f \in C([-1,1]^{\mathsf{d}})$ satisfies the condition
\[ \omega(f; \vect{u})  \tprod_{\mathsf{i}=1}^{\mathsf{d}} \ln u_{\mathsf{i}} \to 0 \quad \text{as}\quad \vect{u} \to \vect{0},\]
then the polynomials $P^{(\epsilon\vect{n})}_{\!\vect{\kappa}} f$ converge in the $L^{\infty}$-norm to $f$ as $\displaystyle \min_{\mathsf{i}=1,\ldots,\mathsf{d}}n_{\mathsf{i}}\to\infty$.


\end{document}